\documentclass{article}
\usepackage[utf8]{inputenc}
\usepackage[english]{babel}  
\usepackage[T1]{fontenc}    
\usepackage{authblk} 
\usepackage{mathtools}
\usepackage[breaklinks]{hyperref}
\usepackage{amsfonts}
\usepackage{amsmath}
\usepackage{amssymb}
\usepackage{amsthm}
\usepackage{esint}
\usepackage[usenames,dvipsnames,svgnames,table]{xcolor}
\usepackage{tocloft}
\usepackage{blindtext}
\usepackage[hang,flushmargin]{footmisc}
 \usepackage[page,toc,titletoc,title]{appendix}
\usepackage{enumitem}
\hypersetup{linktocpage}

\usepackage[margin=1.4in]{geometry}

\mathtoolsset{showonlyrefs}

\newcommand{\upa}{u}
\newcommand{\uf}{U}
\newcommand{\ui}{\mathsf{U}}
\newcommand{\vp}{v}
\newcommand{\vf}{V}
\newcommand{\vi}{\mathsf{V}}

\newcommand{\R}{\mathbb{R}}

  \newcommand{\di}{\diamond}
  
 \allowdisplaybreaks

\newtheorem{theorem}{Theorem}

\newtheorem{proposition}{Proposition}

\newtheorem{corollary}{Corollary}

\newtheorem{lemma}{Lemma}
 
\newtheorem{definition}{Definition}

\newtheorem{remark}{Remark}

  \newcommand\blfootnote[1]{%
  \begingroup
  \renewcommand\thefootnote{}\footnote{#1}%
  \addtocounter{footnote}{-1}%
  \endgroup
}

\title{The Initial Value Problem for Singular SPDEs via Rough Paths}

\author[1]{Claudia Raithel}
\author[2]{Jonas Sauer}

\affil[1]{ TU Dresden, Institute for Scientific Computing, Zellescher Weg 25, 01217 Dresden, Germany}
\affil[2]{ Friedrich-Schiller-Universit\"at Jena, Institut f\"ur Mathematik, Ernst-Abbe-Platz 2, 07737 Jena, Germany}
\begin{document}

\maketitle

\begin{abstract}
In this contribution we develop a solution theory for singular quasilinear stochastic partial differential equations subject to an initial condition. We obtain our solution theory as a perturbation of the rough path approach developed to handle the space-time periodic problem by Otto and Weber (2019). As in their work, we assume that the forcing is of class $C^{\alpha -2}$ for $\alpha \in (\frac{2}{3},1)$ and space-time periodic and, additionally, that the initial condition is of class $C^{\alpha}$ and periodic. We contribute to the analytic aspects of the theory. Indeed, we show that we can enforce the initial condition via correcting the previously obtained space-time periodic solution with an initial boundary layer which may be handled in a completely deterministic manner.
Uniqueness is obtained in the class of solutions which are corrected in this way by an initial boundary layer.
Moreover, stability of the solutions with respect to perturbations of the data is established.
\blfootnote{\textbf{Funding:} CR was partially supported by the Austrian Science Fund (FWF), grants P30000, W1245, and F65.}
\end{abstract}


\newpage

\tableofcontents

\section{Introduction}

In this paper we construct a stable solution operator for a parabolic quasilinear initial value problem in $1+1$ dimensions that is driven by a rough right-hand side. Let $\alpha \in (\frac{2}{3},1)$, we consider the initial value problem
\begin{align}
\label{IVP_intro}
\begin{split}
\partial_2 \uf - a(\uf) \partial_1^2 \uf + \uf & = f \quad \quad \quad \quad \quad \quad \, \textrm{in} \quad \mathbb{R}^2_+, \\
\uf & = \uf_{int} \quad \quad \quad \quad \quad  \textrm{on} \quad \partial \mathbb{R}^2_+,
\end{split}
\end{align}
for $f \in C^{\alpha -2}(\mathbb{R}^2)$, $\uf_{int} \in C^{\alpha}(\mathbb{R})$, and $a  \in C^{\alpha}(\mathbb{R})$ that is uniformly elliptic and bounded with at least three bounded derivatives. All of the data is also assumed to be periodic; in particular, $f$ is space-time periodic and $a$ and $U_{int}$ are periodic. We use the notation $\mathbb{R}^2_+ = \left\{ x \in \mathbb{R}^2 : x_2 >0 \right\}$ and $\partial \mathbb{R}_+^2 = \left\{ x \in \mathbb{R}^2_+ : x_2  = 0\right\}$, where $x \in \mathbb{R}^2$ is written in coordinates as $x = (x_1, x_2)$.

Due to the low regularity of the forcing on the right-hand side of \eqref{IVP_intro}, the equation \textit{a priori} makes no sense: the nonlinear term $a(U) \partial_1^2 U$ has no classical definition. In particular, a heuristic application of classical Schauder theory to \eqref{IVP_intro} suggests that the solution $U$ is of class $C^{\alpha}(\R^2_+)$ --meaning that $a(U)$ is also of class $C^{\alpha}(\R^2_+)$, whereas $\partial_1^2 \uf \in C^{\alpha -2}(\R^2)$. This is a problem: two distributions have a classical well-defined product only if the sum of their H\"{o}lder exponents is positive; here, however, we have that $ \alpha +(\alpha -2) = 2 \alpha -2 < 0$. To handle this situation it is standard to take $f$ to be random and define the nonlinear term of the equation in an ``offline way'' --via a probabilistic and subsequent deterministic step. In particular, one first defines the nonlinear term for certain ``special functions'' (solving, \textit{e.g.}, the frozen linear analogue of \eqref{IVP_intro}) probabilistically --possibly using a renormalization. Then, using that the solution $U$ of \eqref{IVP_intro} behaves similarly to the ``special functions'' on small scales, one transfers the definition to give meaning to $``a(U) \partial_1^2 U"$. 

The treatment of the initial value problem \eqref{IVP_intro} in this paper is meant to compliment the theory for the analogous space-time periodic problem in \cite{OW}. The rough path method developed in \cite{OW} was the first to handle quasilinear singular SPDEs: the theory of regularity structures (developed in \cite{H,BHZ,CH,BCCH}) and paracontrolled distributions \cite{GIP} had previously only handled semilinear equations. For comprehensive expositions of these methods, we point the interested reader to the reviews and books \cite{FH,H_notes_1,H_notes_2,BH,CW,G_review}. Contemporarily to \cite{OW}, in \cite{BDH}, Bailleul, Debussche, and Hofmanov\'{a} treated a generalized parabolic Anderson model with scalar diffusion coefficients in a paracontrolled framework --after a transformation they were able to use the semilinear theory of \cite{GIP}. Mainly inspired by \cite{OW}, Furlan and Gubinelli then gave a paracontrolled treatment of quasilinear evolution problems \cite{FG}. While \cite{OW,BDH,FG} dealt with mildly singular noise (\textit{i.e.}, $\alpha \in (\frac23,1)$) --just barely making the nonlinear terms impossible to define classically-- in \cite{GH_quasi} Gerencs\'er and Hairer gave a general theory for quasilinear singular SPDEs within the framework of regularity structures in the full subcritical regime. They, however, only obtained the renormalized equation for $ \alpha \in (\frac{1}{2},1)$ -- for singular SPDEs of the form \eqref{IVP_intro} (without the massive term) this was improved upon by Gerencs\'er in \cite{Ge}, who obtained the renormalized equation for $\alpha \in (\frac25,1)$ (thereby including the case of space-time white noise in 1+1 dimensions). Inspired by \cite{BDH}, the same regime is considered in \cite{BM}. More recently, Bailleul, Hoshino, and Kusuoka treated the quasilinear generalized (KPZ) equation --obtaining the renormalized equation in the full subcritical regime \cite{BHK}. The rough path method (by now called ``multi-indices'' or ``tree-free'' method) of Otto and Weber has also been generalized to the full subcrictical regime by the second author together with Otto, Smith, and Weber: in \cite{OSSW_1} for $\alpha \in( \frac12,1)$ and \cite{OSSW_2} for $\alpha \in (0,1)$. The analytic treatments in \cite{OSSW_1,OSSW_2} are backed by the stochastic estimates contained in \cite{LOTT} (see also \cite{LOT}), which are based on a spectral gap assumption and Malliavin calculus (see also the recent lecture notes \cite{OST}).


While the setting treated in this paper is covered by both the paracontrolled framework in \cite{FG} and the regularity structures approach in \cite{GH_quasi}, our aim is to show that the space-time periodic rough path approach of Otto and Weber can be generalized to the setting of the evolution problem \eqref{IVP_intro} in a purely deterministic way via a perturbative ansatz. In this sense, our results and methods are completely deterministic: Accordingly, our main results, Theorems \ref{linear_theorem} and \ref{theorem_nonlinear}, are written as to apply to deterministic $f$ --we import stochastic estimates from \cite{OW} (and \cite{LOTT}, see Remark \ref{drop_time_period}) only to illustrate that for certain random $f$ the assumptions contained within the theorems are verified almost-surely (see Proposition \ref{offline_products}).

\paragraph{\textbf{Our perturbative approach:}} To solve the quasilinear problem \eqref{IVP_intro}, we first treat the linearized problem, \textit{i.e.}
\begin{align}
\begin{split}
\label{IVP_intro_linear}
(\partial_2 - a \partial_1^2 +1) \uf & =  f \quad \quad \quad \quad \quad \quad \quad \, \textrm{ in} \quad \mathbb{R}^2_+, \\ 
\uf& = \uf_{int}  \quad \quad \quad \quad  \quad \quad \, \textrm{on} \quad \partial \mathbb{R}^2_+,
\end{split}
\end{align}
for $a \in C^{\alpha}(\mathbb{R}^2)$ periodic in space, and then perform a contraction mapping argument. This is in line with the original work \cite{OW}, whereas it has since been shown that one can work directly on the level of the nonlinear problem \cite{OSSW_2}. On the level of \eqref{IVP_intro_linear}, it is natural to take the ansatz $\uf =  \upa + \ui $ with $\upa \in C^{\alpha}(\mathbb{R}^2)$ and $\ui \in C^{\alpha}(\mathbb{R}^2_+)$ solving
\begin{align}
\label{periodic_problem_intro}
(\partial_2 - a \partial_1^2 +1) u & = f  \quad \quad \quad \quad \quad \quad \quad  \textrm{in} \quad \mathbb{R}^2
\end{align}
and
\begin{align}
\begin{split}
\label{ivp_linear_intro}
(\partial_2 - a\partial_1^2 +1) \ui & =  0  \quad \quad \quad \quad \quad \quad \quad \textrm{ in} \quad \mathbb{R}^2_+,\\ 
\ui& = \uf_{int} - u \quad \quad \quad \quad  \, \textrm{on} \quad \partial \mathbb{R}^2_+
\end{split}
\end{align}
respectively. The point is that the solution $\upa$ of \eqref{periodic_problem_intro} can be obtained using a minor variation of the arguments in \cite{OW} and $\ui$, which we call the  ``initial boundary layer'', can be obtained classically using bounds for the heat semigroup. We remark that the introduction of an ``initial boundary layer'' is reminiscent of the splitting used in the treatment of boundary-value problems within the framework of regularity structures in \cite{GH}. 


To motivate some concepts and ideas we give a quick overview of the rough path method of Otto and Weber (for $\alpha \in (\frac23,1)$):

\paragraph{\textbf{Rough path method of Otto and Weber (for $\alpha \in (\frac23,1)$):}} In \cite{OW}, Otto and Weber treat the space-time periodic problem:
\begin{align}
\label{OW_intro}
\partial_2 u - P(a(u)\partial_1^2 u - \sigma(u) f)  = 0\qquad \textrm{in} \quad \R^2,
\end{align}
where, on top of the assumptions that we have on our data, $\sigma$ is of class $C^{\alpha}$ and satisfies some additional conditions and $P$ is the projection onto mean-zero space-time periodic functions. In our setting of the initial value problem we lose periodicity of solutions in the time direction --making that massive term in \eqref{IVP_intro} essential for controlling the $L^{\infty}$-norm of solutions.

The rough path approach in \cite{OW} relies on the notion of \textit{modelledness}:

\begin{definition}[Modelledness]
\label{model}
Let $\alpha \in (\frac{1}{2}, 1)$ and $\Omega \subseteq \mathbb{R}^2$. For $I \in \mathbb{N}$ we have families of functions $(V_1(\cdot, a_0), ..., V_I(\cdot, a_0))$ indexed by $a_0 \in \mathbb{R}$ defined on $\Omega \times \mathbb{R}$. A function $U: \Omega \rightarrow \mathbb{R}$ is said to be modelled after $(V_1(\cdot, a_0), ..., V_I(\cdot, a_0))$  on $\Omega$ according to functions $(a_1,..., a_I)$ and $(\sigma_1,..., \sigma_I)$ in $C^{\alpha}(\Omega)$ if there exists a function $\nu$ such that 
\begin{align}
\label{defn_model_intro}
\begin{split}
& M :=\\
 & \sup_{x \neq y; x, y \in \Omega} \sum_{i=1}^I\frac{ |U(y) - U(x) - \sigma_i(x) (V_i(y, a_i(x)) - V_i(x, a_i(x))) - \nu(x) (y-x)_1| }{d^{2\alpha}(x,y)}
 \end{split}
\end{align}
is finite. As emphasized in the next section, here $d(\cdot, \cdot)$ represents the parabolic metric on $\R^2$ given by \eqref{distance_defn}. 

\end{definition}

\begin{remark}[Modelledness notational conventions] 
\label{remark_model}
We say that $U$ is ``trivially modelled'' on $\Omega$ if $U \in C^{2\alpha} (\Omega)$, since then we may take $\sigma_i = 0$ and $\nu = \partial_1 U$. When we do not explicitly specify the $\sigma_i$ in the modelling of a function, then $\sigma_i = 1$. Also, for brevity, instead of saying that a function ``is modelled after a family $\left\{ v(\cdot, a_0) \right\}_{a_0 \in [\lambda, 1]}$'', we simply say that it is ``modelled after $v$''. 
\end{remark}

While the interested reader may consult \cite{OW} for a complete accounting of the motivation for this definition, we remark that the concept of ``modelled after'' is essentially a higher dimensional version of ``controlled by'' in the work of Gubinelli \cite{G}. Correspondingly, the $\sigma_i$ in Definition \ref{model} correspond to the \textit{Gubinelli derivative}. ``Modelledness'' as defined above also finds root in the theory of regularity structures: It can be seen as a quasilinear adaptation of the notion of a ``modelled distribution'' \cite[Definition 3.1]{H}. We remark that the definition of ``modelledness'' in Definition \ref{model} is so simple because we are in the mildly singular setting of $\alpha \in (\frac{2}{3},1)$ --in comparison, in the full sub-critical regime considered in \cite{OSSW_1,OSSW_2} it is necessary to consider a notion based on multi-indices (in analogue to trees in regularity structures).

\begin{remark}[Modelling for additive vs. multiplicative noise] Since we only consider \eqref{IVP_intro} with additive noise, in our setting --in contrast to that in \cite{OW}-- the modelling of the solution $U$ will be with respect to $\sigma_i = 1$. We only use general $\sigma_i$ twice: in the application of \cite[Lemma 3.2]{OW} to obtain a modelling for $a(U)$ (in the fixed-point argument contained in the proof of Theorem \ref{theorem_nonlinear}) and in Step 6 of Proposition \ref{linear_forcing}. 
\end{remark}

As already indicated above, to define the nonlinear terms in \eqref{OW_intro}, Otto and Weber have a probabilistic and deterministic step:

\vspace{.4cm}

\noindent $\bullet$ \textit{Probabilistic Step:}  Letting $f$ be random and denoting by $v_{\textrm{OW}}(\cdot, a_0)$ the solution of \eqref{OW_intro} with coefficients frozen at $a_0$ --the subscript $``\textrm{OW}"$ indicating the lack of the massive term in \eqref{OW_intro}--, after a renormalization, they almost-surely obtain ``offline products'' $v_{\textrm{OW}}(\cdot, a_0) \diamond \partial_1^2 v_{\textrm{OW}}(\cdot, a_0^{\prime}) \in C^{\alpha -2}(\mathbb{R}^2)$ for $a_0, a_0^{\prime} \in [\lambda,1]$ with $\lambda>0$ such that 
\begin{align}
\label{commutator_condition}
\sup_{a_0, a_0^{\prime} \in [\lambda, 1]}   \sup_{0 \leq j,k \leq 2 }\sup_{T\leq1} (T^{\frac{1}{4}})^{2 - 2\alpha  } \| \partial^j_{a_0} \partial^k_{a^{\prime}_0} [\vp_{\textrm{OW}}(\cdot, a_0), (\cdot)_T] \diamond \partial_1^2 \vp_{\textrm{OW}}(\cdot, a_0^{\prime})  \|  \lesssim 1. \quad
\end{align} 
Here, as in the sequel, $(\cdot)_T$ denotes convolution with a certain kernel (specified in Section \ref{section:definitions}) at scale $T$. The commutator, in particular, is defined as
\begin{align}
\label{defn_commutator}
\begin{split}
&[\vp_{\rm{OW}}(\cdot, a_0), (\cdot)_T] \diamond \partial_1^2 \vp_{\rm{OW}}(\cdot, a_0^{\prime}) \\
 &:= \vp_{\rm{OW}}(\cdot, a_0) \partial_1^2 (\vp_{\rm{OW}}(\cdot, a_0^{\prime}))_T - (\vp_{\rm{OW}}(\cdot, a_0) \diamond \partial_1^2 \vp_{\rm{OW}}(\cdot, a_0^{\prime}))_T.
\end{split}
\end{align}
As discussed in \cite{OW}, the commutator condition \eqref{commutator_condition} is well-motivated by the previous literature on singular SDEs (see, \textit{e.g.}, \cite[Theorem]{G}). Indeed, \eqref{commutator_condition} should be thought of as a $C^{2\alpha -2}$-control of the commutator.

\vspace{.2cm}

\noindent $\bullet$ \textit{Deterministic Step:} Having access to a family $\left\{v_{\textrm{OW}}(\cdot, a_0) \diamond \partial_1^2 v_{\textrm{OW}}(\cdot, a_0^{\prime})\right\}_{a_0, a_0^{\prime} \in [\lambda,1]}$ satisfying \eqref{commutator_condition}, they then show that if $u, w \in C^{\alpha}(\R^2)$ are modelled after $v_{\textrm{OW}}$ (in the sense of Definition \ref{model} and Remark \ref{remark_model}), then it is possible to define $u \diamond \partial_1^2 w$ such that \eqref{commutator_condition} is preserved. In \cite{OW} this requires two lemmas, which, in analogue to Hairer's Reconstruction Theorem \cite[Theorem 3.10]{H}, are called ``reconstruction lemmas'' --their analogues here are Lemmas \ref{lemma:reconstruct_1} and \ref{lemma:reconstruct_2}.

\vspace{.4cm}

\noindent The solution space for \eqref{OW_intro} is then the space of space-time periodic functions in $C^{\alpha}(\mathbb{R}^2)$ that are modelled after $v_{\textrm{OW}}$. Taking the singular products described above as an input, Otto and Weber then construct a stable solution operator for \eqref{OW_intro} using a completely deterministic approach. 


\vspace{.5cm}

In this paper, injecting the framework of \cite{OW} with our perturbative ansatz, we search for solutions $\uf = \upa +\ui  \in C^{\alpha}(\R^2_+)$ of \eqref{IVP_intro} that are periodic in space and modelled after the family $\left\{ (\vp +\vi)(\cdot, a_0) \right\}_{a_0 \in [\lambda,1]}$, where $v(\cdot, a_0)$ denotes the space-time periodic solution of \eqref{periodic_problem_intro} with coefficients frozen at $a_0$ (see \eqref{periodic_mean_free}) and $\vi(\cdot, a_0)$ denotes the solution of \eqref{ivp_linear_intro} with frozen coefficients and initial condition $\uf_{int} - \vp(\cdot, a_0)$ (see \eqref{constant_coeff_ivp_body}). This means that we must define new ``offline products'' involving $\vi(\cdot,a_0)$ --However, as we will see, calling these new products ``offline'' is a bit misleading as all of the ingredients are actually classical. With respect to the definition of the nonlinear term in \eqref{ivp_linear_intro}, notice that the initial conditions of $\vi$ and $\ui$ do not match-up --this is a technical detail that we must handle when showing the modelling of $\ui$ after $\vi$.

 \subsection{Notation}

We remark that $\langle \cdot \rangle$ is used to denote taking an expectation --since this paper is purely deterministic, simply importing probabilistic results from \cite{OW} (or alternatively 
\cite{LOTT}), this notation is limited to Section \ref{OW_products}. When $f \in \mathcal{D}'(\R^2)$ is a regular distribution, we still use $\langle f , \varphi \rangle$ to denote the distribution applied to the test function --here $\langle \cdot, \cdot \rangle$ is, of course, the $L^2(\mathbb{R}^2)$ inner-product. Following, \textit{e.g.} \cite[Section 2.5, 2.]{Kl}, the convolution of $f$ with a Schwartz function $\psi$, is a smooth function defined as 
\begin{align}
\label{defn_convolution_w_dist}
 f \ast \psi (x) = f (\psi_T(x - \cdot)). 
\end{align}

When we say ``periodic'', the period will always be $1$ and is, therefore, not emphasized. For a distribution, periodicity of $f$ means that,  $\langle f, \varphi \rangle = \langle  f , \varphi(\cdot + p)\rangle$, where $p$ is the period. 

We use the notation $x_i := x \cdot e_i$ for a point $x \in \mathbb{R}^2$; in particular, $x = (x_1, x_2)$.  Furthermore, $\mathbb{R}^2_+ := \left\{ x \in \R^2 \, : \, x_2 >0 \right\}$ and  $\mathbb{R}^2_- := \left\{ x \in \R^2 \, : \, x_2 <0 \right\}$ --correspondingly, $\partial \mathbb{R}^2_+ = \left\{x \in \R^2 \, : \, x_2 = 0 \right\}$. Additionally, for $L>0$, we use $\mathbb{R}^2_L := \R \times (-\infty,-L]$. 

We write ``$\lesssim\,$'' to indicate ``$\leq C\,$'', where $C$ is a universal constant that usually may depend only on the ellipticity ratio $\lambda>0$. The notation $``\ll "$ means $``\leq c\,"$ for an arbitrarily small constant $c$. 

 Throughout this article, we will use the Einstein summation convention. We use $``\rightharpoonup"$ to denote weak convergence, the space is always clear given the context. For $r>0$, the parabolic ball of radius $r$ around $x$ is given by $B_{r}(x) := \left\{ y \in \R^2 \, | \, d(x,y) \leq r \right\}$, where $d(\cdot, \cdot)$ is defined in \eqref{distance_defn}.

In this paper functions/ distributions will either be defined on $\mathbb{R}^2$ or on $\mathbb{R}^2_+$. The domain is usually clear from the context and is, therefore, not mentioned. When the domain that a norm is taken over is slightly ambiguous we indicate it with a subscript. To given an example, we remark that $\| u \|_{\alpha; \R^2_+}$ is the $C^{\alpha}$-norm of $u$ on $\R^2_+$.

\section{Set-up and overview of our strategy}
\label{section:overview}

In the current section our goal is to formally state our results --this requires us to first introduce various notions. The build-up to our main results, which are contained in Section \ref{section:discussion} is regrettably slow --the reader familiar with \cite{OW} may skim through Sections \ref{section:definitions} - \ref{section:reconstruct} and mainly focus on Section \ref{section:discussion}.

As already remarked above, our basic strategy is to construct the solution $U$ of \eqref{IVP_intro} via a perturbation of the space-time periodic theory of \cite{OW}; in particular, we make the ansatz $U = u + \ui$, for $u$ and $\ui$ solving \eqref{periodic_problem_intro} and \eqref{ivp_linear_intro} respectively. In Section \ref{section:definitions} we discuss the expected modelling of $U$ --and also introduce definitions and notions that we will use throughout. In Section \ref{OW_products} we discuss the singular products $v(\cdot, a_0) \diamond \partial_1^2 v(\cdot, a_0')$ --in particular, the stochastic results of \cite{OW} which we import to our setting (see also remark \ref{drop_time_period}). In Section \ref{section:new_products} we introduce the new ``offline'' (actually classical) products --resulting from the introduction of the ``initial boundary layer''. In Section \ref{section:reconstruct}, we give the variants of the reconstruction lemmas from \cite{OW} that we use --the proofs for all these results are analogous to the arguments in \cite{OW}. Since there are minor differences, we provide proofs in abbreviated form in Section \ref{section:reconstruct_proof}. The core results and strategy of this paper are introduced in Section \ref{section:discussion} --the main novel contribution being Theorem \ref{linear_theorem}, which treats the linearized problem \eqref{IVP_intro_linear}. Within the treatment of \eqref{IVP_intro_linear} the interesting part is the ``initial boundary layer'' -- in Proposition \ref{ansatz_IVP} an ansatz for $\ui$ is analyzed and in Proposition \ref{linear_IVP} the ansatz is corrected to solve \eqref{ivp_linear_intro}.

\subsection{Definitions and tools}
\label{section:definitions}

\textit{Modelling and freezing of the non-linearity:} We have already introduced the concept of modelledness (Definition \ref{model}) and have explained that we expect the solution $\uf$ of \eqref{IVP_intro} to be modelled after $\vf(\cdot,a_0)$ solving 
\begin{align}
\begin{split}
\label{frozen_main}
(\partial_2 - a_0 \partial_1^2 + 1) \vf(\cdot,a_0) & = f  \hspace{3.2cm} \textrm{in} \quad \mathbb{R}^2_+, \\
\vf(\cdot,a_0)& = \uf_{int}  \hspace{2.75cm}   \textrm{on} \quad \partial \mathbb{R}^2_+,
\end{split}
\end{align}
where this function decomposes as 
\begin{align}
\label{decompose_modelling}
\vf(\cdot, a_0) = \vp(\cdot, a_0) + \vi(\cdot, a_0, \uf_{int} - \vp(a_0)).
\end{align}
Here, we use the following convention:

\begin{definition}[Parameterized constant coefficient solutions]
\label{constant_soln}
Let  $a_0 \in [\lambda, 1]$ for some $\lambda>0$ and $ \vi_{int}(\cdot, a_0) \in C^{\alpha}(\mathbb{R})$ be periodic in space. Furthermore, let $\vp(\cdot, a_0) \in C^{\alpha}(\mathbb{R}^2)$ denote the space-time periodic solution of 
\begin{align}
\label{periodic_mean_free}
\begin{split}
(\partial_2 - a_0 \partial_1^2 +1) \vp(\cdot, a_0)& = f \quad \quad \quad  \textrm{ in } \quad \mathbb{R}^2
\end{split}
\end{align}
and $\vi(\cdot, a_0, \vi_{int}(a_0)) \in C^{\alpha}(\mathbb{R}^2_+)$ denotes the solution of
\begin{align}
\begin{split}
\label{constant_coeff_ivp_body}
(\partial_2 - a_0 \partial_1^2 + 1 )\vi (\cdot, a_0,  \vi_{int}( a_0)) & = 0  \hspace{3cm}   \textrm{in} \quad \mathbb{R}^2_+,\\
 \vi(\cdot, a_0,  \vi_{int}( a_0)) & = \vi_{int}(\cdot, a_0) \hspace{1.7cm}  \textrm{on} \quad  \partial \mathbb{R}^2_+,
\end{split}
\end{align}
which is periodic in space. 

For two right-hand sides $f_i$ or two initial conditions $U_{int,i}$ with $i=0,1$, the corresponding solutions of \eqref{periodic_mean_free} and \eqref{constant_coeff_ivp_body} are denoted by $\vp_i(\cdot, a_0)$ and $\vi_i(\cdot, a_0)$ respectively. 
\end{definition}

\noindent In the sequel, for brevity, when the forcing $f \in C^{\alpha-2}(\mathbb{R}^2)$ and initial condition $U_{int} \in C^{\alpha}(\mathbb{R})$ are fixed, we use the notation
\begin{align}
\label{vi_short}
\vi(\cdot, a_0) :=\vi(\cdot,a_0, \uf_{int} - \vp(a_0)),
\end{align}
which allows us to rewrite \eqref{decompose_modelling} as 
\begin{align}
\begin{split}
\label{decompose_modelling_short}
\vf(\cdot, a_0) = (\vp + \vi)(\cdot, a_0).\\
\end{split}
\end{align}

\noindent \textit{Norms and Seminorms:} 
We are interested in regularity in terms of \textit{parabolic} H\"{o}lder spaces. In particular, when we write $C^{\alpha}(\mathbb{R}^2)$ or $C^{\alpha}(\mathbb{R}^2_+)$  for $\alpha \in (0,1)$ we are referring to the H\"{o}lder space that is defined in terms of the Carnot-Carath\'{e}odory metric induced by the parabolic operator $\partial_2-a_0\partial_1^2$ on $\mathbb{R}^2$, given by
\begin{align}
\label{distance_defn}
d(x,y) := |x_1 - y_1 | + |x_2 - y_2|^{\frac{1}{2}}
\end{align}
for $x,y \in \mathbb{R}^2$. Of course, $C^{\alpha}(\mathbb{R})$ refers to the standard H\"{o}lder space. 

We will use the typical convention that for $\beta \in (1,2)$ one defines
\begin{align}
\label{definition_beta_12}
[u]_{\beta} := [\partial_1 u ]_{\beta -1}
\end{align}
and analogously for $\beta \in (2,3)$ we have 
\begin{align}
\label{definition_beta_12.2}
[u]_{\beta} := [\partial_1^2 u]_{\beta -2} + [\partial_2 u]_{\beta -2}.
\end{align}

Throughout this paper we use the notation
\begin{align*}
\|u \| : = \sup_{x \in \R^2} | u(x)|.
\end{align*}
If we have a family of functions $\left\{u(\cdot, a_0, a_0^{\prime})\right\}_{a_0, a_0^{\prime} \in [\lambda, 1]}$, then we use the convention 
\begin{align*}
\|u\| : = \displaystyle\sup_{a_0, a_0^{\prime} \in [\lambda, 1]} \|u(\cdot, a_0, a_0^{\prime})\| .
\end{align*}

We also define a negative H\"older norm:

\begin{definition}[Negative H\"{o}lder norm]
\label{negative_norm}
Let $\alpha \in (0,1) \cup (1,2)$. We define the $C^{\alpha - 2}$-norm of a distribution $u$ as 
\begin{align}
\label{definition_neg_beta}
[u]_{\alpha - 2} := \inf_{(u^1, u^2, u^3)}\Big( [u^1]_{\alpha} +  [u^2]_{\alpha} + [u^3]_{\alpha} + \|u^3\|\Big),
\end{align}
where the infimum is taken over triplets of functions $(u^1, u^2, u^3)$ such that $ u = \partial_1^2 u^1 + \partial_2 u^2 + u^3$.
\end{definition}
\noindent  Notice that even though we choose to use a seminorm notation on the left-hand side of  \eqref{definition_neg_beta}, thanks to the $\| u^3\|$-term on the right-hand side, this is actually a norm.

At one point in our arguments it is necessary to use a local version of the $C^{\alpha}$-seminorm. Here is the definition:

\begin{definition}[Local H\"{o}lder seminorm]
\label{local_seminorm} Let $\alpha \in (0,1)$. We define the local $C^{\alpha}$-seminorm of a function $u$ as 
\begin{align}
\label{definition_beta_01_loc}
[u]_{\alpha}^{loc} : = \sup_{d(x,y) \leq 1} \frac{|u(x) - u(y)|}{d^{\alpha}(x,y)}.
\end{align}
\end{definition}

For a family of functions $\left\{u(\cdot, a_0, a_0^{\prime})\right\}_{a_0, a_0^{\prime} \in [\lambda, 1]}$, we use the notation 
\begin{align}
\label{parameter_derivs_1}
\| u  \|_{j,k}  : = \sup_{m \leq j} \sup_{n \leq k} \| \partial_{a_0}^m \partial_{a_0^{\prime}}^n u \| \quad \textrm{and} \quad \|u  \|_j  : =  \sup_{m \leq j} \| \partial_{a_0}^m  u \|.
\end{align}
We use the same convention for the $C^{\alpha}$-norm and seminorm; \textit{i.e.}, we write
\begin{align}
\begin{split}
\label{parameter_derivs_2}
\| u   \|_{\alpha, j,k} & : = \sup_{m \leq j} \sup_{n \leq k}  \| \partial_{a_0}^m \partial_{a_0^{\prime}}^n u \|_{\alpha},\\
\|u  \|_{\alpha,j} & : =  \sup_{m \leq j}  \| \partial_{a_0}^m  u \|_{\alpha},\\
[ u  ]_{\alpha, j,k} & : =  \sup_{m \leq j} \sup_{n \leq k}  [ \partial_{a_0}^m \partial_{a_0^{\prime}}^n u]_{\alpha}, \\
\textrm{and} \quad [u  ]_{\alpha,j} & : =  \sup_{m \leq j} [ \partial_{a_0}^m u ]_{\alpha}.
\end{split}
\end{align}
Similar notation can be introduced for the local H\"{o}lder seminorm from Definition \ref{local_seminorm}.\\

\noindent \textit{Convolution kernel:} Throughout many of our arguments we rely on regularization via convolution with a specific kernel. The convolution kernel that we use is the same as that in \cite{OW} and is most easily defined (up to a normalizing multiplicative constant $C \in \R$) in terms of its Fourier transform: 
\begin{align}
\label{definition_psi_T}
\hat{\psi}_T (k):=  C \exp(-T(k_1^4 + k_2^2)).
\end{align}
This definition implies that $\psi_T$ is a positive Schwartz function. This kernel is chosen because it is the semigroup associated to the operator $\partial_1^4 - \partial_2^2$, which is positive and has the same relative scaling as $\partial_2 - \partial_1^2$. Usually, we will use the convention 
\begin{align*}
 (\cdot)_T = \cdot \ast \psi_T;
 \end{align*} 
occasionally, we even drop the parentheses and simply use the subscript $T$.

We now list and prove some useful properties of $\psi_T$. We will use the change of coordinates 
\begin{align}
\label{standard_rescaling}
\hat{x} = (\hat{x}_1, \hat{x}_2) = \Big( \frac{x_1}{T^{\frac{1}{4}}}, \frac{x_2}{T^{\frac{1}{2}}} \Big).
\end{align}
Fix $T>0$, here is the list of properties of $\psi_T$: 

\begin{itemize}
\item Using \eqref{definition_psi_T} and \eqref{standard_rescaling} we find that 
\begin{align}
\label{scaling}
\psi_T(x_1, x_2) =(T^{\frac{1}{4}})^{-3} \psi_1\left( \hat{x}_1, \hat{x}_2 \right).
\end{align}
Therefore, assuming that $C = \| \psi_1 \|_{L^1}^{-1}$ in \eqref{definition_psi_T}, we obtain $\| \psi_T  \|_{L^1} = \| \psi_1\|_{L^1} = 1$.

\item \textit{(Bound on the moments of $\psi_T$)} For any $i,j \geq 0$, $\alpha \geq 0$, and $y \in \mathbb{R}^2$ we have that 
\begin{align}
\label{moment_bound}
 \int_{\mathbb{R}^2} d^{\alpha}(x,y)  |\partial^i_1 \partial_2^j \psi_T(x - y)|\,\textrm{d}x \lesssim  (T^{\frac{1}{4}})^{\alpha - i - 2j }.
\end{align}

To see this we may assume that $y = 0$, after which rescaling with \eqref{standard_rescaling} gives
\begin{align*}
 \int_{\mathbb{R}^2} d^{\alpha}(x,0)   |\partial^i_1 \partial_2^j \psi_T(x )|\,\textrm{d}x & = (T^{\frac{1}{4}})^{\alpha - i - 2j } \int_{\mathbb{R}^2}   d^{\alpha}(\hat{x},0)   |\partial^i_1 \partial_2^j \psi_1(\hat{x})|\,\textrm{d} \hat{x}.
\end{align*}
The fact that $\psi_1$ is a Schwartz function yields \eqref{moment_bound}.

\item \textit{(Semigroup property of $\psi_T$)} For a distribution $u$ and two scales $t, T>0$, we have that
$ (u \ast \psi_t ) \ast \psi_T = u \ast (\psi_{t} \ast \psi_T)$ and, by \eqref{definition_psi_T}, that $\psi_t \ast \psi_T = \psi_{t+ T}$. Combining these two yields
\begin{align}
\label{semigroup}
(u_{t})_T = u_{t+T}.
\end{align}

\item For any $i,j \geq 0$ such that $i+j \geq 1$ and $u \in C^{\alpha}(\mathbb{R}^2)$, by \eqref{moment_bound} we have that 
\begin{align}
\label{alpha_kernel_bound}
\begin{split}
\left| \int_{\mathbb{R}^2} \partial_1^i \partial_2^j u(y) \psi_T(x - y) \, \textrm{d}y \right|& \leq \left| \int_{\mathbb{R}^2} (u(y) - u(x)) \partial_1^i \partial_2^j  \psi_T(x - y) \, \textrm{d}y \right| \\
& \leq [u]_{\alpha} \int_{\mathbb{R}^2} d^{\alpha}(x,y) |\partial_1^i \partial_2^j  \psi_T(x - y)| \, \textrm{d}y\\
& \lesssim [u]_{\alpha} (T^{\frac{1}{4}})^{\alpha -i - 2j}.
\end{split}
\end{align}
Notice that we have again used that $\psi_T$ is a Schwartz function (in the first line).

\item  \textit{(Monotonicity of the $L^{\infty}$-norm in terms of the convolution scale)} For a distribution $u$ and $T\geq t >0$ it holds that
\begin{align}
\label{monotonicity}
\begin{split}
\| u \ast \psi_{T}\| \lesssim \| u \ast \psi_t \| \| \psi_{T-t}\|_{L^1} =\| u \ast \psi_t \|,
\end{split}
\end{align}
where we have used \eqref{semigroup}, Young's inequality for convolutions, and \eqref{scaling}.
\end{itemize}

While Definition \ref{negative_norm} gives the standard notion that we use for the $C^{\alpha-2}$-norm, we often also need an equivalent formulation, which is developed in Lemma \ref{equivnorm} below and relies on convolution with $\psi_T$ at scales $T \leq 1$. 
\begin{lemma}[Equivalent $C^{\alpha -2}$-norm]
\label{equivnorm}
Let  $\alpha \in (0,1)$, then a distribution $f$ on $\mathbb{R}^2$ satisfies  
\begin{align}
\label{equiv_negative_1}
 [ f ]_{\alpha-2} \sim \displaystyle\sup_{T \leq 1} (T^{\frac{1}{4}})^{2 - \alpha} \| f_T \|.
\end{align}
\end{lemma}
\noindent This lemma is an analogue of \cite[Lemma A.1]{OW}, the proof of which is not immediately adaptable to our setting due to the loss of periodicity in the $x_2$-direction. Instead, one can adapt an argument from a work by Ignat and Otto \cite{IO} --the proof of Lemma \ref{equivnorm} is contained in Section \ref{section:equivnorm}.

We use the notation 
\begin{align}
\label{convolution_notation_neg}
\|f\|_{-\beta}:= \displaystyle\sup_{T \leq 1} (T^{\frac{1}{4}})^{\beta} \| f_T \| 
\end{align}
for $\beta>0$.

The alternate formulation of the $C^{\alpha-2}$-norm is useful when working with the singular products: The family $\left\{ v(\cdot, a_0) \diamond \partial_1^2 v(\cdot, a_0^{\prime}) \right\}_{a_0, a_0^{\prime} \in [\lambda,1 ]}$ should satisfy an estimate of the form \eqref{commutator_condition}. We abbreviate \eqref{commutator_condition} as
\begin{align}
\label{commutator_condition_short}
  \sup_{T \leq 1} (T^{\frac14})^{2-2\alpha}  \| [\vp, (\cdot)_T ]  \diamond \partial_1^2 \vp\|_{2,2}  =:  \| [\vp, (\cdot) ]  \diamond \partial_1^2 \vp\|_{2 \alpha -2, 2,2} \lesssim 1,
\end{align}
where we additionally use the convention \eqref{parameter_derivs_1}. Following Lemma \ref{equivnorm}, we interpret \eqref{commutator_condition_short} as a $C^{2\alpha-2}$-control for the commutator.  As expounded on in Section \ref{OW_products}, the family of offline products satisfying \eqref{commutator_condition_short} exists almost-surely for a certain class of random $f$. The equivalence in Lemma \ref{equivnorm} is, in particular, used to prove the reconstruction lemmas --which are given in Section \ref{section:reconstruct} and are proven in Section \ref{section:reconstruct_proof}. Here, one passes to the limit (up to subsequences) in sequences of distributions that are uniformly controlled in the sense of the right-hand side of \eqref{equiv_negative_1}. For this, one must rely on compactness in the H\"{o}lder space on the left-hand side of \eqref{equiv_negative_1} --which follows from Definition \ref{negative_norm}.\\

\noindent \textit{A hierarchy of norms:} There is a natural hierarchy of norms. We measure:
\begin{itemize}
\item functions (\textit{e.g.}, the solution $\uf$ or the initial condition $\uf_{int}$ of \eqref{IVP_intro}) in $C^{\alpha}$,
\item distributions (\textit{e.g.}, the forcing $f$ or the singular product $a(\uf) \diamond \partial_1^2 \uf$ in \eqref{IVP_intro}) in $C^{\alpha -2}$,
\item and commutators (\textit{e.g.}, $[a(\uf), (\cdot)] \diamond \partial_1^2 \uf$) in $C^{ 2 \alpha -2}$.\\
\end{itemize}

\noindent \textit{Extensions to negative times:} In order for our arguments to make sense, it will often be necessary to extend various functions defined only for positive times to negative times. We will do this in two ways:
\begin{definition}[Extensions to negative times]
\label{extensions}
For a function $f$ defined on $\mathbb{R}^2_+ \cup \partial \mathbb{R}^2_+$,  we use $\tilde{f}$ to denote the even-reflection across the axis $\left\{ x_2 = 0 \right\}$ and $f^E$ to denote the trivial extension by $0$. So, in particular, we have that 
\begin{align*}
\tilde{f}(x) : = f(\tilde{x}),
\end{align*}
where we use the convention $ \tilde{x} = (x_1, |x_2|)$ for $x = (x_1, x_2)$, and 
\begin{align*}
f^E(x) : = \begin{cases}
f(x) \quad \quad \quad  &\textrm{if} \quad x \in \mathbb{R}^2_+\\
0 \quad \quad \quad & \textrm{if} \quad x \in \mathbb{R}^2_-.
\end{cases}
\end{align*}
\end{definition}
\noindent Notice that, for $\alpha \in (0,1)$, if $f \in C^{\alpha}(\R^2_+)$, then $\tilde{f} \in C^{\alpha}(\R^2)$. 

\begin{remark}[Usage of the extensions] We use the even-reflection defined above, \textit{e.g.}, in the construction of the new offline products --see Corollaries \ref{new_family_1} and \ref{new_family_2} in Section \ref{section:new_products} below. The reason is we want an analogue of \eqref{commutator_condition} (equivalently \eqref{commutator_condition_short}) to be satisfied --and a function being convolved with $\psi_T$ should be defined on the whole-space. 
The trivial extension is used in the construction of $w$ in the proof of Proposition \ref{linear_IVP} --due to the application of Lemma \ref{KrylovSafonov}, which again has a whole-space character. We use the trivial extension in the construction of $w$ to ensure that $w$ has $0$ initial condition. 
\end{remark}

\subsection{Usage of the periodic offline products }

\label{OW_products}
The point of this section is to demonstrate the applicability of our main theorems --Theorems \ref{linear_theorem} and \ref{theorem_nonlinear}, which are stated in Section \ref{section:discussion}.

As in \cite{OW}, we now assume that $f$ is random --sampled from a class of stationary, space-time periodic, and centered Gaussian distributions. The regularity conditions on the stationary $f$ are expressed in terms the discrete Fourier transform $\hat C$ of its covariance function.
Namely, we postulate that there are $\lambda_1,\lambda_2\in \mathbb{R}$ and $\alpha'\in (\frac14,1)$ such that
\begin{align}\label{A1}
 \hat{C}(k)  &\leq \frac{1}{(1+|k_1|)^{\lambda_1} (\sqrt{1+|k_2|})^{\lambda_2}},  \qquad k =(k_1,k_2) \in (2\pi \mathbb{Z})^2,\\
\notag
\lambda_1 +  \lambda_2 &= -1 + 2 \alpha'  \qquad  \lambda_1, \frac{\lambda_2}{2} < 1.
\end{align}
We refer to \cite[Section 3]{OW} for a discussion of admissible $f$, but note that this class includes, \textit{e.g.\,}, the case that 
$f$ is ``white'' in the time-like variable $x_2$ and has covariance operator $(1+ |\partial_{1}|)^{-\lambda_1}$ for $\lambda_1 >\frac13$ in the $x_1$ variable.

For such $f$, the construction of the $\left\{v(\cdot, a_0)  \diamond \partial_1^2 v(\cdot, a_0^{\prime})\right\}_{a_0, a_0^{\prime} \in [\lambda,1]}$, where we use Definition \ref{constant_soln}, necessitates a renormalization procedure. More precisely, let $\psi'$ be an arbitrary positive, $L^1$-normalized Schwartz function and set $\psi^{\prime}_\varepsilon(x_1,x_2)=\frac{1}{\varepsilon^\frac{3}{4}}\psi_1^{\prime}(\frac{x_1}{\varepsilon^\frac{1}{4}},\frac{x_2}{\varepsilon^\frac{1}{2}})$. Then, for $f_\varepsilon = f \ast \psi'_\varepsilon$ and $a_0 \in [\lambda, 1]$, we let $v_\varepsilon(\cdot, a_0)$ solve
$(\partial_2 -a_0\partial_1^2 + 1)v_\varepsilon(\cdot, a_0)=f_\varepsilon$ and construct $v(\cdot, a_0)  \diamond \partial_1^2 v(\cdot, a_0^{\prime})$  as
\begin{align}
\begin{split}\label{renormalized_products}
v( \cdot, a_0) \diamond \partial_1^2 v(\cdot, a_0') 
:=& 
\lim_{\varepsilon \to 0} 
	\big( 	
		v_\varepsilon( \cdot, a_0) \partial_1^2 v_\varepsilon(\cdot, a_0')  - 
		\big\langle 
			v_\varepsilon( \cdot, a_0) \partial_1^2 v_\varepsilon(\cdot, a_0') 
		\big\rangle 
	\big),
\end{split}
\end{align} 
the existence of this limit being part of the assertion of the proposition below (recall that the notation $\langle \cdot \rangle$ denotes taking the expectation).
In general, the expectation $\langle v_\varepsilon( \cdot, a_0) \partial_1^2 v_\varepsilon(\cdot, a_0')\rangle$ diverges as $\varepsilon \rightarrow 0$, but we mention that no renormalization procedure is needed if $f$ is ``white'' in $x_1$ and ``trace-class'' in $x_2$.

\medskip

The results of \cite{OW}, adapted to our setting, can be summarized as:
\begin{proposition}
\label{offline_products}
Let $\alpha'\in(\frac14,1)$ and let $f$ be a centered, space-time periodic, stationary Gaussian random distribution satisfying the regularity assumption \eqref{A1}. Let $f_\varepsilon = f \ast \psi'_\varepsilon$ be as described above. We use the notation from Definition \ref{constant_soln}. Furthermore, suppose that $p<\infty$, $n,m\ge 0$ and $\alpha < \alpha'$.

Then the renormalized product \eqref{renormalized_products} converges almost surely and in every stochastic $L^p$ space uniformly in $a_0$, $a_0'$ with respect to the $C^{\alpha-2}$-norm. Furthermore, we find that 
\begin{align}
\label{FinalStoch1AAA}
 \Big\langle \Big(  \| f \|_{\alpha-2} \Big)^p \Big\rangle^{\frac{1}{p}} \lesssim 1 \quad \textrm{and} \quad 
    \Big\langle \Big( \| [ v, ( \cdot)] \diamond \partial_1^2 v   \|_{2\alpha-2,n,m}\Big)^p \Big\rangle^{\frac{1}{p}} \lesssim 1,
\end{align}
where the universal constants depend only on $\lambda_1$, $\lambda_2$, $p$, $n$, $m$, $\alpha$, the ellipticity contrast $\lambda$ and the choice of the regularizing kernel $\psi'$.
\end{proposition}
Observe that, in contrast to \cite{OW}, we do not impose the condition $\hat{C}(0) = 0$ corresponding to a mean-free condition on $f$. On the same token, the functions $v$ and $v_\varepsilon$ are solutions with respect to the operator $\partial_2 - a_0 \partial_1^2 + 1$ instead of $\partial_2 - a_0 \partial_1^2$, which would incur the additional subscript ``$\textrm{OW}$'' as indicated in the introduction. An inspection of the argument for \cite[Lemma 4.1]{OW} yields that these two modifications actually compensate for each other --this is because the massive term gives an additional factor of $e^{-x_2}$ in the Green's function. We find, in particular, that the relevant results carry over to our setting --and do not give an explicit proof of Proposition \ref{offline_products} here.

\begin{remark}[Dropping the time periodicity of $f$]
\label{drop_time_period} 
The only place where $f$ being periodic in time is essential is in Proposition \ref{offline_products} --although it is convenient also in some of our deterministic arguments. In particular, it is used in the current proof of Lemma \ref{small_v} --where, however, it is nonessential (one could drop the assumption of time periodicity and use an argument similar as that for Lemma \ref{equivnorm}). We remark that since the original appearance of this manuscript the contents of Proposition \ref{offline_products} (and much more) have been proven for a rather general class of noises satisfying a spectral gap assumption --see \cite{LOT}. Since, via minor modifications of our arguments, we very much expect Theorems \ref{linear_theorem} and \ref{theorem_nonlinear} to hold under the dropping of periodicity in the time direction, they will also be applicable almost-surely for random $f$ as considered in \cite{LOT}. 
\end{remark}


\subsection{New ``offline'' products} 
\label{section:new_products}

For $a_0, a_0^{\prime} \in [\lambda, 1]$, we construct two new types of generalized products:
\begin{align} 
\vp(\cdot, a_0) \diamond \partial_1^2\tilde{\vi}(\cdot, a_0^{\prime})  \qquad \textrm{and} \qquad \tilde{\vi}(\cdot, a_0) \diamond \partial_1^2 \vp(\cdot, a_0^{\prime}),
\end{align}
where $\tilde{\vi}(\cdot, a_0)$ is the even-reflection of the function defined in \eqref{vi_short} and $\vp(\cdot, a_0)$ solves \eqref{periodic_mean_free}. Each of these families should satisfy a $C^{2\alpha -2}$- commutator estimate similar to \eqref{commutator_condition_short}.  These new ``offline'' products along with those from Section \ref{OW_products} and the two reconstruction lemmas (see Section \ref{section:reconstruct}) make it possible to give meaning to the nonlinear term in \eqref{IVP_intro}.

As already mentioned, the new ``offline'' products are, in fact, constructed classically --not requiring any probabilistic tools, but instead relying on the following estimates for the constant coefficient solutions from Definition \ref{constant_soln}. We start by compiling bounds for $\vi(\cdot, a_0, \vi_{int}( a_0))$. Here, we rely on the heat kernel formulation of $\vi(\cdot, a_0, \vi_{int}(a_0))$, \textit{i.\,e.\,}, using the notation 
\begin{align}
\label{heat_kernel}
G(x_1, x_2, a_0) :=  \frac{1}{(4 \pi a_0 x_2)^{\frac{1}{2}}}  e^{\frac{-x_1^2}{4 x_2 a_0} - x_2},
\end{align}
for any $x \in \mathbb{R}^2_+$, we write
\begin{align}
\label{heat_kernel_representation}
\vi(x, a_0,\vi_{int}(a_0)) = \int_{\mathbb{R}} \vi_{int}(y, a_0) G( x_1 - y, x_2, a_0 )\textrm{d}y.
\end{align}
We, in particular, obtain the following estimates:

\begin{lemma}[Bounds for the heat semigroup]
\label{semigroup_bounds}
Let $\alpha \in (0,1)$, $a_0 \in [\lambda, 1]$ for $\lambda>0$, and $\vi(\cdot, a_0, \vi_{int}(a_0))$ solve \eqref{constant_coeff_ivp_body}. Then the following observations hold:
\begin{itemize}[leftmargin=.2in]
\item[i)] For $1 \leq k \leq 2$ and $ 0 \leq j$,  $\partial_{a_0}^j \partial_1^k  \vi (\cdot, a_0, \vi_{int}(a_0))$ satisfies
\begin{align}
\label{heat_kernel_bound_1}
\begin{split}
 |\partial_{a_0}^j \partial_1^k  \vi (x, a_0, \vi_{int}(a_0))| \lesssim [ \vi_{int}  ]_{\alpha, j} x_2^{\frac{\alpha -k}{2}} 
\end{split}
\end{align}
 for $x \in \mathbb{R}^2_+$. In particular, if  $\partial_{a_0}^m \vi_{int} (\cdot, a_0) \in C^{\alpha}(\mathbb{R})$ for $m \leq j$, then $\partial_{a_0}^j \partial_1^k  \vi (\cdot, a_0, \vi_{int}(a_0))$ is a well-defined distribution. 
 
\quad If the initial condition $\vi_{int}$ does not depend on $a_0$, then the relation \eqref{heat_kernel_bound_1} also holds in the case that $k=0$ and $j>0$.
 
\item[ii)] For $j \geq 0$ and $x \in \mathbb{R}^2_+$, we have the $L^{\infty}$-estimate 
\begin{align}
\label{heat_kernel_bound_2}
 \| \vi (\cdot, x_2, a_0, \vi_{int}(a_0)) \|_{j}  \lesssim  \| \vi_{int} \|_j e^{-x_2}.
\end{align}

\item[iii)] For $0 \leq j \leq 3$, we have the relation 
\begin{align}
\label{heat_kernel_bound_3.5}
[\vi (\cdot, a_0,  \vi_{int}(a_0))]_{\alpha, j} \lesssim \| \vi_{int} \|_{\alpha, j}.
\end{align}

\item[iv)] For $0 \leq j \leq 1$ and $x,y \in \mathbb{R}^2_+$, we have that 
\begin{align}
\label{heat_kernel_bound_4}
  | \partial_{a_0}^j \vi(x, a_0, \vi_{int}(a_0) ) - \partial_{a_0}^j\vi(y, a_0, \vi_{int}(a_0) )|
\lesssim  \|\vi_{int} \|_{\alpha,j}(  x_2^{-\frac{\alpha}{2}}  +   y_2^{-\frac{\alpha}{2}})  d^{2\alpha} (x,y).
 \end{align}
 
 \item[v)] If $\vi(\cdot, a_0, \vi_{int}(a_0))$ solves \eqref{constant_coeff_ivp_body} without the massive term, then the estimates \eqref{heat_kernel_bound_1}, \eqref{heat_kernel_bound_3.5}, and \eqref{heat_kernel_bound_4} still hold. The estimate \eqref{heat_kernel_bound_2} still holds in a modified form; in particular, there is no factor of $e^{-x_2}$ on the right-hand side. 
 
\end{itemize}

\end{lemma}
\noindent These estimates are all elementary and surely they already exist somewhere --for completeness we have included a proof of Lemma \ref{semigroup_bounds} in Appendix \ref{semigroup_bounds_appendix}.

On the level of the space-time periodic constant coefficient solutions $\vp(\cdot, a_0)$, we often use the following estimate:

\begin{lemma}
\label{small_v} 
Let $a_0 \in [\lambda, 1]$ for $\lambda>0$ and $\vp(\cdot, a_0) \in C^{\alpha}(\mathbb{R}^2)$ solve \eqref{periodic_mean_free}. Then the bound 
\begin{align}
\label{result_appendix_lemma}
\|\vp\|_{\alpha, 2} \lesssim [f]_{\alpha  -2}
\end{align}
holds.
\end{lemma}
\noindent This lemma is essentially a corollary of the classical Schauder estimate and Definition \ref{negative_norm} --for completeness we have included a proof in Appendix \ref{semigroup_bounds_appendix}.

To construct the first type of new reference product we use the following lemma:

\begin{lemma}
\label{lemma:new_reference_products_1} Let $\alpha \in (0,1)$. Assume that $F \in C^{\alpha}(\mathbb{R}^2)$ and for $G$, a function defined on $\mathbb{R}^2$, there exists a constant $C(G) \in \mathbb{R}$ satisfying  
\begin{align}
\label{assumption_G_1}
|\partial_1^2 G(x)| \lesssim C(G)( |x_2|^{\frac{\alpha-2}{2}} + |x_2|^{\frac{2\alpha-2}{2}}),
\end{align}
for any $x \in \mathbb{R}^2$. Then $F  \partial_1^2 G$  is a well-defined regular distribution on $\mathbb{R}^2$ and 
\begin{align}
\label{assumption_reconstruction_2}
\|  [F , (\cdot) ]  \partial_1^2 G  \|_{2\alpha -2}
  \lesssim  C(G) [F]_{\alpha}.
\end{align}
\end{lemma}

\noindent Notice that in \eqref{assumption_reconstruction_2} we have used the notational convention introduced in \eqref{commutator_condition_short}. We use this lemma in conjunction with Lemmas \ref{semigroup_bounds} and \ref{small_v} to obtain the first type of new ``offline'' (actually classical) product: 

\begin{corollary}
\label{new_family_1}
Let $\alpha \in (0,1)$ and $a_0 , a_0^{\prime} \in [\lambda, 1]$ for $\lambda>0$. Furthermore, $\vp(\cdot, a_0) \in C^{\alpha}(\mathbb{R}^2)$ solves \eqref{periodic_mean_free} and $\vi(\cdot, a_0)$ is defined in \eqref{vi_short}; we use the notation from Definition \ref{extensions}. We then obtain
\begin{itemize}[leftmargin=.2in]
\item[i)] For any $F \in C^{\alpha}(\mathbb{R}^2)$, the products $F \partial_1^2 \tilde{\vi}(\cdot, a_0)$ are well-defined as distributions and this family satisfies 
\begin{align}
\label{assumption_reconstruction_2_new_2}
 \|  [F, (\cdot) ] \partial_1^2\tilde{\vi} \|_{2\alpha -2, 2}
  \lesssim  (\|U_{int}\|_{\alpha} + [f]_{\alpha -2}) [F]_{\alpha}.
\end{align}

\item[ii)] For $0 \leq j, k \leq 2$, the products $\partial_{a_0}^j \vp(\cdot, a_0) \partial_1^2 \partial_{a^{\prime}_0}^k\tilde{\vi}(\cdot, a^{\prime}_0)$ are well-defined as distributions and this family satisfies 
\begin{align}
\label{assumption_reconstruction_2_new_2.5}
 \|  [ \vp, (\cdot) ]  \partial_1^2\tilde{\vi}\|_{2\alpha -2, 2,2}
  \lesssim  (\|U_{int}\|_{\alpha} + [f]_{\alpha -2})[f]_{\alpha-2}.
\end{align}
\end{itemize}
\end{corollary}
\noindent Since these new ``offline'' products are actually completely standard classical products, we do not use the $``\diamond"$ notation. The proofs of Lemma \ref{lemma:new_reference_products_1} and Corollary \ref{new_family_1} are contained in Section \ref{products_proof_1}.

To construct the second type of new reference products we use the following lemma:

\begin{lemma} 
\label{lemma:new_reference_products_2} 
Let $\alpha \in (0,1)$. For $k = 1,2$, let $G\in C^{\alpha}(\mathbb{R}^2)$ satisfy the pointwise estimates
\begin{align}
\label{assumption_G_2}
|\partial_1^k G(x)| \lesssim C(G) |x_2|^{\frac{\alpha-k}{2}}
\end{align}
for some $C(G) \in \mathbb{R}$ and for any point $x \in \mathbb{R}^2$. Then, for $ F \in C^{\alpha}(\mathbb{R}^2)$, there exists a $C^{\alpha -2}$-distribution $G \diamond \partial_1^2 F$ satisfying 
 \begin{align}
\label{reconstruction_assumption_2_2}
 \| [ G , (\cdot)] \diamond \partial_1^2 F \|_{2\alpha -2}  \lesssim (  C(G)  + [G]_{\alpha}) [ F ]_{\alpha} .
\end{align}

The operation $\diamond$ is bilinear.


\end{lemma}

\noindent In this lemma we have used the notation \eqref{commutator_condition_short} and the direct analogue of the definition of the commutator given in \eqref{defn_commutator}. To apply Lemma \ref{lemma:new_reference_products_2}  we again use Lemmas \ref{semigroup_bounds} and \ref{small_v}:

\begin{corollary}
\label{new_family_2}
Let $\alpha \in (0,1)$,  $a_0 , a_0^{\prime} \in [\lambda, 1]$ for $\lambda>0$, and $i,j=0,1$. We use the notation from Definitions \ref{constant_soln} and \ref{extensions}. Assume that each $f_i \in C^{\alpha -2}(\mathbb{R}^2)$ is periodic and they satisfy the condition \hyperlink{A}{\textbf{(A)}}:

\begin{itemize}
\item[\hypertarget{A}{\textbf{(A)}}] For each pair $(f_i, f_j)$ there exists $\left\{ \vp_i(\cdot, a_0)\di \partial_1^2 \vp_j(\cdot, a_0^{\prime}) \right\}_{a_0, a_0^{\prime}}$, a family of $C^{\alpha -2}$-distributions, satisfying 
\begin{align}
\| [\vp_i, (\cdot)]  \di \partial_1^2 \vp_j \|_{2\alpha -2, 2,2} & \lesssim [f_i]_{\alpha -2} [f_j]_{\alpha -2} \label{assumption_offline_products_appendix_1}\\
\| [\vp_1, (\cdot)]  \di \partial_1^2 \vp_j - [\vp_0, (\cdot)]  \di \partial_1^2 \vp_j\|_{2\alpha -2, 1,1} & \lesssim [f_j]_{\alpha -2}  [f_1 - f_0]_{\alpha -2},\label{assumption_offline_products_appendix_2}\\
\textrm{and }  \| [\vp_i , (\cdot)]  \di \partial_1^2 \vp_1  - [\vp_i, (\cdot)]  \di \partial_1^2 \vp_0\|_{2\alpha -2, 1,1}  & \lesssim [f_i]_{\alpha -2} [f_1 - f_0]_{\alpha -2}.\label{assumption_offline_products_appendix_3}
\end{align}
\end{itemize}

\noindent Also, assume that each $\uf_{int,i} \in C^{\alpha}(\mathbb{R})$ is periodic.\\

Under these assumptions, for every $i,j = 0,1$, the following hold:
\begin{itemize}[leftmargin=.2in]
\item[i)] There exists a family of distributions $\{ \tilde{\vi}_i(\cdot, a_0)\diamond \partial_1^2 \vp_j(\cdot, a^{\prime}_0) \}_{a_0, a_0^{\prime}}$ such that 
\begin{align}
\label{assumption_reconstruction_2_new_1}
\|  [\tilde{\vi}_i , (\cdot) ] \diamond \partial_1^2 \vp_j\|_{2\alpha -2, 2,2}
  \lesssim  ( [f_i]_{\alpha-2} + [U_{int,i}]_{\alpha})  [f_j]_{\alpha-2} .
\end{align}

\item[ii)] Defining the family of distributions  
\begin{align}
\label{reference_now_in_lemma}
\begin{split}
  (\tilde{\vi}_i+ \vp_i) (\cdot, a_0) \diamond \partial_1^2 \vp_j(\cdot, a_0^{\prime})
  : =  \tilde{ \vi}_i (\cdot, a_0) \diamond \partial_1^2 \vp_j(\cdot, a_0^{\prime}) + \vp_i(\cdot, a_0) \di \partial_1^2 \vp_j(\cdot, a_0^{\prime}), \quad \,
\end{split}
\end{align}
we find that
\begin{align}
\| [\tilde{\vi}_i + \vp_i, (\cdot) ] \diamond \partial_1^2 \vp_j \|_{2\alpha -2, 2,2} \lesssim ([U_{int,i}]_{\alpha} + [f_i]_{\alpha -2}) [f_j]_{\alpha -2} . \label{assumption_reconstruction_2_new}
\end{align}

\item[iii)]The distributions constructed in part $ii)$ satisfy
\begin{align}
\begin{split}
\label{reference_cross_lemma_new}
& \| [\tilde{\vi}_0 + \vp_0, (\cdot)] \diamond  \partial_1^2 \vp_j - [\tilde{\vi}_1 +\vp_1, ( \cdot) ] \diamond  \partial_1^2 \vp_j\|_{2\alpha -2, 1,1}\\
 &
\lesssim ([U_{int,1} - U_{int,0}]_{\alpha} + [f_1 - f_0]_{\alpha -2})[f_j]_{\alpha -2} \, \, \, \, \, \,
\end{split}
\end{align}
and
\begin{align}
\label{reference_cross_lemma_2_new}
\begin{split}
& \| [\tilde{\vi}_i + \vp_i, ( \cdot ) ] \diamond  \partial_1^2 \vp_0 -  [\tilde{\vi}_i  + \vp_i,  (\cdot) ] \diamond  \partial_1^2 \vp_1 \|_{2\alpha -2, 1,1}\\
& \lesssim ([U_{int,i}]_{\alpha} + [f_i]_{\alpha -2})[f_1 - f_0]_{\alpha -2}.
 \end{split}
 \end{align}
 
\item[iv)] Letting
\begin{align}
\label{full_product_definition}
\begin{split}
  & (\tilde{\vi}_i+ \vp_i)(\cdot, a_0) \diamond \partial_1^2 (\tilde{\vi}_j + \vp_j) (\cdot, a_0^{\prime})\\
  & : =   (\tilde{\vi}_i+ \vp_i)(\cdot, a_0) \diamond \partial_1^2 \vp_j (\cdot, a_0^{\prime})  +  (\tilde{\vi}_i+ \vp_i)(\cdot, a_0) \partial_1^2\tilde{\vi}_j(\cdot, a_0^{\prime}),
\end{split}
\end{align}
where we use the distributions defined in $ii)$ and Corollary \ref{new_family_1}, we obtain
\begin{align}
\label{commutator_full}
\|  [ (\tilde{\vi}_i + \vp_i), (\cdot)] \diamond \partial_1^2 (\tilde{\vi}_j + \vp_j) \|_{2\alpha -2, 2,2}
  \lesssim  ([U_{int,1}]_{\alpha} + [f_1]_{\alpha -2})([U_{int,0}]_{\alpha} + [f_0]_{\alpha -2}).
\end{align}
\end{itemize}
\end{corollary}

\noindent The construction of the second new type of ``offline'' product is not as straightforward as the first type, but still entirely classical --it proceeds via the classical Leibniz' rule. Lemma \ref{lemma:new_reference_products_2} and Corollary \ref{new_family_2} are proven in Section \ref{products_proof_2}.
 
\subsection{Reconstruction lemmas}
\label{section:reconstruct}

In this section we introduce the two ``Reconstruction Lemmas'' --all proofs are given in Section \ref{section:reconstruct_proof}.  The first of the reconstruction lemmas gives a map
\begin{align}
\label{map_reconstruct_1}
\{ (\vp +\tilde{\vi})(\cdot, a_0) \diamond \partial_1^2 (\vp +\tilde{\vi})(\cdot, a_0^{\prime})  \} \mapsto \{ \uf \di \partial_1^2   (\vp +\tilde{\vi})(\cdot, a_0^{\prime}) \}
\end{align}
whenever $\uf$ is modelled after $\tilde{\vi} + \vp$. The intended application requires part $i)$ of the following result, which shows that the modelling of $\uf$ is preserved under smooth enough pointwise nonlinear transformations.\\

\newpage

\noindent \textbf{Lemma 3.2 of \cite{OW}.} 
\begin{itemize}[leftmargin=.2in]
\item[i)] \textit{Let $\uf \in C^{\alpha}(\mathbb{R}^2)$ be modelled after $V$ according to $a$ and $\sigma$, both of class $C^{\alpha} (\mathbb{R}^2)$ with modelling constant $M$; and the function $b$ be twice differentiable. Then, $b(\uf)$ is modelled after $V$ according to $a$ and $\mu = b^{\prime}(\uf) \sigma$ with modelling constant $\tilde{M}$ and $\|b(u) \|_{\alpha}$ satisfying
\begin{align}
\tilde{M} & \lesssim \| b^{\prime}\| M + \|b^{\prime \prime} \| [\uf]_{\alpha}^2 \quad \text{and }\quad  \|b(u) \|_{\alpha} \lesssim \| b^{\prime}\| [\uf]_{\alpha} + \|b\|.   \label{new_modelling}
\end{align}}
\item[ii)] \textit{For $i = 0,1$, let $\uf_i$ be modelled after $V_i(\cdot, a_0)$ according to $a_i$ and $\sigma_i$ with modelling constant $M_i$ as in part $i)$. Assume, furthermore, that $\uf_1 - \uf_0$ is modelled after $\left(V_1,V_0  \right)$ according to $(a_1, a_0)$ and $(\sigma_1, - \sigma_0)$ with modelling constant $\delta M$; and that $b$ is three times differentiable. Then, $b(\uf_1) - b(\uf_0)$ is modelled after $\left(V_1,V_0 \right)$ according to $(a_1, a_0)$ and $ (\mu_1, - \mu_0 ) = (b^{\prime}(\uf_1) \sigma_1,  -b^{\prime}(\uf_0) \sigma_0)$ with modelling constant $\delta \tilde{M}$ and $\| b(\uf_1) - b(\uf_0)\|_{\alpha}$ satisfying
\begin{align}
 \label{new_modelling_2}
 \begin{split}
\delta \tilde{M}   \lesssim &  \| \uf_1 - \uf_0\|_{\alpha} \Big( \| b^{\prime \prime } \|  \displaystyle\max_{i=0,1} [\uf_i]_{\alpha}+  \frac{1}{2}  \| b^{\prime \prime \prime} \| \displaystyle\max_{i=0,1} [\uf_i]_{\alpha}^2 + \| b^{\prime \prime} \| \displaystyle\max_{i=0,1} M_i \Big) + \| b^{\prime}\| \delta M 
\end{split}
\end{align}
and
\begin{align}
\label{Calpha_norm_new}
\| b(\uf_1) - b(\uf_0)\|_{\alpha} \lesssim \| \uf_1 - \uf_0\|_{\alpha}  \Big(\| b^{\prime}\|  + \| b^{\prime \prime} \| \displaystyle\max_{i=0,1} [\uf_i]_{\alpha} \Big).
\end{align}}
\end{itemize}
We omit the proof of this lemma --it amounts to an application of Taylor's formula (see, \textit{e.g.},  \cite[Proposition 6]{G}).

Next comes the statement of the first reconstruction lemma --to avoid confusion, let us emphasize that we use Einstein's summation convention:

\begin{lemma}[Modified Lemma 3.3 of \cite{OW}] 
\label{lemma:reconstruct_1}
Let $\alpha \in \left(\frac{2}{3}, 1\right)$ and all functions and distributions be $x_1$-periodic. Let $h$ be a distribution and $\left\{w(\cdot, x) \right\}_x$ a family of functions and $\left\{w(\cdot, x) \di h \right\}_x$ a family of distributions, both indexed by $x \in \mathbb{R}^2$, satisfying
\begin{align}
[ w(\cdot, x) ]_{\alpha} & \leq N,  \label{Lemma2_0} \\
[ w(\cdot, x) - w(\cdot, x^{\prime})]_{\alpha} & \leq N d^{\alpha}(x, x^{\prime}),  \label{Lemma2_1} \\
 \| h \|_{\alpha-2} & \leq N_0,  \label{Lemma2_2}\\
\| [ w(\cdot, x), (\cdot) ] \diamond h  \|_{2\alpha -2} & \leq N N_0, \label{Lemma2_2.5}\\
 \textrm{ and } \| [ w(\cdot, x), (\cdot) ] \di h - [w(\cdot, x^{\prime}), (\cdot) ] \di h \|_{2\alpha -2} & \leq N N_0 d^{\alpha}(x, x^{\prime}) \label{Lemma2_3} 
\end{align}
for any points $x, x^{\prime} \in \mathbb{R}^2$ and some constants $N, N_0 \in \mathbb{R}$.

Assume that for $\uf \in C^{\alpha}(\mathbb{R}^2)$ there is a function $\nu$ and $M \in \mathbb{R}$ such that 
\begin{align}
\label{Lemma2_3.5} 
|\uf(y) - \uf(x) - (w(y,x) - w(x,x)) - \nu(x)(y-x)_1| \leq M d^{2\alpha}(x,y)
\end{align}
for any points $x,y \in \mathbb{R}^2$. Then, letting $E_{diag}$ denote evaluation of a function of $(x,y)$ at $(x,x)$, there exists a unique distribution $U \di h \in C^{\alpha -2}(\mathbb{R}^2)$ satisfying
\begin{align}
 \label{Lemma2_4}
\lim_{T \rightarrow 0} \| \left[ \uf, (\cdot)_T \right] \di h - E_{diag} \left[w, (\cdot)_T \right] \di h - \nu \left[x_1, (\cdot)_T \right] h \| = 0.
\end{align}
The distribution $\uf \diamond h$ satisfies 
\begin{align}
\|[\uf, (\cdot)] \diamond h \|_{2\alpha -2} \lesssim (M + N)N_0. \label{Dec_2019_1}
\end{align}
\end{lemma}
\noindent The difference between Lemma \ref{lemma:reconstruct_1} and \cite[Lemma 3.3]{OW} is that we  have lost periodicity in the $x_2$-direction. 

With Lemma \ref{lemma:reconstruct_1} in-hand, we can define the mapping \eqref{map_reconstruct_1} and show that this map satisfies some continuity properties:

\begin{corollary}[Modified Corollary 3.4 of \cite{OW}]  
\label{post_process_reconstruct_1}
Let $\alpha \in (\frac{2}{3},1)$ and all functions and distributions be $x_1$-periodic. We adopt the assumptions and notation of Corollary \ref{new_family_2}. 

For $i, j = 0,1$, we find that the following observations hold:

\begin{itemize}[leftmargin=.2in]
\item[i)] Let $\uf \in C^{\alpha}(\mathbb{R}^2)$  be modelled after $\vp_i+\tilde{\vi}_i$ according to $a_i$ and $\sigma_i$ on $\mathbb{R}^2$ with modelling constant $M$ and, furthermore, assume that $\|a_i\|_{\alpha}, \|\sigma_i\|_{\alpha}  \leq 1$. Then for every $a_0 \in [\lambda, 1]$ there exists a unique $\uf \di \partial_1^2 \vp_j(\cdot, a_0) \in C^{\alpha -2}(\mathbb{R}^2)$ such that 
\begin{align}
\label{inhomo_product_1}
\begin{split}
&\displaystyle\lim_{T \rightarrow 0} \|[\uf, (\cdot)_T  ] \di \partial_1^2 \vp_j(\cdot, a_0) \\
& \hspace{1cm}- \sigma_i E_i [(\vp_i + \tilde{\vi}_i)(\cdot, a_0^{\prime}), (\cdot)_T ] \di \partial_1^2 \vp_j(\cdot, a_0)
  - \nu _i [x_1, (\cdot)_T] \partial_1^2 \vp_j(\cdot, a_0)  \| = 0 ,
\end{split}
\end{align}
where $E_i$ is the evaluation of a function depending on $(x, a^{\prime}_0, a_0)$ at $(x, a_i(x), a_0)$. The distributions $\uf \diamond \partial_1^2  \vp_j(\cdot, a_0)$ satisfy
\begin{align}
\label{conclusion_lemma_2_1}
\begin{split}
\| [\uf , (\cdot)  ]  \di  \partial_1^2  \vp_j \|_{2\alpha -2,2}
  \lesssim   ([U_{int,i}]_{\alpha} + [f_i]_{\alpha -2} + M) [f_j]_{\alpha -2}
 \end{split}
\end{align}
and
\begin{align}
\label{Lemma_3_3_1_1}
\begin{split}
  \|  [\uf, (\cdot)] \diamond \partial_1^2 \vp_1 
-  [\uf, (\cdot)] \diamond \partial_1^2  \vp_0  \|_{2\alpha -2,1}
  \lesssim   ([U_{int,i}]_{\alpha} + [f_i]_{\alpha -2} + M) [f_1 - f_0]_{\alpha -2}.
 \end{split}
\end{align}

\item[ii)] Let $\uf_i \in C^{\alpha}(\mathbb{R}^2)$ be modelled after $\vp_i+\tilde{\vi}_i$ according to $a_i$ and $\sigma_i$ as in part $i)$. Furthermore, assume that $\uf_1 - \uf_0$ is modelled after $(  \vp_1+ \tilde{\vi}_1 , \vp_0+\tilde{\vi}_0)$ according to $(a_1, a_0)$ and $(\sigma_1, - \sigma_0)$ with modelling constant $\delta M \in \mathbb{R}$. For the $ \uf_i \diamond \partial_1^2 \vp_j (\cdot, a_0)$ from part $i)$, we have that 
\begin{align}
\label{conclusion_lemma_2_ii}
\begin{split}
& \| [ \uf_1 , (\cdot)_T ]  \di  \partial_1^2 v_j- [ \uf_0 , (\cdot)_T  ]  \di  \partial_1^2 v_j \|_{2\alpha -2,1}\\
&  \lesssim [f_j]_{\alpha -2} \Big(\delta M  + \displaystyle\max_{i=0,1} ([U_{int,i}]_{\alpha} + [f_i]_{\alpha -2}) ( \| a_1 - a_0\|_{\alpha} + \| \sigma_1 - \sigma_0\|_{\alpha} ) \\
& \hspace{5cm} + [f_1 - f_0]_{\alpha -2} + [U_{int,1} - U_{int,0}]_{\alpha}\Big)   .
 \end{split}
\end{align}
\end{itemize}
\end{corollary}
The proof of Corollary \ref{post_process_reconstruct_1} is essentially the same as that for \cite[Corollary 3.4]{OW}, but relies on modelling information in terms of $\tilde{\vi} + v$ as opposed to $v_{\textrm{OW}}$. This, however, does not change the character of the calculations. 

We then move on to the second reconstruction lemma: Assuming that for $F \in C^{\alpha}(\mathbb{R}^2)$ there is a family of distributions $\{ F \diamond \partial_1^2 (v + \tilde{\vi})(\cdot, a_0)\}_{a_0 \in [\lambda, 1]}$, satisfying a $C^{2\alpha -2}$-commutator condition, this lemma gives a map 
\begin{align}
\{F \diamond \partial_1^2 (v + \tilde{\vi})(\cdot, a_0) \} \mapsto F \diamond \partial_1^2 U,
\end{align}
whenever $\uf$ is modelled after $v + \tilde{\vi}$. Here is the statement --again using Einstein's summation convention:

\begin{lemma}[Modified Lemma 3.5 of \cite{OW}]
\label{lemma:reconstruct_2}
Let $\alpha \in (\frac{2}{3},1)$, $I \in \mathbb{N}$, $\lambda >0$, and all functions and distributions be $x_1$-periodic. Assume that for $F \in C^{\alpha}(\mathbb{R}^2)$ and $( V_1(\cdot, a_0), ..., V_I(\cdot, a_0) )$, families of $C^{\alpha}$-functions indexed by $a_0 \in[\lambda,1]$, there exist $( F \diamond \partial_1^2 V_1(\cdot, a_0), ....,F \diamond \partial_1^2 V_I(\cdot, a_0)  )$, families of $C^{\alpha -2}$-distributions indexed by $a_0 \in[\lambda,1]$, such that the bounds
\begin{align}
 [ V_i ]_{\alpha,1} \leq N_i \quad \textrm{and} \quad  \| [F, (\cdot) ] \di \partial_1^2 V_i \|_{2\alpha-2,1} \leq N N_i \label{Lemma4_2}
\end{align}
hold for some constants $N,N_i \in \mathbb{R}$. Then, for a function $\uf \in C^{\alpha} (\mathbb{R}^2)$ that is modelled after $(V_1, ..., V_I)$ according to the $C^{\alpha}$-functions $a$ and $(\sigma_1,..., \sigma_I)$, there exists a unique distribution $F \diamond \partial_1^2 U \in C^{\alpha -2}(\mathbb{R}^2)$ such that
\begin{align}
\label{Lemma4_3}
\displaystyle\lim_{T \rightarrow 0}  \| [ F, (\cdot)_T ] \diamond \partial_1^2 \uf - \sigma_i E[ F, (\cdot)_T ] \diamond \partial_1^2 V_i\| =0,
\end{align}
where $E$ denotes the evaluation of a function of $(x, a_0)$ at $(x, a(x))$. Under the further assumption that $\|a\|_{\alpha} \leq 1$, we obtain the bound 
\begin{align}
\label{Lemma4_4} 
\| [ F, (\cdot) ] \diamond \partial_1^2 \uf \|_{2\alpha -2} \lesssim [ F  ]_{\alpha} M  + \|\sigma_i \|_{\alpha}  N N_i .
\end{align}
\end{lemma} 
\noindent The difference to \cite[Lemma 3.5]{OW} is again the loss of periodicity in the $x_2$-direction. 

\subsection{Discussion and statement of our results } 
\label{section:discussion}

In this section we state our main results and further expound upon the details of our perturbative ansatz.


As already emphasized in the introduction, our main strategy is to first treat a linearized version of \eqref{IVP_intro} and, on this level, enforce the right-hand side and initial condition separately --by introducing an ``initial boundary layer''. This strategy is summarized in the following theorem, which should be seen as the main result of this paper:

\begin{theorem}[Analysis of the linear problem]
\label{linear_theorem}
Let $\alpha \in \left(\frac{2}{3},1\right)$ and $\lambda>0$.

\begin{itemize}[leftmargin=.2in]

\item[i)] (Construction of Solution Operator)  Assume that we have:

\begin{itemize}[leftmargin=.5in]
\item[\hypertarget{B1}{\textbf{(B1)}}] a space-time periodic distribution $f$ and $N_0 \in \mathbb{R}$ such that 
 \begin{align}
 \label{assumption_on_forcing}
  \| f \|_{\alpha -2} \leq  N_0,
 \end{align}
 
\item[\hypertarget{B2}{\textbf{(B2)}}] $a \in C^{\alpha}(\mathbb{R}^2)$ that is periodic in the $x_1$-direction and satisfies $a \in \left[\lambda, 1 \right]$ and $[a]_{\alpha} \ll 1$,\\

\item[\hypertarget{B3}{\textbf{(B3)}}] a periodic function $\uf_{int} \in C^{\alpha}(\mathbb{R})$ and $N_0^{int} \in \mathbb{R}$ such that
\begin{align}
\label{assumption_initial_condition}
\| \uf_{int}\|_{\alpha} \leq N_0^{int},
\end{align}

\item[\hypertarget{B4}{\textbf{(B4)}}] a family of $C^{\alpha -2}$-distributions $\left\{a \di \partial_1^2 \vp (\cdot, a_0) \right\}_{a_0 \in \left[ \lambda , 1 \right]}$ and $N \in \mathbb{R}$ such that $[a]_{\alpha} \leq N \leq 1$ and 
\begin{align}
\label{prop1_3}
 \| \left[a, (\cdot) \right] \di \partial_1^2 \vp \|_{2\alpha -2, 2} \lesssim N N_0.
\end{align}
\end{itemize}

\quad Under these assumptions, there exists a solution $\uf \in C^{\alpha}(\mathbb{R}^2_+)$ of 
\begin{align}
\begin{split}
\label{theorem_1_main_ivp}
(\partial_2 - a \diamond  \partial_1^2 + 1) \uf & = f \hspace{3cm} \textrm{ in } \quad \mathbb{R}^2_+,\\
\uf & = \uf_{int}   \hspace{2.55cm} \textrm{ on } \quad \partial \mathbb{R}^2_+
\end{split}
\end{align}
that may be decomposed as $\uf = \upa+ \ui$, where $\upa \in C^{\alpha}(\mathbb{R}^2)$ solves
\begin{align}
(\partial_2 - a \di \partial_1^2 +1) \upa & = f &&  \textrm{in} \quad  \mathbb{R}^2\label{forcing_theorem_1}
\end{align}
and is modelled after $\vp$, solving \eqref{periodic_mean_free}, according to $a$ on $\mathbb{R}^2$ and $\ui \in C^{\alpha}(\mathbb{R}^2_+)$ solves 
\begin{align}
\begin{split}
\label{IVP_theorem_1}
(\partial_2 - a \diamond \partial_1^2 +1) \ui & = 0 \hspace{3cm} \textrm{in} \quad \mathbb{R}^2_+ ,\\
\ui & = \uf_{int} - \upa  \hspace{1.95cm}  \textrm{on} \quad  \partial \mathbb{R}^2_+.
\end{split}
\end{align}
The function $\ui$ may be further decomposed as $\ui = \tilde{q} + w$, where $\tilde{q}$ is the even-reflection of the function defined in Definition \ref{q} (below) and $w \in C^{2\alpha}(\mathbb{R}^2)$ such that $w \equiv 0$ on $\mathbb{R}^2_-$. The function $\tilde{q}$ is modelled after $\tilde{\vi}$, the even-reflection of the function defined in \eqref{vi_short}, according to $a$ on $\R^2$. We find that the solution $U = u + q + w$ is unique in the class of functions admitting such a splitting.

\quad We, furthermore, obtain the bounds
\begin{align}
\label{theorem_1_i_holder}
 \| \tilde{q} \|_{\alpha} + \| w \|_{\alpha} + \| u \|_{\alpha} & \lesssim N_0  + N_0^{int}
\end{align}
and 
\begin{align}
\label{theorem_1_i_modelling}
M \leq M_q + M_u + [w]_{2\alpha} \lesssim N_0 + N_0^{int},
\end{align}
where $M_q$ corresponds to the modelling of $\tilde{q}$ after $\tilde{\vi}$ and $M_u$ corresponds to the modelling of $u$ after $\vp$, both according to $a$. The constant $M$ is associated to the modelling of $ u + \tilde{q}  +w$ after  $\tilde{\vi} + v$ according to $a$.
\vspace{.2cm}

\item[ii)] (Stability) Let $i, j = 0,1$. Assume that we have:

\begin{itemize}[leftmargin=.5in]
\item[\hypertarget{C1}{\textbf{(C1)}}] $f_i \in C^{\alpha -2}(\mathbb{R}^2)$ satisfying \hyperlink{B1}{\textbf{(B1)}} and $\delta N_0 \in \mathbb{R}$ such that 
 \begin{align}
 \label{assumption_on_forcing_difference}
  \|f_1 - f_0\|_{\alpha -2} \leq  \delta N_0,
 \end{align}
 
\item[\hypertarget{C2}{\textbf{(C2)}}] $a_i\in C^{\alpha}(\mathbb{R}^2)$ satisfying \hyperlink{B2}{\textbf{(B2)}} and $[a_i]_{\alpha} \leq N$ with $N \in \R$ from \hyperlink{B4}{\textbf{(B4)}},\\

\item[\hypertarget{C3}{\textbf{(C3)}}]  $\uf_{int,i} \in C^{\alpha}(\mathbb{R})$ satisfying \hyperlink{B3}{\textbf{(B3)}} and  $\delta N_0^{int}\in \mathbb{R}$ such that 
 \begin{align}
\label{initial_condition_assumption_thm_1}
\| \uf_{int,1} - \uf_{int,0}\|_{\alpha} \leq \delta N_0^{int},
\end{align}

\item[\hypertarget{C4}{\textbf{(C4)}}]  $\left\{ a_i \diamond \partial_1^2 v_j(\cdot, a_0) \right\}_{a_0 \in [\lambda, 1]}$ satisfying \hyperlink{B4}{\textbf{(B4)}} and $\delta N \in \mathbb{R}$ such that
\begin{align}
\|a_1 - a_0 \|_{\alpha} & \leq \delta N,\\
  \left\| [a_i, (\cdot)] \diamond \partial_1^2 v_0  - [a_i, (\cdot)] \diamond \partial_1^2 v_1\right\|_{2\alpha -2,1} & \leq N \delta N_0 \label{stability_prop_1_4},\\
\textrm{ and } \quad  \left\| [a_0, (\cdot)] \diamond \partial_1^2 v_i  - [a_1, (\cdot)] \diamond \partial_1^2 v_i \right\|_{2\alpha -2,1} & \leq \delta N N_0. \label{stability_prop_1_5}
\end{align}
\end{itemize}

\quad We denote the solution of \eqref{theorem_1_main_ivp} provided by \textit{i)} that corresponds to $f_i$, $\uf_{int,i}$, and $a_i$ as $\uf_i$, which is decomposed as $\uf_i = \upa_i + \ui_i = \upa_i + \tilde{q}_i + w_i$. Under the above assumptions, we obtain
\begin{align}
\label{theorem_1_ii_holder}
\begin{split}
 \| u_1 - u_0 \|_{\alpha} + \| \tilde{q}_1 - \tilde{q}_0 \|_{\alpha} + \| w_1 - w_0 \|_{\alpha}
 \lesssim \delta N  (N_0 + N_0^{int})+ \delta N_0 + \delta N_0^{int}. \quad
\end{split}
\end{align}
Furthermore, $\upa_1 + \tilde{q}_1 + w_1 - (\upa_0 + \tilde{q}_0 + w_0)$ is modelled after $(v_1+ \tilde{\vi}_1, -(v_0+ \tilde{\vi}_0))$ according to $(a_1, a_0)$ with modelling constant $\delta M$ satisfying
\begin{align}
\label{theorem_1_ii_modelling}
\delta M \lesssim \delta M_q + \delta M_u + [w_1 - w_0]_{2 \alpha} \lesssim \delta N (N_0 + N_0^{int}) + \delta N_0 + \delta N_0^{int},
\end{align}
where $\delta M_q$ corresponds to the modelling of $\tilde{q}_1 - \tilde{q}_0$ after $(\tilde{\vi}_1, -\tilde{\vi}_0)$ according to $(a_1, a_0)$ and $\delta M_u$ to the modelling of $u_1 - u_0$ after $(v_1, -v_0)$ according to $(a_1, a_0)$.
\end{itemize}
\end{theorem}

\noindent Since $(\vp +\tilde{\vi})(\cdot, a_0)|_{\mathbb{R}^2_+} = V(\cdot, a_0)$, where $V(\cdot, a_0)$ solves \eqref{frozen_main}, we have recovered the predicted modelling for $\uf$ solving  \eqref{IVP_intro_linear}.

\begin{remark}[Domains of $U$ and $\ui$] We remark that both $U$ and $\ui$ are the restrictions of $C^{\alpha}$-functions that are actually defined on all of $\mathbb{R}^2$. Throughout this paper, we use, \textit{e.g.}, $U$ and $U|_{\mathbb{R}^2_+}$ interchangeably -- the domain that is meant being clear from the context. The same goes for $\ui$.  
\end{remark}

\begin{remark}[Singular products in Theorem \ref{linear_theorem}]
All of the $\diamond$-products in Theorem \ref{linear_theorem} are obtained using Lemma \ref{lemma:reconstruct_2}:
\begin{itemize}[leftmargin=.5cm]
\item $a \diamond \partial_1^2 \upa := a \diamond \partial_1^2 \upa|_{\mathbb{R}^2_+}$ via the modelling of $\upa$ after $v$ according to $a$,
\item $ a\diamond \partial_1^2 \ui := a \diamond \partial_1^2 (\tilde{q} + w) |_{\mathbb{R}^2_+}$ via the modelling of $\tilde{q} + w$ after $\tilde{\vi}$ according to $a$, 
\item and $a \diamond \partial_1^2 \uf := a \diamond \partial_1^2 (u + \tilde{q} + w)|_{\mathbb{R}^2_+}$ via the modelling of  $u+\tilde{q} + w$ after $v+ \tilde{\vi}$ according to $a$. 
\end{itemize}
\end{remark}

The proof of Theorem \ref{linear_theorem}, which is given in Section \ref{section:proof_theorem_1}, is a combination of the following three propositions -- the first of which handles \eqref{forcing_theorem_1} and the second and third of which contain the treatment of the ``initial boundary layer'', \textit{i.e.} $\ui$ solving \eqref{IVP_theorem_1}. Here is the first of these propositions:

\begin{proposition}[Modified Proposition 3.8 of \cite{OW}] 
\label{linear_forcing}
We adopt the assumptions from Theorem \ref{linear_theorem}, under which we obtain:

\begin{itemize}
\item[i)](Construction of Solution Operator) There exists a unique $\upa \in C^{\alpha}(\mathbb{R}^2)$ that is modelled after $\vp$ according to $a$ such that 
\begin{align}
\label{prop1_5}
(\partial_2 - a \di \partial_1^2 +1) \upa & = f  && \textrm{in} \quad  \mathbb{R}^2.
\end{align}
The modelling constant $M$ and $C^{\alpha}$-norm of $\upa$ are bounded as
\begin{align}
M +  \|\upa\|_{\alpha} & \lesssim  N_0 \label{edit_edit_3}.
\end{align}

\item[ii)] (Stability) Let $i = 0,1$. Denoting the solutions given by part $i)$ corresponding to $a_i$ and $f_i$ as $\upa_i$, we find that $\upa_1 - \upa_0$ is modelled after $(\vp_1, -\vp_0)$ according to $(a_1, a_0)$.  The modelling constant $\delta M$ and $\|\upa_1 - \upa_0 \|_{\alpha}$ satisfy 
\begin{align}
\delta M + \|\upa_0 - \upa_1\|_{\alpha} & \lesssim    N_0 \delta N  +   \delta N_0  \label{modelling_difference_prop_1}.
\end{align}
\end{itemize}
\end{proposition}

\noindent This is a variant of \cite[Proposition 3.8]{OW} --use of periodicity in the $x_2$-direction is replaced by exploitation of the massive term. The proof is contained in Section \ref{section:treat_linear_2}.

Propositions \ref{ansatz_IVP} and \ref{linear_IVP} handle the ``initial boundary layer''. The point is that, thanks to the bounds in Lemmas \ref{semigroup_bounds} and \ref{small_v}, \eqref{IVP_theorem_1} may be treated in an entirely classical manner --the strategy for solving \eqref{IVP_theorem_1} is to postulate an ansatz, which we then correct. For this we introduce the notation
\begin{align}
\label{notation_inter}
\vi^{\prime}(\cdot, a_0):= \vi(\cdot, a_0, \uf_{int} - \upa),
\end{align}
where $\upa$ is the solution of \eqref{prop1_5}, and notice that the most naive ansatz for the solution of \eqref{IVP_theorem_1} is $\vi^{\prime}(\cdot, a)$. However, in order for our arguments to work we need more smoothness for $a$, which leads us to the following definition:

\begin{definition}[Ansatz for $\ui$]
\label{q}
The ansatz for $\ui$ solving \eqref{IVP_theorem_1} is 
\begin{align}
\label{defn_q}
q:=\vi^{\prime}(\cdot, \bar{a}),
\end{align}
where $\bar{a}$ solves 
\begin{align}
\begin{split}
\label{a_equn}
(\partial_2 - \partial_1^2) \bar{a}& = 0 \hspace{3cm} \textrm{in} \quad \mathbb{R}^2_+,\\
\bar{a} & = a  \hspace{3cm}  \textrm{on} \quad \partial \mathbb{R}^2_+.
\end{split}
\end{align} 
\end{definition}

\noindent Notice that the definition of $q$ only depends on $a|_{\left\{x_2 =0 \right\}}$ and that, thanks to the lack of a massive term in \eqref{a_equn}, $\bar{a} \geq \lambda$ whenever $a|_{\left\{x_2 =0 \right\}}\geq \lambda$. Also,  ``$\, ^\prime \, $'' in \eqref{notation_inter} does not indicate a derivative, but is only meant to distinguish \eqref{notation_inter} from \eqref{vi_short}. 

In Proposition \ref{ansatz_IVP}, we investigate the modelling of the ansatz from Definition \ref{q}, but with a slightly more general initial condition. In particular, we find that:

\begin{proposition}[Analysis of the ansatz for the ``initial boundary layer'']
\label{ansatz_IVP}
Let $\alpha \in (\frac{1}{2},1)$ and $\lambda>0$. We use the notation from Definition \ref{constant_soln} and the convention  \eqref{vi_short}. The constants $N_0^{int}$, $N_0$, $\delta N_0$, and $\delta N_0^{int}$ are taken from Theorem \ref{linear_theorem}. We obtain that:

\begin{itemize}[leftmargin=.2in] 
\item[i)]  Assume that $f \in C^{\alpha-2}(\mathbb{R}^2)$ satisfies the condition \hyperlink{B1}{(B1)}, $a \in C^{\alpha}(\mathbb{R}^2)$ satisfies $\|a\|_{\alpha}\leq1$ and $a\in [\lambda,1]$, $U_{int} \in C^{\alpha}(\mathbb{R})$ satisfies \hyperlink{B3}{(B3)}, and $u \in C^{\alpha}(\mathbb{R})$ is modelled after $v$ (solving \eqref{periodic_mean_free}) according to $a$ on $\left\{x_2 = 0 \right\}$ with modelling constant $M_{\partial}$ and with respect to $\nu_{\partial}$. In analogue to \eqref{notation_inter} and \eqref{defn_q}, we use the convention 
\begin{align}
\label{ansatz_defn_new}
\vi^{\prime}_{u}(\cdot, a_0): = \vi(\cdot, a_0, U_{int} - u) \quad \textrm{ and } \quad q_{u} := \vi^{\prime}_{u}(\cdot, \bar{a}),
\end{align}
where $\bar{a}$ solves \eqref{a_equn} with initial condition $a$. We remark that these new conventions only differentiate themselves from \eqref{notation_inter} and \eqref{defn_q} in that now $u$ must not be the solution of \eqref{prop1_5} from Proposition \ref{linear_forcing}.

\quad Under these assumptions, the function $q_{u}$ is modelled after $\vi$ (defined in \eqref{vi_short}) according to $a$ on $\mathbb{R}^2_+$ with modelling constant $M$ and $C^{\alpha}$-norm  satisfying 
\begin{align}
M & \lesssim M_{\partial} + \| \nu_{\partial}\|_{2\alpha -1} +\|u \|_{\alpha} + N_0 + N_0^{int} \label{modelling_q_constant}\\
\textrm{and} \quad \|q_{u}\|_{\alpha}& \lesssim \|u \|_{\alpha} + N_0^{int}. \label{holder_q_constant}
\end{align}
The even-reflection $\tilde{q}_{u}$ is modelled after $\tilde{\vi}$ according to $a$ on $\mathbb{R}^2$ and the modelling constant still satisfies \eqref{modelling_q_constant}. 
 
\item[ii)] Let $i = 0,1$. Assume that the $f_i \in C^{\alpha -2}(\mathbb{R}^2)$ satisfy the condition \hyperlink{C1}{(C1)}, the $a_i \in C^{\alpha}(\mathbb{R}^2)$ satisfy the conditions of part $i)$, the $U_{int,i} \in C^{\alpha}(\mathbb{R})$ satisfy \hyperlink{C3}{(C3)}, and the $u_i$ are of the class $C^{\alpha}(\mathbb{R})$. Additionally, we assume that $u_1 - u_0$ is modelled after $(v_1, -v_0)$ according to $(a_1, a_0)$ on $\{x_2 = 0 \}$ with modelling constant $\delta M_{\partial}$ and the associated $\delta\nu_{\partial}$.

\quad Under these assumptions, $q_{u_1} - q_{u_0}$ is modelled after $(\vi_1, -\vi_0)$ according to $(a_1, a_0)$ on $\mathbb{R}^2_+$ with modelling constant $\delta M$ and $ \| q_{u_1} - q_{u_0}\|_{\alpha}$  bounded by 
\begin{align}
 \label{modelling_q_constant_difference}
\begin{split}
\delta M  \lesssim & \delta M_{\partial} + \|\delta \nu_{\partial}\|_{2\alpha -1} +\|a_1 - a_0\|_{\alpha} ( \displaystyle\max_{i = 0,1} \|u_i\|_{\alpha} + N_0^{int} ) \\
& + \|u_1- u_0\|_{\alpha} + \delta N_0^{int} + \delta N_0
\end{split}
\end{align}
and
\begin{align}
 \| q_{u_1} - q_{u_0}\|_{\alpha}  \lesssim\|a_1 - a_0\|_{\alpha}  ( \displaystyle\max_{i = 0,1} \|u_i\|_{\alpha} + N_0^{int} ) + \|u_1- u_0\|_{\alpha} + \delta N_0^{int} \label{holder_q_constant_difference}
\end{align}
The even-reflection $\tilde{q}_{u_1} - \tilde{q}_{u_0}$ is modelled after $(\tilde{\vi}_1, -\tilde{\vi}_0)$ according to $(a_1, a_0)$ on $\mathbb{R}^2$ and the modelling constant still satisfies \eqref{modelling_q_constant_difference}. 
\end{itemize}
\end{proposition}
\noindent The proof of Proposition \ref{ansatz_IVP} is contained in Section \ref{section:modelling_q}. In our proof of Theorem \ref{linear_theorem} we apply Proposition \ref{ansatz_IVP} with $u$ taken as the solution of \eqref{prop1_5} obtained in Proposition \ref{linear_forcing}. We write Proposition \ref{ansatz_IVP} in its slightly more general form due to its application in Step 1 of the proof of Theorem \ref{theorem_nonlinear} (see Section \ref{section:treat_nonlinear}).

To finish up the ingredients needed for our proof of Theorem \ref{linear_theorem}, we correct the ansatz defined in Definition \ref{q} in order to solve \eqref{IVP_theorem_1}. In particular, we prove the following:

\begin{proposition}[Analysis of the linear problem with a trivial forcing]
\label{linear_IVP}
We adopt the assumptions and notations from Theorem \ref{linear_theorem}. We, furthermore, use the notation from Definition \ref{q}.
\begin{itemize}[leftmargin=.2in]
\item[i)] (Construction of Solution Operator) There exists a unique $w \in C^{2\alpha}( \mathbb{R}^2)$ with $w \equiv 0$ on $\mathbb{R}^2_-$ such that $\ui = \tilde{q} + w$ solves
\begin{align}
\begin{split}
\label{IC_2}
(\partial_2 - a \diamond \partial_1^2 +1) \ui& = 0 \hspace{3cm} \textrm{in} \quad \mathbb{R}^2_+,\\
\ui & = \uf_{int} - \upa \hspace{1.95cm} \textrm{on}\quad  \partial \mathbb{R}^2_+.
\end{split}
\end{align}
The $C^{\alpha}$-norm of $w$ and the $C^{2\alpha}$-seminorm, which corresponds to its trivial modelling, satisfy
\begin{align}
\label{modelling_estimate_ivp_w}
\|w\|_{\alpha} + [w]_{2\alpha} \lesssim N(N_0 + N_0^{int}).
\end{align}

\item[ii)] (Stability) Let $i = 0,1$. We denote the solutions corresponding to $\uf_{int,i}$, $a_i$, and $f_i$ from part $i)$ as $w_i$. Then, the $C^{\alpha}$-norm of $w_1 - w_0$ and the $C^{2\alpha}$-norm, corresponding to its trivial modelling, satisfy
\begin{align}
\label{modelling_prop_2_ii}
[w_1 - w_0]_{2\alpha} + \| w_1 - w_0\|_{\alpha}  \lesssim   \delta N_0^{int} + \delta N_0  +  \delta N (N_0 + N_0^{int}).
\end{align}
\end{itemize}

\end{proposition}

\noindent The proof of Proposition \ref{linear_IVP} is given in Section \ref{section:linear_IVP}.

The main analytic tool that we use in our arguments for Propositions \ref{linear_forcing} and \ref{linear_IVP} is the following adaption of Safonov's approach to Schauder theory:

\begin{lemma}[Modified Lemma 3.6 of \cite{OW}]
\label{KrylovSafonov}
Let $\alpha \in \left( \frac{1}{2}, 1\right)$,  $I \in \mathbb{N}$, and $\lambda >0$. Assume we have $I$ families of periodic distributions $\left\{f_1(\cdot, a_0), ..., f_I(\cdot, a_0) \right\}$ indexed by $a_0 \in \left[ \lambda,1 \right]$,  $I$ constants $N_i \in \mathbb{R}$ such that 
\begin{align}
\label{lemma5_assumption_1}
 \left\|  f_{i}(\cdot, a_0) \right\|_{\alpha -2,1} \leq N_i,
\end{align}
and $a: \mathbb{R}^2 \rightarrow \left[\lambda, 1\right]$ satisfying  $[ a ]_{\alpha} \ll 1$. Let $u \in C^{\alpha}(\mathbb{R}^2)$ be $x_1$-periodic and modelled after $(v_1, ... , v_I)$ defined in terms of \eqref{periodic_mean_free} according to $a$ and $(\sigma_1, ..., \sigma_I)$, with modelling constant $M$, and satisfy
\begin{align}
\label{Lemma5_3}
\displaystyle\sup_{T\leq1} (T^{\frac{1}{4}})^{2 -2\alpha} \| (\partial_2  - a \partial_1^2 +1)u_T -  \sigma_i E  f_{iT}(\cdot, a_0) \| \leq K
\end{align}
for some $ K \in \mathbb{R}$, where $E$ denotes evaluation of a function of $(x, a_0)$ at $(x, a(x))$. Then we find that 
\begin{align}
\label{lemma_5_result_model_holder}
M + \| u\|_{\alpha}  & \lesssim  K + \| \sigma_i\|_{\alpha} N_i.
\end{align}
 \end{lemma}
\noindent In Proposition \ref{linear_forcing}, the purpose of this lemma is to pass to the limit in a family of regularized solutions. In Proposition \ref{linear_IVP}, we obtain the correction $w$ by using a trivial version of Lemma \ref{KrylovSafonov} (with $\sigma_i = 0$). Since the argument for this analytic workhorse of our paper actually sees quite substantial modification from the analogous result in \cite{OW}, we give the full proof in Section \ref{section:KS}.

We now come to the nonlinear result of this paper, Theorem \ref{theorem_nonlinear}, which treats the quasilinear initial value problem. The proof of this result relies on a fixed-point argument that takes Theorem \ref{linear_theorem} as input. Here is the statement of our result:

\begin{theorem}[Analysis of the quasilinear problem]
\label{theorem_nonlinear}
 Let $\alpha \in \left(\frac{2}{3},1\right)$. 
 \begin{itemize}[leftmargin=.2in]
\item[i)](Construction of Solution Operator)  Assume that $f \in C^{\alpha -2}(\mathbb{R}^2)$ satisfies \hyperlink{B1}{\textbf{(B1)}}  and the pair $(f,f)$ satisfies condition \hyperlink{A}{\textbf{(A)}}, $\uf_{int} \in C^{\alpha}(\mathbb{R})$ satisfies \hyperlink{B3}{\textbf{(B3)}}, and $a: \mathbb{R} \rightarrow [\lambda,1]$ for $\lambda>0$ satisfies $\| a^{\prime} \|, \| a^{\prime \prime} \|, \| a^{\prime \prime \prime} \| \leq 1$. We use the notation from Definitions \ref{constant_soln} and \ref{extensions}. Let $N_0, N_0^{int} \ll 1$.

\quad Then there exist $u \in C^{\alpha}(\mathbb{R}^2)$ and $w \in C^{2\alpha}(\mathbb{R}^2)$ such that $w\equiv 0$ on $\mathbb{R}^2_-$ and $\ui:= w + q$ solves \eqref{theorem_1_main_ivp} with $a:= a(u + w + \tilde{q})$. Here, $q$ is defined in terms of Definition \ref{q} with $a|_{\left\{ x_2 = 0 \right\}}:= a(U_{int}-u|_{\left\{ x_2 = 0 \right\}})$. The function $u$ solves \eqref{forcing_theorem_1} with $a:= a(u + w + \tilde{q})$ and is modelled after $v$ according to $a(u + w + \tilde{q})$. Lastly, the function $\uf := u + w + q$ solves 
\begin{align}
\label{IVP_theorem}
\begin{split}
\partial_2 \uf - a(\uf) \diamond \partial_1^2 \uf + \uf & = f \quad \quad \quad \quad \quad \quad \, \textrm{in} \quad \mathbb{R}^2_+, \\
\uf & = \uf_{int} \quad \quad \quad \quad \quad  \textrm{on} \quad \partial \mathbb{R}^2_+,
\end{split}
\end{align}
and $u+w +\tilde{q}$ is modelled after $\tilde{\vi} + v$ according to $a(u + w + \tilde{q})$.

\quad Under the additional smallness condition 
\begin{align}
\label{smallness_assumption}
 \| u\|_{\alpha} + \|  w \|_{\alpha} \ll 1,
\end{align}
the solution $U = u +w + q$ is unique within the class of solutions admitting such a splitting. We, furthermore, have the a priori estimates
\begin{align}
\label{thm_2_conclusion}
\begin{split}
\|\uf  \|_{\alpha}& \leq \|  u \|_{\alpha} +  \|q \|_{\alpha} + \| w \|_{\alpha}  \lesssim N_0 + N_0^{int}\\
 \textrm{and} \quad  M & \leq M_u + M_q + [w]_{2\alpha} \lesssim N_0 + N_0^{int},
\end{split}
\end{align}
where $M$ is associated to the modelling of $u + \tilde{q} +w$ after $\tilde{\vi} + v$ and $M_u$ to the modelling of $u$ after $v$, both according to $a(u + \tilde{q} +w)$.\\

\item[ii)] (Stability) Let $i,j =0,1$. Assume that the $f_i \in C^{\alpha -2}(\mathbb{R}^2)$ satisfy \hyperlink{C1}{\textbf{(C1)}}, every pair $(f_i, f_j)$ satisfies the condition \hyperlink{A}{\textbf{(A)}}, and the $\uf_{int,i} \in C^{\alpha}(\mathbb{R})$ satisfy \hyperlink{C3}{\textbf{(C3)}}. Let $U_i$ denote the solutions constructed in part $i)$ that decompose as $U_i = u_i + q_i + w_i$.

\quad We find that  $u_1 + \tilde{q}_1 + w_1 - (u_0 + \tilde{q}_0 + w_0)$ is modelled after $(\tilde{\vi}_1+ v_1, \tilde{\vi}_0+ v_0)$ according to $(a(u_1 + \tilde{q}_1 + w_1  ), a(u_0 + \tilde{q}_0 + w_0))$ and $(1,-1)$ with modelling constant $\delta M$ and $C^{\alpha}$-norm satisfying 
\begin{align}
\label{thm_2_conclusion_2}
\| \uf_1 - \uf_0 \|_{\alpha}  \leq  \| u_1  - u_0 \|_{\alpha} + \| q_1 - q_0 \|_{\alpha} + \| w_1 - w_0\|_{\alpha} & \lesssim \delta N_0 + \delta N_0^{int}\\
\textrm{and} \quad  \delta M \leq  \delta M_u + \delta M_q + [w_1 - w_0]_{2\alpha} & \lesssim \delta N_0 + \delta N_0^{int}.
\end{align}
Here, $\delta M_u$ corresponds to the modelling of $u_1 - u_0$ after $(v_1, -v_0)$ and $\delta M_q$ to the modelling of $\tilde{q}_1 - \tilde{q}_0$ after $(\tilde{\vi}_1, -\tilde{\vi}_0)$, both  according to $(a(u_1 + \tilde{q}_1 + w_1  ), a(u_0 + \tilde{q}_0 + w_0))$.
\end{itemize}
\end{theorem}

\noindent The proof of Theorem \ref{theorem_nonlinear} is given in Section \ref{section:treat_nonlinear}. 

\begin{remark}[Short-time existence and uniqueness]
We remark that a similar proof as that for Theorem \ref{theorem_nonlinear} would yield short-time existence and uniqueness without the smallness assumption $N_0, N_0^{int} \ll 1$ on the data.
This may be more convenient in cases when the smallness assumption is difficult to verify.
\end{remark}


\begin{remark}[Higher dimensions]
The choice of $1+1$ dimensions is due to notational convenience --Theorems \ref{linear_theorem} and \ref{theorem_nonlinear} also hold in $d+1$ dimensions for $d>1$.
\end{remark}

\section{Proof of Lemma \ref{KrylovSafonov} (\textit{Safonov Lemma}): Main PDE ingredient} 
\label{section:KS}

\begin{proof} We adapt the argument for \cite[Lemma 3.6]{OW} --substantial changes may be found in the first five steps of the argument.\\
 
\noindent\textbf{Step 1:} \textit{(Bound for $\nu$)} \quad One begins by obtaining the bound:
\begin{align}
\| \nu \|_{2\alpha -1}  & \lesssim M + \| \sigma_i\| N_i \label{nu_bound_1}.
\end{align}
Using the same strategy as for \cite[(5.108)]{OW}, yields the necessary bound on the seminorm of $\nu$; here, one uses the assumption \eqref{lemma5_assumption_1} and Lemma \ref{small_v}.

The modelling assumption also yields the $L^{\infty}$-bound for $\nu$. In particular, using the triangle inequality we obtain
\begin{align*}
|\nu(x) (y-x)_1 | \leq M d^{2\alpha}(x,y) + |u(y) - u(x) - \sigma_i(v_i(y, a_i(x)) - v_i(x, a_i(x)))|,
\end{align*}
for any $x,y \in \R^2$. Exploiting the periodicity of $u$ and $v_i(\cdot , a_i(x))$ in the $x_1$-direction, gives that $\| \nu \| \leq M$.\\

\noindent \textbf{Step 2:} \textit{($u$ is Lipschitz on large scales and bound for $[u]_{\alpha}^{loc}$)} \quad We claim that
\begin{align}\label{local_u_norm_bound}
 [u]_{\alpha}^{loc}\lesssim M + \|\sigma_i \| N_i.
\end{align}
Indeed, let $x,y  \in \mathbb{R}^2$ such that $d(x,y) \leq 1$ and notice that
\begin{align*}
\frac{|u(x) - u(y)|}{d^{\alpha}(x,y)} & \leq M d^{\alpha}(x,y) + \| \sigma_i \| \frac{|v_i(x,a(y)) - v_i(y,a(y))|}{d^{\alpha}(x,y)} + \frac{|\nu(y) (x-y)_1| }{d^{\alpha}(x,y)}\\
& \lesssim M + \|\sigma_i \| N_i,
\end{align*}
where, in addition to the modelling of $u$, we used the $L^{\infty}$-bound for $\nu$ from Step 1.

Moreover, it is a consequence of the triangle inequality for $|\cdot |$ and \eqref{definition_beta_01_loc} that
\begin{align}
\label{large_scale_lip}
|u(x) - u(y)| \leq \lceil d(x,y) \rceil [u]^{loc}_{\alpha} \lesssim (M + \|\sigma_i \| N_i) d(x,y)
\end{align}
for any points $x,y \in \mathbb{R}^2$ such that $d(x,y) \geq 1$. In particular, we consider a sequence of points along the line connecting $x$ and $y$: Starting at $x$ we move a distance $1$ along the line for the next point. We choose points like this $\lfloor d(x,y) \rfloor$ times --and then take the final point as $y$. Letting $z_1, z_2$ be two subsequent points we use the bounds of the form $|u(z_1)  - u(z_2)| \leq [u]^{loc}_{\alpha} d^{\alpha}(z_1, z_2)$ --which, whenever $d(z_1, z_2) = 1$, simply reduces to $|u(z_1)  - u(z_2)| \leq [u]^{loc}_{\alpha}$.\\

\noindent \textbf{Step 3:} \textit{(Equations satisfied by $u_T$)} \quad  In this step, we show that, for any point $x_0 \in \mathbb{R}^2$ and $T \in (0,1]$, the function $u_T$ solves 
 \begin{align}
 \label{Lemma5_9}
 (\partial_2 -  a(x_0)\partial_1^2  +1) (u_T - \sigma_i(x_0) v_{iT}(\cdot, a(x_0))) = g_{x_0}^T \quad \quad \textrm{ in } \quad \mathbb{R}^2,
 \end{align}
where, for any $x \in \R^2$, $g_{x_0}^T(x)$ satisfies the pointwise estimate
 \begin{align}
 \label{Lemma5_10}
  |g_{x_0}^T(x)| \lesssim  \tilde{N} ((T^{\frac{1}{4}})^{2\alpha -2}   +  d^{\alpha}(x, x_0) (T^{\frac{1}{4}})^{\alpha -2})
 \end{align}
 with $\tilde{N} = K + [a]_{\alpha} M + \|\sigma_i\|_{\alpha}N_i $.
 
Additionally, we find that, for $T \in (0,1]$, $u_T$ solves 
\begin{align}
\label{Lemma5_12}
 (\partial_2 -  a \partial_1^2 +1)u_T = h^T \quad \quad \quad  \textrm{in} \quad \mathbb{R}^2,
\end{align}
where
\begin{align}
\label{h_bound_L_infty}
\begin{split}
\| h^T \| \lesssim K (T^{\frac{1}{4}})^{2\alpha -2} + \|\sigma_i\| N_i(T^{\frac{1}{4}})^{\alpha -2}.\\
\end{split}
\end{align}

We begin by showing the first result. Simple manipulations show the tautological observation that $u_T$ solves 
\begin{align}
\label{Lemma5_12.5}
 (\partial_2 -  a(x_0)\partial_1^2 +1)u_T =  \sigma_i(x_0)   f_{iT}(\cdot, a(x_0) ) +  g_{x_0}^T \quad \quad \textrm{in} \quad \mathbb{R}^2,
\end{align}
where 
\begin{align}
\label{Lemma5_13}
\begin{split}
g_{x_0}^T := & (\partial_2 -  a \partial_1^2+1) u_T - \sigma_i E  f_{iT} (\cdot, a_0 ) \\
& + (a - a(x_0)) \partial_1^2 u_T  + (\sigma_i - \sigma_i(x_0)) E f_{iT}(\cdot, a_0 )\\
& + \sigma_i(x_0) ( E   f_{iT}(\cdot,a_0 ) -  f_{iT}(\cdot, a(x_0) )).
\end{split}
\end{align}
For every $x \in \mathbb{R}^2$, we then bound:
\begin{align}
\label{Lemma5_14} 
\begin{split}
   & | g_{x_0}^T(x)|\\
&  \lesssim  \| (\partial_2 -  a \partial_1^2 +1) u_T - \sigma_i E f_{iT}(\cdot, a_0) \| + |a(x) - a(x_0)| | \partial_1^2 u_T(x) | \\
& + | ( \sigma_i(x) - \sigma_i (x_0)) E f_{iT} (\cdot, a_0)|  +  \| \sigma_i(x_0) ( E   f_{iT}(\cdot, a_0 ) -  f_{iT}(\cdot, a(x_0) )) \| \\ 
 & \leq  K (T^{\frac{1}{4}})^{2\alpha -2}   + d^{\alpha}(x, x_0) \Big(  [a]_{\alpha}  | \partial_1^2 u_T (x)| + [\sigma_i]_{\alpha}   \|  f_{iT}(\cdot, a_0)\| +  \| \sigma_i \|  [a]_{\alpha} \|  f_{iT}\|_1 \Big)\\
 & \leq  K (T^{\frac{1}{4}})^{2\alpha -2}   + d^{\alpha}(x, x_0)  \Big( [a]_{\alpha} | \partial_1^2 u_T (x)|  +\| \sigma_i\|_{\alpha}  N_i (T^{\frac{1}{4}})^{\alpha -2}\Big),
\end{split}
\end{align}
where we have used assumptions \eqref{lemma5_assumption_1},  \eqref{Lemma5_3}, and $[a]_{\alpha} \leq 1$.

By \eqref{local_u_norm_bound}, in order to obtain \eqref{Lemma5_10}, it suffices to show the bound 
\begin{align}
\label{Lemma5_14.1}
| \partial_1^2 u_T (x)| \lesssim \left[ u \right]^{loc}_{\alpha} (T^{\frac{1}{4}})^{\alpha-2}
\end{align}
for any point $x \in \mathbb{R}^2$. For \eqref{Lemma5_14.1} we use the moment bound \eqref{moment_bound} and the triangle inequality to write
 \begin{align*}
   &| \partial_1^2 u_T (x)|\\
   & =    \Big| \int_{\mathbb{R}^2} (u(y) - u(x)) \partial_1^2 \psi_T(x-y) \, \textrm{d}y \Big|\\
& \lesssim \left[ u\right]_{\alpha}^{loc} \Big(\int_{B_1(x)} | \partial_1^2 \psi_T(x-y) |d^{\alpha}(x,y) \, \textrm{d}y  +  \int_{B_1^c(x)} | \partial_1^2 \psi_T(x-y) | d(x,y) \, \textrm{d}y\Big)\\
& \lesssim \left[ u\right]_{\alpha}^{loc} \left(  (T^{\frac{1}{4}})^{\alpha -2} +   (T^{\frac{1}{4}})^{-1} \right) \lesssim \left[ u\right]_{\alpha}^{loc} (T^{\frac{1}{4}})^{\alpha -2}.
\end{align*}
Plugging \eqref{Lemma5_14.1} into \eqref{Lemma5_14} yields the desired \eqref{Lemma5_10}. 

A different set of manipulations yields that $u_T$ solves \eqref{Lemma5_12} with
\begin{align*}
h^T =   (\partial_2 -  a \partial_1^2 +1)u_T -  \sigma_i E f_{iT}(\cdot, a_0)  +  \sigma_i E f_{iT}(\cdot, a_0).
\end{align*} 
Using the assumptions  \eqref{lemma5_assumption_1} and \eqref{Lemma5_3}, we obtain \eqref{h_bound_L_infty}.\\

\noindent \textbf{Step 4:} \textit{($L^{\infty}$-estimates on $u$ and $u_T$)} \quad  In this step we prove two $L^{\infty}$-bounds.
%
%
The first estimate is given by
\begin{align}
\label{new_linfty_bound}
\| u\|   \lesssim K + \|\sigma_i \| N_i  + M .
\end{align}
To see this, observe that an application of  \cite[Theorem 8.1.7]{Kr} to \eqref{Lemma5_12} yields together with \eqref{h_bound_L_infty} for all $T \in (0,1]$
\begin{align*}
\| u_T \| \lesssim  K (T^{\frac{1}{4}})^{2\alpha -2} + \|\sigma_i \| N_i(T^{\frac{1}{4}})^{\alpha -2}.
\end{align*}
We use this estimate with $T =1$ and combine it with \eqref{large_scale_lip} to the effect of
\begin{align*}
\begin{split}
\|u \| & = \|u_1 \| + \|u_1 - u \| 
   \lesssim K + \|\sigma_i \| N_i +  \int_{\mathbb{R}^2} \|u(\cdot - y) - u \| | \psi_1(y)| \textrm{d}y \\
&   \lesssim K + \|\sigma_i \| N_i +  [u]_{\alpha}^{loc} \int_{\mathbb{R}^2} (|y| + |y|^{\alpha}) |\psi_1(y)| \textrm{d}y
 \lesssim K + \|\sigma_i \| N_i +  [u]_{\alpha}^{loc}.
\end{split}
\end{align*} 
This yields \eqref{new_linfty_bound} via an application of \eqref{local_u_norm_bound}.

\medskip

The second estimate we prove is
\begin{align}
\label{uT_linfty_1_final}
 \| u_T - \sigma_i(x_0) v_{iT}(\cdot, a(x_0)) \|_{B_L(x_0)}
 \lesssim   \tilde{N} L^{\alpha}  (T^{\frac{1}{4}})^{\alpha -2} ,
\end{align}
which holds for $T \in (0,1]$ and $L \ge 1$. To obtain \eqref{uT_linfty_1_final} we use the equation \eqref{Lemma5_9} and, letting $G(a(x_0), x_1, x_2)$ be as in \eqref{heat_kernel}, we write $u_T - \sigma_i(x_0) v_{iT}(\cdot, a(x_0))$ as 
\begin{align}
\label{heat_kernel_rep_Lemma_5}
\begin{split}
 u_T(x) - \sigma_i(x_0) v_{iT}(x, a(x_0)) =  \int_{0}^{\infty} \int_{\mathbb{R}} g_{x_0}^T(x_1- y, x_2-s) G(a(x_0), y, s) \, \textrm{d} y \, \textrm{d} s. \qquad
\end{split}
\end{align}
Combining \eqref{heat_kernel_rep_Lemma_5} with the bound \eqref{Lemma5_10} and using the notation $x_0 = (x_{01}, x_{02})$, for $x \in B_L(x_0)$, we obtain that
\begin{align*}
\begin{split}
& |u_T(x) - \sigma_i(x_0) v_{iT}(x, a(x_0))| \\
\lesssim & \tilde{N} \int_{0}^{\infty} \int_{\mathbb{R}} \Big( (T^{\frac{1}{4}})^{2\alpha -2}  +  (T^{\frac{1}{4}})^{\alpha -2} (|x_1 - y - x_{01}|^{\alpha} +  |x_2 -s  - x_{02}|^{\frac{\alpha}{2}}  ) \Big) |G(a(x_0),  y, s)| \, \textrm{d} y \, \textrm{d} s\\
\lesssim &  \tilde{N}\Big( \vphantom{\int_{\mathbb{R^2}}} (T^{\frac{1}{4}})^{2\alpha -2}  +   L^{\alpha} (T^{\frac{1}{4}})^{\alpha -2} +  \int_{0}^{\infty} \int_{\mathbb{R}} (T^{\frac{1}{4}})^{\alpha -2}(|y|^{\alpha} + s^{\frac{\alpha}{2}} ) |G(a(x_0), y, s)|  \, \textrm{d} y \, \textrm{d} s \Big)\\
\lesssim & \tilde{N}\Big(  (T^{\frac{1}{4}})^{2\alpha -2}  +   L^{\alpha} (T^{\frac{1}{4}})^{\alpha -2} + (T^{\frac{1}{4}})^{\alpha -2}\Big) \lesssim \tilde{N} L^{\alpha}  (T^{\frac{1}{4}})^{\alpha -2} ,
\end{split}
\end{align*}
where we have used that $B_L(x_0)$ refers to the ball in terms of the parabolic metric and $T\le 1\le L$.\\

\noindent \textbf{Step 5:} \textit{(An excess decay)} \quad Let $0<R \ll L $, $T \in (0,1]$, and $x_0 \in \mathbb{R}^2$. Then, in this step we find that 
 \begin{align}
 \label{Lemma5_15}
 \begin{split}
 &\frac{1}{R^{2\alpha}} \displaystyle\inf_{l \in \textrm{Span}\{1, x_1 \}} \| u_T - \sigma_i(x_0) v_{iT}(\cdot, a(x_0)) - l \|_{B_R(x_0)}\\
  \lesssim& \Big(\frac{R}{L} \Big)^{2 (1 -\alpha)} \frac{1}{L^{2\alpha}} \displaystyle\inf_{l \in \textrm{Span}\{1, x_1 \}} \| u_T - \sigma_i(x_0) v_{iT}(\cdot, a(x_0)) - l  \|_{B_L(x_0)}\\
 &+ \tilde{N} \Big(\frac{L^2}{R^{2\alpha} (T^{\frac{1}{4}})^{2-2\alpha} } + \frac{ L^{2+\alpha}}{R^{2\alpha} (T^{\frac{1}{4}})^{2-\alpha}}  \Big).
 \end{split}
 \end{align}
%
%
First, assume $L\ge 1$.
On $B_L(x_0)$ we decompose the function  $u_T -  \sigma_i(x_0) v_{iT}(\cdot, a(x_0))$ into a ``near-field'' and ``far-field'' contribution. Letting $w_<$ be the decaying solution of
 \begin{align}
  \label{Lemma5_16.1}
  \begin{split}
(\partial_2 - a(x_0)\partial_1^2 ) w_< = & \chi_{B_L} (g_{x_0}^T -(  u_T - \sigma_i(x_0)v_{iT}(\cdot, a(x_0)) ) ) \quad \quad \quad \textrm{in} \quad \mathbb{R}^2
 \end{split}
 \end{align}
and defining $w_> := u_T - \sigma_i(x_0)v_{iT}(\cdot, a(x_0)) -w_<$, we find that $w_>$ satisfies
 \begin{align}
  \label{Lemma5_16.2}
  (\partial_2 - a(x_0) \partial_1^2 ) w_> = 0 \quad \quad \quad \textrm{ in } \quad B_L(x_0).
 \end{align}
We may then write 
 \begin{align}
 \label{Lemma5_17}
  \| w_< \|  & \lesssim   L^2 ( \| g_{x_0}^T\|_{B_L(x_0)}  +  \| u_T -  \sigma_i (x_0) v_{iT}(\cdot, a(x_0)) \|_{B_L(x_0)})
  \end{align}
  and
  \begin{align}
 \label{Lemma5_18}
  \| \{  \partial_1^2 , \partial_2 \} w_>\|_{B_{L/2}(x_0)} & \lesssim L^{-2} \| w_> -l \|_{B_L(x_0)}
 \end{align}
for any $l \in \textrm{Span}\{1, x_1 \}$ --the notation $\left\{ \partial_1^2, \partial_2 \right\}$ on the left-hand side of \eqref{Lemma5_18} indicates that the estimate holds for both $\partial_1^2 w_>$ and $\partial_2 w_>$. The estimate \eqref{Lemma5_17} follows immediately from the heat kernel representation of $w_<$ and the triangle inequality in $L^{\infty}$. The relation \eqref{Lemma5_18} is proven via Bernstein's argument in \cite[Theorem 8.4.4]{Kr} for $l = 0$. One can reduce to the case that $ l = 0$ since $w_> - l$ still solves \eqref{Lemma5_16.2} when $l \in \textrm{Span}\left\{1, x_1 \right\}$.

To maneuver ourselves into a position to apply \eqref{Lemma5_17} and \eqref{Lemma5_18} we write 
\begin{align*}
u_T - \sigma_i(x_0) v_{iT}(\cdot, a(x_0)) =w_< + w_>
\end{align*}
and use the triangle inequality and Taylor's theorem to the effect of
 \begin{align*}
  \| u_T - \sigma_i(x_0) v_{iT}(\cdot, a(x_0))  - l_R \|_{B_R(x_0)} \lesssim R^2 \| \{ \partial_1^2, \partial_2 \} w_> \|_{B_R(x_0)} + \| w_< \|_{B_R(x_0)}
 \end{align*}
 for $l_R = w_>(x_0) + \frac{\partial w_>}{\partial x_1}(x_0) (x-x_0)_1$. Again by the triangle inequality, now along with \eqref{Lemma5_17} and \eqref{Lemma5_18}, we then have that 
 \begin{align*}
 \begin{split}
 & \| u_T - \sigma_i(x_0) v_{iT}(\cdot, a(x_0))  - l_R \|_{B_R(x_0)}\\
& \lesssim   \left( \frac{R}{L}\right)^2 \| w_> -l \|_{B_L(x_0)} +\| w_< \|_{B_R(x_0)}\\
& \lesssim \left( \frac{R}{L}\right)^2 \|  u_T - \sigma_i(x_0) v_{iT}(\cdot, a(x_0)) -l \|_{B_L(x_0)} + 2 \| w_< \|\\
& \lesssim \left( \frac{R}{L}\right)^2 \|  u_T - \sigma_i(x_0) v_{iT}(\cdot, a(x_0)) -l \|_{B_L(x_0)} \\
& \qquad + 2 L^2 \big( \| g^T_{x_0} \|_{B_L(x_0)}  + \|  u_T - \sigma_i(x_0) v_{iT}(\cdot, a(x_0)) \|_{B_L(x_0)}  \big),
\end{split}
\end{align*}
for any $l \in \textrm{Span}\{1, x_1 \}$. We remark that to apply \eqref{Lemma5_18} we use that $R \ll L$, which we have assumed in this step. To finish, we use \eqref{Lemma5_10} and \eqref{uT_linfty_1_final}, the latter of which requires $L\geq 1$.

If $L< 1$, we instead decompose $u_T -  \sigma_i(x_0) v_{iT}(\cdot, a(x_0))$ by setting $w_<$ to be the solution to
\begin{align*}
 (\partial_2 - a(x_0)\partial_1^2 + 1) w_< = & \chi_{B_L} g_{x_0}^T \quad \quad \quad \textrm{in} \quad \mathbb{R}^2,
\end{align*}
and $w_>:=u_T -  \sigma_i(x_0) v_{iT}(\cdot, a(x_0)) - w_<$. Then $w_<$ satisfies the analogue of \eqref{Lemma5_17}, namely that
\begin{align*}
 \| w_< \|  & \lesssim   L^2 \| g_{x_0}^T\|_{B_L(x_0)}.
\end{align*}
Moreover, we find that now
\begin{align}\label{jc02}
  (\partial_2 - a(x_0) \partial_1^2 +1) w_> = 0 \quad \quad \quad \textrm{ in } \quad B_L(x_0).
 \end{align}
To finish, we want to show that
\begin{align}\label{jc03}
 \| \{  \partial_1^2 , \partial_2 \} w_>\|_{B_{L/2}(x_0)} & \lesssim L^{-2} \|w_> - l\|_{B_L(x_0)}
\end{align} 
for any $l \in \textrm{Span}\{1, x_1 \}$.

We first remark that \eqref{jc02} implies for any $l(x)=c_0+c_1(x-x_0)_1$ with $c_0,c_1\in\R$ the estimates
\begin{align}\label{jc01}
|c_0|\lesssim L^{-2}\|w_> - l\|_{B_L(x_0)}, \quad  |c_1|\lesssim L^{-3}\|w_> - l\|_{B_L(x_0)}.
\end{align}
Indeed, we have $(\partial_2 - a(x_0) \partial_1^2 +1) (w_> -l) = -l$ on $B_L(x_0)$, so that by testing this equation with a suitable smooth cut-off function on scale $L$ which annihilates $(x-x_0)_1$ we obtain $|c_0|\lesssim (L^{-2}+1)\|w_> - l\|_{B_L(x_0)}$. In particular, letting $\eta$ be an $L^1$-normalized, smooth, and radial cut-off function for $B_{\frac12}(x_0)$ in $B_1(x_0)$, we rescale it and let $\eta_L(x_1, x_2):= L^{-3} \eta(\frac{x_1}{L}, \frac{x_2}{L^2})$ --we then test the equation with $\eta_{L}(x - x_0)$:
\begin{align*}
& (L^{-2}+1)\|w_> - l\|_{B_L(x_0)} \\
& \gtrsim \int_{B_L(x_0)} (- \partial_2 \eta_L(x-x_0)  + a(x_0) \partial_1^2 \eta_L(x-x_0) ) (w_{>}(x) - l(x)) + (w_{>} (x)- l(x)) \, \textrm{d} x\\
& = \int_{B_L(x_0)} \eta_{L}(x - x_0) ( c_1 (x-x_0) + c_0)_1 = c_0,
\end{align*}
where we have used the radial symmetry of $\eta$ and that it is $L^1$-normalized. Similarly, taking another $\partial_1$-derivative of the equation, yields $|c_1|\lesssim (L^{-3}+L^{-1})\|w_> - l\|_{B_L(x_0)}$. Together, this implies \eqref{jc01} in virtue of $L\le 1$. 

Furthermore, by setting $\phi(x_2):=e^{(x-x_0)_2}$, we can write \eqref{jc02} as
\begin{align*}
 (\partial_2 - a(x_0) \partial_1^2 ) \phi w_> = 0 \quad \quad \quad \textrm{ in } \quad B_L(x_0).
\end{align*}
Using $L\le 1$ and \eqref{jc01}, we have for any $l(x)=c_0+c_1(x-x_0)_1$
\begin{align*}
 \| \phi w_>\|_{B_{L/2}(x_0)} & \lesssim \| w_>\|_{B_{L/2}(x_0)} \\
 &\lesssim \| w_> - l\|_{B_{L/2}(x_0)} + |c_0| + L |c_1| \lesssim L^{-2}\| w_> - l\|_{B_{L}(x_0)},
\end{align*}
so that by virtue of $\partial_2 (\phi w_>) = \phi (\partial_2 +1)w_>$ we obtain
\begin{align*}
 \| \{  \partial_1^2 , \partial_2 \} w_>\|_{B_{L/2}(x_0)} & \lesssim \| \phi \{  \partial_1^2 , \partial_2 \}  w_>\|_{B_{L/2}(x_0)} \\
 &\lesssim \| \{  \partial_1^2 , \partial_2 \}  \phi w_>\|_{B_{L/2}(x_0)} + \| \phi w_>\|_{B_{L/2}(x_0)} \\
 &\lesssim \| \{  \partial_1^2 , \partial_2 \}  \phi w_>\|_{B_{L/2}(x_0)} + L^{-2}\| w_> - l\|_{B_{L}(x_0)}.
\end{align*}
Estimating $\| \{  \partial_1^2 , \partial_2 \} w_>\|_{B_{L/2}(x_0)}$ by the right-hand side in \eqref{jc03} is hence reduced to estimating $\| \{  \partial_1^2 , \partial_2 \} \phi w_>\|_{B_{L/2}(x_0)}$.
Since $\phi$ is Lipschitz in $[x_{02}-1,x_{02}+1]$, we have $\|\phi-1\|_{B_L(x_0)}=\|\phi-\phi(x_{02})\|_{B_L(x_0)}\lesssim \sup_{x\in B_L(x_0)} |x_2-x_{02}|= L^2$. Hence, we obtain from Bernstein's argument as above for any $l(x)=c_0+c_1(x-x_0)_1$ 
\begin{align}
\begin{split} 
 \| \{  \partial_1^2 , \partial_2 \} \phi w_>\|_{B_{L/2}(x_0)} & \lesssim L^{-2} \| \phi w_> - l\|_{B_L(x_0)} \\
 & \lesssim L^{-2} \|\phi (w_> - l)\|_{B_L(x_0)} + L^{-2}\|(\phi-1)l\|_{B_L(x_0)} \\
 &\lesssim L^{-2} \|w_> - l\|_{B_L(x_0)} + |c_0| + L|c_1| \lesssim L^{-2} \|w_> - l\|_{B_L(x_0)},
\end{split}
\end{align}
where we have used again \eqref{jc01} in the last step. This proves \eqref{jc03}.

With $l_R = w_>(x_0) + \frac{\partial w_>}{\partial x_1}(x_0) (x-x_0)_1$, using the same arguments as in the case $L\ge 1$ yields
\begin{align*}
\begin{split}
 &\| u_T - \sigma_i(x_0) v_{iT}(\cdot, a(x_0))  - l_R \|_{B_R(x_0)} \\
 \lesssim & \left( \frac{R}{L}\right)^2 \|  u_T - \sigma_i(x_0) v_{iT}(\cdot, a(x_0)) - l \|_{B_L(x_0)} + 2 L^2  \| g^T_{x_0} \|_{B_L(x_0)},
\end{split}
\end{align*}
so that it suffices to use \eqref{Lemma5_10}.\\

\noindent \textbf{Step 6:} \textit{(An equivalent definition of the modelling constant)} \quad  In this step we observe that $M \sim M^{\prime}$, where $M^{\prime}$ is defined as
 \begin{align}
 \label{Lemma5_19}
 M^{\prime} := \displaystyle\sup_{x_0 \in \mathbb{R}^2}\displaystyle\sup_{R >0} R^{-2 \alpha} \displaystyle\inf_{l \in \textrm{Span}\{1, x_1 \}} \| u -  \sigma_i(x_0) v_i(\cdot, a(x_0)) - l  \|_{B_R(x_0)}.
 \end{align}
Since the argument for this observation is an easy modification of that in \cite[Step 4 of Lemma 3.6]{OW}, we do not repeat it here.\\

\noindent \textbf{Step 7:} \textit{(Use of the modelling)} \quad In this step we remark that for $T \in (0, 1] $, $L>0$, and $x_0 \in \mathbb{R}^2$ the estimate 
\begin{align}
\label{Lemma5_22}
  \frac{1}{(T^{\frac{1}{4}})^{2\alpha}}  \| u_T - u - \sigma_i(x_0) (v_{iT} - v_i)(\cdot, a(x_0)) \|_{B_L(x_0)} \lesssim M + \tilde{N}\Big( \frac{L}{T^{\frac{1}{4}}}\Big)^{\alpha}
\end{align}
holds. Since we have access to Lemma \ref{small_v} and the moment bound \eqref{moment_bound}, the argument does not change from \cite[Step 5 of Lemma 3.6]{OW} and we, therefore, do not give it here.\\

\noindent \textbf{Step 8:} \textit{(Conclusion)} \quad We now show that $M \lesssim \tilde{N}$.
Combining \eqref{Lemma5_15} with \eqref{Lemma5_22} we find that 
\begin{align}
 \label{Lemma5_23}
\begin{split}
 &\frac{1}{R^{2\alpha}} \displaystyle\inf_{l \in \textrm{Span}\{1, x_1 \}} \| u - \sigma_i(x_0) v_{i}(\cdot, a(x_0)) - l \|_{B_R(x_0)} \\
    \lesssim & \Big( \frac{R}{L}\Big)^{2(1 - \alpha)}  \frac{1}{L^{2\alpha}} \displaystyle\inf_{l \in \textrm{Span}\{1, x_1 \}} \| u - \sigma_i(x_0) v_{i}(\cdot, a(x_0)) - l  \|_{B_L(x_0)}\\
 & \qquad +  \tilde{N}\Big(\frac{L^2}{R^{2\alpha} (T^{\frac{1}{4}})^{2-2\alpha} } + \frac{L^{2+\alpha}}{R^{2\alpha} (T^{\frac{1}{4}})^{2-\alpha}}  \Big) + \Big(\frac{T^{\frac{1}{4}}}{R}\Big)^{2 \alpha} \Big( M + \tilde{N} \Big( \frac{L}{T^{\frac{1}{4}}}\Big)^{\alpha} \Big). 
\end{split}
\end{align}

Let $\varepsilon \ll1$. For the case $R \leq \varepsilon^{-1}$ we make use of \eqref{Lemma5_23} and let $L = \varepsilon^{-1} R$ and $T^{\frac{1}{4}} = \varepsilon R$; the restriction on $R$ guarantees that $T \leq 1$. Making these identifications and using the definition of $M$, we obtain
\begin{align}
\label{R_less_than_one}
\begin{split}
& \frac{1}{R^{2\alpha}}  \displaystyle\inf_{l \in \textrm{Span}\{1, x_1 \}} \| u - \sigma_i(x_0) v_{i}(\cdot, a(x_0)) - l \|_{B_R(x_0)}\\
   \lesssim & (\varepsilon^{2 - 2\alpha} + \varepsilon^{2 \alpha} ) M  + (\varepsilon^{-(4-2\alpha)}  + \varepsilon^{-4}+1)\tilde{N}.
\end{split}
\end{align} 
Observe that for $R>\varepsilon^{-1}$ we have by \eqref{new_linfty_bound} that
\begin{align}
\label{R_bigger_than_one}
\begin{split}
&  \frac{1}{R^{2\alpha}}  \displaystyle\inf_{l \in \textrm{Span}\{1, x_1 \}} \| u - \sigma_i(x_0) v_{i}(\cdot, a(x_0)) - l \|_{B_R(x_0)} \\
 \le & \varepsilon^{2\alpha} (\|u\| + \|\sigma_i\| \|v_i\|) \lesssim \varepsilon^{2\alpha} (K + \|\sigma_i\| N_i + M)
\end{split}
\end{align}
where we have also used Lemma \ref{small_v}. 
Combining \eqref{R_less_than_one} and \eqref{R_bigger_than_one} we find that 
\begin{align}
\label{all_R}
\begin{split}
& \sup_{R >0} \frac{1}{R^{2\alpha}} \displaystyle\inf_{l \in \textrm{Span}\{1, x_1 \}} \| u - \sigma_i(x_0) v_{i}(\cdot, a(x_0)) - l \|_{B_R(x_0)}\\
    \leq & (\varepsilon^{2 - 2\alpha} + \varepsilon^{2 \alpha} ) M  + (  \varepsilon^{ 2\alpha} + \varepsilon^{-(4-2\alpha)}  + \varepsilon^{-4} + 1) \tilde{N} . 
\end{split}
\end{align} 
Using $M \sim M^{\prime}$ and choosing $\varepsilon$ small enough yields $M  \lesssim \tilde{N}$.

After plugging in $\tilde{N}$ from \eqref{Lemma5_10} of Step 2 this gives
\begin{align*}
M \lesssim K + [a]_{\alpha}  M + \|\sigma_i\|_{\alpha} N_i.
\end{align*}
Using \eqref{local_u_norm_bound} and $[a]_{\alpha} \ll 1$, we then find that 
\begin{align*}
M  \lesssim K + \| \sigma_i\|_{\alpha} N_i .
\end{align*}
Finally, the bound on $\| u\|_{\alpha}$ follows now from \eqref{local_u_norm_bound} and \eqref{new_linfty_bound}.
\end{proof}

\section{Argument for Theorem 
\ref{linear_theorem}: Treatment of the linear problem}
\label{section:treat_linear}
In this section we prove Propositions \ref{linear_forcing}, \ref{ansatz_IVP}, and \ref{linear_IVP} and Theorem \ref{linear_theorem}.
 
\subsection{Proof of Proposition \ref{linear_forcing}: Treatment of linear space-time periodic problem with irregular forcing}
\label{section:treat_linear_2}

The main difference between the proof we present below and the proof of \cite[Proposition 3.8]{OW} is our use of the modified Lemma \ref{KrylovSafonov}. In this proof it is essential that \eqref{prop1_5} has a massive term since $a$ is not periodic in time --if  \eqref{prop1_5} were posed without a massive term, there may not be solutions even when $f$ is of class $C^{\alpha}$.

\begin{proof} The argument has eight steps, of which the first four correspond to $i)$ and the further four to $ii)$.\\

\noindent  $i)$ \textbf{Step 1:} \textit{(Regularized reference products)} \quad Throughout this step we adopt the conditions and notations of Lemma \ref{lemma:reconstruct_2}. For any $\tau>0$, we use the convention that $V_{i \tau}(\cdot, a_0) = V_i(\cdot, a_0) \ast \psi_{\tau}$  and \textit{define}
\begin{align}
\label{prop1_8}
F \di \partial_1^2 V_{i \tau}(\cdot, a_0) := (F \di \partial_1^2 V_i (\cdot, a_0))_{\tau}.
\end{align}
These new offline products are taken as input for Lemma \ref{lemma:reconstruct_2} to obtain, for $u \in C^{\alpha}(\mathbb{R}^2)$ modelled after $V_{i\tau}$ according to $a_i$ and $\sigma_i$, the singular product $F \diamond \partial_1^2 u \in C^{\alpha -2} (\mathbb{R}^2)$. 

To apply Lemma \ref{lemma:reconstruct_2} we must check that the relations listed in $\eqref{Lemma4_2}$ hold. For this, we remark that the $L^1$-normalization of $\psi_T$ gives that $\left[V_{i\tau}\right]_{\alpha,1} \lesssim N_i$ and \eqref{monotonicity} yields
 \begin{align*}
\begin{split}
&\| [F, (\cdot)] \di \partial_1^2 V_{i\tau}\|_{2\alpha -2,1}\\
& =    \displaystyle\sup_{T \leq 1} (T^{\frac{1}{4}})^{2- 2 \alpha} \| [F, (\cdot)_{T+\tau} ] \di \partial_1^2 V_i \|_1\\
& =  \displaystyle\sup_{T \leq 1} (T^{\frac{1}{4}})^{2- 2 \alpha} ( \|(F(\partial_1^2 V_i)_T - (F \diamond \partial_1^2 V_i)_T)_{\tau} \|_1 + \| F \partial_1^2V_{i T+\tau} - (F \partial_1^2 V_{i T})_{ \tau} \|_1 )\\
& \lesssim   N N_i  + \displaystyle\sup_{T \leq 1} (T^{\frac{1}{4}})^{2- 2 \alpha} \| F \partial_1^2V_{i T+\tau} - (F \partial_1^2 V_{i T})_{ \tau} \|_1 . 
\end{split}
\end{align*}
To treat the second term, we assume that $\tau \leq T$ (the general case follows from switching the roles of $\tau$ and $T$) and use \eqref{alpha_kernel_bound}, to write
\begin{align}
\label{prop1_9.5.2}
\begin{split}
  \|  F \partial_1^2V_{i T+\tau} - (F \partial_1^2 V_{i T} )_{ \tau}  \|_1
\lesssim &  \left[F \right]_{\alpha}  \| \partial_1^2 V_{iT}  \|_1 \Big\| \int_{\mathbb{R}^2} |\psi_{ \tau} (\cdot-y)| d^{\alpha}(\cdot,y) \, \textrm{d}y \Big\| \\
\lesssim & \left[F \right]_{\alpha}  \left[ V_i \right]_{\alpha,1} (T^{\frac{1}{4}})^{\alpha -2} (\tau^{\frac{1}{4}})^{\alpha}\\
\lesssim & \left[F \right]_{\alpha}  N_i (T^{\frac{1}{4}})^{2\alpha -2}.
\end{split}
\end{align}
Combining these estimates, we find that 
\begin{align}
\label{prop_1_step_1_result}
\| [F, (\cdot)_T] \di \partial_1^2 V_{i \tau} \|_{2\alpha -2,1} \lesssim( [F]_{\alpha} +N) N_i.
 \end{align}

Having verified the assumptions of Lemma \ref{lemma:reconstruct_2}, we then characterize the distribution $F \diamond \partial_1^2 u$ under the assumption that $ \partial_1^2u \in C^{\alpha}(\mathbb{R}^2)$ --recall from under \eqref{prop1_8} that $u$ is modelled after $V_{i\tau}$ according to $a_i$ and $\sigma_i$. Notice that, as already used above, by \eqref{semigroup} and \eqref{prop1_8}, we have that  
 \begin{align*}
 \left[ F, (\cdot)_T\right] \di \partial_1^2 V_{i \tau}(\cdot, a_0) = \left[ F, (\cdot)_{T+\tau}\right] \di \partial_1^2 V_{i}(\cdot, a_0).
 \end{align*}
This means that as $T \rightarrow 0$, $\left[ F, (\cdot)_T\right] \di \partial_1^2 V_{i \tau}(\cdot, a_0)  \rightarrow  \left[ F, (\cdot)_{\tau}\right] \di \partial_1^2 V_{i }(\cdot, a_0)$ uniformly in $x$ for all $a_0 \in [\lambda, 1]$; whereby \eqref{Lemma4_2} implies that this convergence is uniform in $(x, a_0)$. By \eqref{Lemma4_3}, we then find that the condition $\partial_1^2 u \in C^{\alpha}(\mathbb{R}^2)$ gives that 
 \begin{align}
 \label{edit_2019_3}
\displaystyle\lim_{T \rightarrow 0} \| F \partial_1^2 u - ( F \di \partial_1^2 u)_T  - \sigma_i E\left[ F, (\cdot)_{\tau}\right] \di \partial_1^2 V_{i}(\cdot, a_0)\| =0,
\end{align}
where we have used the notation from Lemma \ref{lemma:reconstruct_2} that $E$ denotes evaluation of a function of $(x, a_0)$ at $(x, a(x))$. By the uniqueness in Lemma \ref{lemma:reconstruct_2}, we obtain
\begin{align}
\label{prop1_10}
F \di \partial_1^2 u = F \partial_1^2 u  -  \sigma_i E \left[ F, (\cdot)_{\tau} \right] \di \partial_1^2 V_{i}(\cdot, a_0).
\end{align}

\vspace{.2cm}

\noindent \textbf{Step 2:} \textit{(Analysis of the regularized problem)} \quad  Let $\tau \in (0,1)$. We show that there exists $\upa^{\tau}\in C^{\alpha +2}(\mathbb{R}^2)$, modelled after $\vp_{\tau}(\cdot, a_0)$ according to $a$, that solves
\begin{align}
\label{prop1_12}
(\partial_2-a \diamond \partial_1^2 +1) \upa^{\tau} & =  f_{\tau} && \textrm{ in } \quad \mathbb{R}^2
\end{align}
distributionally. 

Notice that by the previous step applied with $F =  a$, $I =1$, $V_{1}(\cdot, a_0) = v(\cdot, a_0)$, and $\sigma_1 = 1$, the formulation \eqref{prop1_12} is equivalent to 
\begin{align}
\label{prop1_12.01}
(\partial_2-a \partial_1^2 +1) u^{\tau}& =  f_{\tau} -E\left[a, (\cdot)_{\tau} \right] \di \partial_1^2 v  && \textrm{ in } \quad \mathbb{R}^2.
\end{align}
If $f_{\tau} -E\left[a, (\cdot)_{\tau} \right] \di \partial_1^2 v \in C^{\alpha}(\mathbb{R}^2)$, then the existence of $u^{\tau} \in C^{\alpha +2}(\mathbb{R}^2)$ solving \eqref{prop1_12.01} follows from \cite[Theorem 8.7.3]{Kr}. The desired modelling follows trivially due to the high regularity of $u^{\tau}$.

To see that $f_{\tau}-  E[a, (\cdot)_{\tau}] \diamond \partial_1^2 v \in C^{\alpha}(\mathbb{R}^2)$, we first notice that $g_{\tau} \in C^{\alpha}(\R^2)$ whenever $g \in L^{\infty}(\R^2)$. In particular, in that case, for any points $x,z \in \R^2$ such that $d(x,z) \leq1 $, we can write
\begin{align}
\begin{split}
|g_{\tau}(x) - g_{\tau}(z)| &= \int_{\R^2} |g(y)| \, | \psi_{\tau}(x-y) - \psi_{\tau}(z-y)  | \textrm{d} y\\
& \leq \| g \| \int_{\R^2}   |\partial_1 \psi_{\tau}(x-y) (x_1-z_1) -\partial_2 \psi_{\tau}(x-y) (x_2-z_2)| \textrm{d} y\\
& \lesssim \|g\| (\tau^{\frac{1}{4}})^{-1} d(x,z) \leq \|g\|  (\tau^{\frac{1}{4}})^{-1} d^{\alpha}(x,z).
\end{split}
\end{align}
Then, notice that \hyperlink{B1}{(B1)} implies $f_{\tau/2} \in L^{\infty}(\R^2)$. To treat the term $E\left[a, (\cdot)_{\tau} \right] \di \partial_1^2 v $, we remark that
\begin{align}
\| E ( a \diamond  \partial_1^2 v )_{\tau} \|_{\alpha} \leq \| (a \diamond  \partial_1^2 v)_{\tau} \|_{\alpha} + \|  (a \diamond  \partial_1^2 v)_{\tau}  \|_{1}  \| a \|_{\alpha} < \infty.
\end{align}
For the second inequality $``< \infty"$ above we have used the previous argument for $g=(a \diamond  \partial_1^2 v)_{\tau/2}$, which is admissible by \hyperlink{B4}{(B4)}, and \hyperlink{B2}{(B2)}. The observation that $\| E a \partial_1^2 v_{\tau} \|_{\alpha} < \infty$ is a trivial consequence of Lemma \ref{small_v} and \hyperlink{B2}{(B2)}.\\

\noindent \textbf{Step 3:} \textit{(Passing to the limit in the regularization)} \quad  Let $\tau \in (0,1)$. We apply Lemma \ref{KrylovSafonov} to $u^{\tau}$ from the previous step with $I =1$, $f_1 = f_{\tau}$, $a = a$, and $\sigma_1 = 1$. We first check that \eqref{Lemma5_3} holds by convolving \eqref{prop1_12} with $\psi_T$:
\begin{align*}
 (\partial_2  - a  \partial_1^2 +1) u^{\tau}_T -    f_{\tau + T}
 =[a , (\cdot)_T] \diamond \partial_1^2 u^{\tau}  \quad \quad \quad \textrm{in} \quad \mathbb{R}^2.
\end{align*}
By \eqref{monotonicity}, \eqref{Lemma4_4}, \eqref{prop1_3}, and \eqref{prop_1_step_1_result} we have that 
\begin{align*}
\| [a , (\cdot)] \diamond \partial_1^2 u^{\tau} \|_{2\alpha -2} \lesssim [a]_{\alpha} M_{\tau} + N_0 ( [a]_{\alpha} + N),
\end{align*}
where $M_{\tau}$ refers to the modelling of $u^{\tau}$ after $v_{\tau}(\cdot, a_0)$.  Applying Lemma \ref{KrylovSafonov} and using that $N\leq1$ we find that
\begin{align}
\label{prop_1_unifrom_bound_1}
M_{\tau} + \|u^{\tau} \|_{\alpha}  & \lesssim N_0,
\end{align}
where we have also used that $[a] \ll 1$.

By \eqref{prop_1_unifrom_bound_1} we know that, up to a subsequence, $u^{\tau} \rightarrow u$ uniformly as $\tau \rightarrow 0$, where we define the desired solution $u$ of \eqref{prop1_5} as this limit.  We must still pass to the limit in \eqref{prop1_12} and show that we recover \eqref{prop1_5}. The limits $f_{\tau} \rightharpoonup f$ and $\partial_2 u^{\tau} \rightharpoonup \partial_2 u$ are clear. It remains to check that $a \diamond \partial_1^2 u^{\tau} \rightharpoonup a \diamond \partial_1^2 u$, where the limiting modelling is a result of Definition \ref{model} in tandem with the uniform in $(x, a_0)$ convergence $v_{\tau}(\cdot, a_0) \rightarrow v(\cdot, a_0)$ and $u^{\tau} \rightarrow u$. This convergence can be deduced using the condition \eqref{Lemma4_3} from Lemma \ref{lemma:reconstruct_2}; the full argument, which sees no alteration in the passage to our setting, can be found in \cite[Proof of Proposition 3.8, Steps 9 and 10]{OW}.\\

\noindent \textbf{Step 4:} \textit{(Uniqueness)} \quad  Assume that there are two solutions $u$ and $u^{\prime}$ satisfying \eqref{prop1_5} with the desired modelling. Subtracting the two solutions we find that the difference $u - u^{\prime}$ is now trivially modelled. Using  \eqref{Lemma4_3} from Lemma \ref{lemma:reconstruct_2}, the triangle inequality yields
\begin{align}
\lim_{T \rightarrow 0} \| (a \diamond \partial_1^2 u)_T - (a \diamond \partial_1^2 u^{\prime})_T  - (a \diamond \partial_1^2 (u - u^{\prime}))_T \| =0,
\end{align}
which implies that
\begin{align}
\label{linear_product}
a \diamond \partial_1^2 u - a \diamond \partial_1^2 u^{\prime} = a \diamond \partial_1^2 (u - u^{\prime}).
\end{align}
So, the difference $u - u^{\prime}$ solves 
\begin{align}
\label{equation_for_difference}
(\partial_2 - a \diamond \partial_1^2 +1) (u - u^{\prime})_T = [a, (\cdot)_T] \diamond \partial_1^2 (u - u^{\prime}) &&  \textrm{in} \quad \mathbb{R}^2.
\end{align}
Moreover, by \eqref{Lemma4_4} of Lemma \ref{lemma:reconstruct_2}  we have that 
\begin{align*}
\| [a, (\cdot)] \diamond \partial_1^2 (u - u^{\prime})\|_{2\alpha -2} \lesssim [a]_{\alpha} [u - u^{\prime}]_{2\alpha}.
\end{align*}
Therefore, Lemma \ref{KrylovSafonov} applied with $I=1$, $f_1(\cdot, a_0) = 0$, $\sigma_1 = 0$, and $a = a$ gives:
\begin{align}
[u - u^{\prime}]_{2\alpha} + \| u - u^{\prime}\|_{\alpha} \lesssim [a]_{\alpha}[u - u^{\prime}]_{2\alpha},
\end{align}
which, since $[a]_{\alpha} \ll 1$, means that $[u - u^{\prime}]_{2\alpha} +  \| u - u^{\prime}\|_{\alpha} = 0$.\\

\noindent ii) \textbf{Step 5:} \textit{(Interpolation of the data)} We linearly interpolate the $a_i$ and $f_i$. In particular,  for $s \in [0,1]$ we define
\begin{align}
\label{interpolate_defn}
a_s :=  (1-s) a_0 + s a_1  \quad \textrm{and} \quad  f_s := (1-s) f_0  + s f_1.
\end{align} 
Of course, then $v_s(\cdot, a_0)$ defined as 
\begin{align*}
v_{s}(\cdot, a_0) := (1-s) v_0(\cdot, a_0)  +  s v_1(\cdot, a_0) 
\end{align*}
solves \eqref{periodic_mean_free} with right-hand side $f_s$. In order to keep notation lean, in this section we occasionally suppress the dependence of $v_s$ on the parameter $a_0$. To make sure that Leibniz' rule 
\begin{align}
\label{leibniz_prop_1}
\partial_s ( a_s \diamond \partial_1^2 v_s) = \partial_s  a_s \diamond \partial_1^2 v_s +  a_s \diamond  \partial_1^2 \partial_s v_s,
\end{align}
holds, the offline products are interpolated bilinearly as
\begin{align*}
  a_s \diamond \partial_1^2 v_ s
 := (s-1)^2  a_0 \diamond \partial_1^2 v_0 + s(1-s) (a_0 \diamond \partial_1^2 v_1+ a_1 \diamond \partial_1^2 v_0)  + s^2 a_1 \diamond \partial_1^2 v_1 .
\end{align*}
We, furthermore, define 
\begin{align}
\label{products_for_interpolation}
\begin{split}
  a_s \diamond \partial_1^2  \partial_s  v_s &  :=   (1-s) a_0 \diamond \partial_1^2  v_1+ s a_1 \diamond \partial_1^2  v_1 - (1-s) a_0  \diamond \partial_1^2 v_0  - s a_1  \diamond \partial_1^2 v_0 ,\\
 \partial_s a_s \diamond \partial_1^2  v_s &  :=  (1-s)  a_1 \diamond \partial_1^2  v_0(\cdot, a_0)  +  s   a_1 \diamond \partial_1^2 v_1(\cdot, a_0) \\
 & \quad \quad  -(1-s)  a_0 \diamond \partial_1^2  v_0(\cdot, a_0)  -  s   a_0 \diamond \partial_1^2 v_1(\cdot, a_0) ,\\
 \textrm{and} \, \,a_s \diamond \partial_1^2 \partial_{a_0} v_s &  := \partial_{a_0} (a_s \diamond \partial_1^2 v_s),
\end{split}
\end{align}
where we have used \eqref{prop1_3} to ensure that the right-hand side of the last definition is well-defined.

We remark that the assumptions \hyperlink{B4}{(B4)} and \hyperlink{C4}{(C4)} guarantee that
\begin{align}
 \| [a_s, (\cdot)] \diamond \partial_1^2 \partial_s  v_s  \|_{2\alpha -2,1} & \lesssim N \delta N_0,  \label{prop_1_part_2_0}\\
\| [\partial_s a_s, (\cdot)_T] \diamond \partial_1^2  v_s  \|_{2\alpha-2,1} & \lesssim \delta N N_0.  \label{prop_1_part_2_0.75}\\
 \textrm{and} \quad \| [a_s, (\cdot)_T] \diamond \partial_1^2 \partial_{a_0}  v_s  \|_{2\alpha -2,1} & \lesssim N N_0. \label{prop_1_part_2_0.5}
\end{align}
We can regularize all of the new offline products as in Step 1;  \textit{e.\,g.} we set
\begin{align*}
a_s \diamond \partial_1^2 v_{s \tau} := (a_s \diamond \partial_1^2  v_s )_{\tau}.
\end{align*}

\vspace{.2cm}

\noindent \textbf{Step 6:} \textit{(A continuous curve of solutions $u_s^{\tau}$ and an equation for $\partial_s u_s^{\tau}$)} \quad By the previous steps, for every $\tau \in (0,1]$, there exists a curve of $C^{\alpha +2}$ solutions $u^{\tau}_{s}$ of
\begin{align}
\label{prop_1_part_2_1}
(\partial_2 - a_s \partial_1^2 +1) u^{\tau}_s & =f_{s \tau} - E_s[a_s, (\cdot)_{\tau}] \diamond \partial_1^2 v_s && \textrm{in} \quad \mathbb{R}^2,
\end{align}
where $E_s$ denotes evaluation of a function of $(x, a_0)$ at $(x, a_s(x))$ and $f_{s\tau} = f_s \ast \psi_{\tau}$. The solution $u^{\tau}_{s}$ is modelled after $v_{s \tau}$, where $v_{s \tau} = v_s \ast \psi_{\tau}$, according to $a_s$, which by Step 1 gives
\begin{align}
a_s \diamond \partial_1^2  u^{\tau}_s = a_s \partial_1^2 u^{\tau}_s - E_s[a_s, (\cdot)_{\tau}] \diamond \partial_1^2 v_s.
\end{align}
This allows us to rewrite \eqref{prop_1_part_2_1} as
\begin{align}
\label{prop_1_part_2_2}
(\partial_2 - a_s \diamond \partial_1^2 +1) u^{\tau}_s & =f_{s \tau} && \textrm{in} \quad \mathbb{R}^2. 
\end{align}

To obtain an equation for $\partial_s u^{\tau}_s$ we differentiate \eqref{prop_1_part_2_1}. For this we use \eqref{leibniz_prop_1}, which gives the relation
\begin{align}
\label{prop_1_part_2_3}
\begin{split}
 & \partial_s (E_s [a_s, (\cdot)_{\tau}] \diamond \partial_1^2 v_s )\\
  = & E_s [\partial_s a_s, (\cdot)_{\tau}] \diamond \partial_1^2 v_s + E_s [a_s, (\cdot)_{\tau}] \diamond \partial_1^2  \partial_s v_s
  +   \partial_s a_s E_s  [a_s, (\cdot)_{\tau}] \diamond \partial_1^2  \partial_{a_0}  v_s.
\end{split}
\end{align}
So, $\partial_s u_s^{\tau}$ solves
\begin{align*}
\begin{split}
& (\partial_2 - a_s \partial_1^2 +1 ) \partial_s u_s^{\tau}   - \Big(  \partial_s f_{s \tau} + \partial_s a_s \partial_1^2  u_s^{\tau} - E_s [\partial_s a_s, (\cdot)_{\tau}] \diamond \partial_1^2 v_s    \\
& \qquad - E_s [a_s, (\cdot)_{\tau}] \diamond \partial_1^2  \partial_s v_s -   \partial_s a_s E_s  [a_s, (\cdot)_{\tau}] \diamond \partial_1^2  \partial_{a_0}  v_s   \Big) = 0 \quad \quad \quad \textrm{in} \quad \mathbb{R}^2.
\end{split}
\end{align*}
Since the term in parentheses is in $C^{\alpha}(\mathbb{R}^2)$, which can be checked using the same tools as in Step 2, we find that $\partial_s u^{\tau}_s \in C^{\alpha +2}(\mathbb{R}^2)$. Due to the high regularity of $\partial_s u^{\tau}_s$, we know that it is modelled after $( \partial_s v_{s\tau} , \partial_{a_0} v_{s \tau})$ according to $a_s$ and $(1, \partial_s a_s )$. Using the identities
\begin{align}
\label{prop_1_part_2_6}
\partial_s a_s \partial_1^2  u_s^{\tau} - E_s [\partial_s a_s, (\cdot)_{\tau}] \diamond \partial_1^2 v_s = \partial_s a_s \diamond \partial_1^2  u_s^{\tau}
\end{align}
and 
\begin{align}
\label{prop_1_part_2_6.5}
\begin{split}
 a_s \partial_1^2 \partial_s  u_s^{\tau} - E_s [ a_s, (\cdot)_{\tau}] \diamond \partial_1^2 \partial_s v_s  - \partial_s a_s E_s [ a_s, (\cdot)_{\tau}] \diamond \partial_1^2 \partial_{a_0} v_s
  = a_s \diamond \partial_1^2  \partial_s  u_s^{\tau},
\end{split}
\end{align}
which both follow from Step 1, we can rewrite the equation for $\partial_s u_s^{\tau}$ as
\begin{align}
\label{prop_1_part_2_4.5}
(\partial_2 - a_s \diamond \partial_1^2  +1) \partial_s u_s^{\tau} & = \partial_s f_{s \tau}  + \partial_s a_s \diamond \partial_1^2 u_s^{\tau}  && \textrm{in} \quad  \mathbb{R}^2.
\end{align}

\vspace{.2cm}

\noindent \textbf{Step 7:} \textit{(Estimates for $\partial_s u^{\tau}_s$)} \quad We now apply Lemma \ref{KrylovSafonov} to $\partial_s u^{\tau}_s$ with $I = 2$ and $f_1(\cdot, a_0) = \partial_s f_{s \tau}$, $\sigma_1 = 1$,  $f_2(\cdot, a_0) = \partial_1^2 v_{s\tau}(\cdot, a_0)$, $\sigma_2 = \partial_s a_s$, and $a = a_s$.  Notice that by \eqref{monotonicity} and assumption \hyperlink{C1}{(C1)} we have that
\begin{align}
\label{prop_1_part_2_5}
\| \partial_s f_{s\tau} \|_{\alpha -2} \lesssim\|f_{0}-f_{1 }\|_{\alpha -2} \lesssim \delta N_0;
\end{align}
using additionally \eqref{alpha_kernel_bound} and  Lemma \ref{small_v}, we obtain
 \begin{align}
 \label{prop_1_part_2_7}
 \|  \partial_1^2  v_{s\tau} \|_{\alpha-2,1}
  \lesssim  [f_{s}]_{\alpha -2} 
  \lesssim N_0.
 \end{align}
The relations \eqref{prop_1_part_2_5} and \eqref{prop_1_part_2_7} verify the assumption \eqref{lemma5_assumption_1}. 

We then check that $\partial_s u^{\tau}_s$ is an approximate solution in the sense of \eqref{Lemma5_3}. For this, we convolve \eqref{prop_1_part_2_4.5} with $ \psi_T$, which gives that 
\begin{align}
\begin{split}
\label{final_edit_1}
& \sup_{T\leq 1} (T^{\frac{1}{4}})^{2 - 2\alpha} \| (\partial_2 - a_s \diamond \partial_1^2  +1) \partial_s u_{sT}^{\tau} - \partial_s( f_{s \tau})_T - \partial_s a_s E_s  \partial_1^2 (v_{s \tau}(\cdot, a_0))_T \|\\
 \lesssim & \| [a_s, (\cdot)] \diamond \partial_1^2 \partial_s u_s^{\tau}   \|_{2\alpha -2} +\sup_{T\leq 1}  (T^{\frac{1}{4}})^{2 - 2\alpha} \|  ( \partial_s a_s \diamond \partial_1^2 u_s^{\tau})_T- \partial_s a_s E_s  \partial_1^2( v_{s \tau}(\cdot, a_0))_T \|.
\end{split}
\end{align}
By \eqref{Lemma4_4} of Lemma \ref{lemma:reconstruct_2} in conjunction with \eqref{prop_1_part_2_0} and \eqref{prop_1_part_2_0.5} the first term is bounded as
\begin{align}
\label{prop_1_part_2_11}
 \| [a_s, (\cdot)] \diamond \partial_1^2 \partial_s u_s^{\tau}   \|_{2\alpha -2}  \lesssim  [a_s]_{\alpha} \delta M_s^{\tau} +  N\delta N_0 + N_0 \delta N,
\end{align}
where $\delta M_s^{\tau}$ belongs to the modelling of $\partial_s u_s^{\tau}$ after $(\partial_s v_{s\tau}, \partial_{a_0} v_{s\tau})$ according to $a_s$ and $(1, \partial_s a_s)$. The second term of \eqref{final_edit_1} requires another application of the triangle inequality to write
\begin{align}
\begin{split}
\label{final_edit_2} 
&\sup_{T\leq 1}  (T^{\frac{1}{4}})^{2 - 2\alpha} \|  ( \partial_s a_s \diamond \partial_1^2 u_s^{\tau})_T- \partial_s a_s E_s  \partial_1^2 (v_{s \tau}(\cdot, a_0) )_T\|\\
 \leq  & \| [\partial_s a_s, (\cdot) ] \diamond \partial_1^2 u_s^{\tau}\|_{2\alpha -2}+  \sup_{T\leq 1}  (T^{\frac{1}{4}})^{2 - 2\alpha} \| \partial_s a_s \partial_1^2u_{sT}^{\tau} - \partial_s a_s E_s \partial_1^2 (v_{s \tau} (\cdot, a_0))_T\|.
\end{split}
\end{align}
The first term is bounded by $\delta N M^{\tau}_s + N N_0$  by \eqref{Lemma4_4} of Lemma \ref{lemma:reconstruct_2} and \eqref{prop_1_part_2_0.75}, where $M_s^{\tau}$ belongs to the modelling of $u_s^{\tau}$ after $v_{s\tau}$ according to $a_s$.  We complete our argument by using this modelling in conjunction with  $\psi_1$ being even in $x_1$ and a Schwartz function to obtain
\begin{align}
\label{final_edit_3} 
\begin{split}
|\partial_s a_s(x) \partial_1^2 u_{sT}^{\tau}(x) - \partial_s a_s (x) E_{s} \partial_1^2 v_{s \tau+T} (x, a_0)|
\leq \delta N M_s^{\tau},
\end{split}
\end{align}
where we have also used $\| a_1 - a_0\|_{\alpha} \leq \delta N$. Combining  \eqref{final_edit_1}, \eqref{prop_1_part_2_11}, \eqref{final_edit_2}, \eqref{final_edit_3}, the bound $M_s^{\tau} \lesssim N_0$ from part \textit{i)}, and that $N\leq1$, we find that for large enough $c \in \mathbb{R}$ we can set $K$ in \eqref{Lemma5_3} as
\begin{align}
K = c( [a_s]_{\alpha} \delta M_s^{\tau} + N_0 \delta N  +  \delta N_0).
\end{align}

Together with \eqref{prop_1_part_2_5} and  \eqref{prop_1_part_2_7}, an application of Lemma \ref{KrylovSafonov} to the $\partial_s u_s^{\tau}$ gives:
\begin{align}
\label{edit_edit_2}
\begin{split}
 \delta M_s^{\tau}  + \|\partial_s u^{\tau}_s\|_{\alpha}  \lesssim N_0 \delta N+  \delta N_0.
\end{split}
\end{align}

\vspace{.2cm}

\noindent \textbf{Step 8:} \textit{(Integration and passing to the limit)} \quad Since we have \eqref{edit_edit_2} for all $s \in [0,1]$, we may integrate it up to obtain
\begin{align}
\label{prop_1_part_2_13}
\| u^{\tau}_1 - u^{\tau}_0 \|_{\alpha} \lesssim  \int_0^1 \| \partial_s u^{\tau}_s \, \textrm{d} s \|_{\alpha} \lesssim N_0 \delta N  +  \delta N_0.
\end{align}
To obtain a bound for $\delta M^{\tau}$ we notice that 
\begin{align}
\label{prop_1_part_2_14}
\begin{split}
& \partial_s (u_s^{\tau}(y) - v_{s\tau} (y, a_s(x)))  = \partial_s u_s^{\tau}(y)  -  \partial_s v_{s\tau}(y, a_s(x)) - \partial_s a_s(x) \partial_{a_0} v_s(y, a_s(x)),
\end{split}
\end{align}
which allows us to integrate up our bound on $\delta M^{\tau}_s$ to obtain that $u_1^{\tau} - u_0^{\tau}$ is modelled after $(v_{1 \tau}, -v_{0\tau})$ according to $(a_1, a_0)$ with $\delta \nu^{\tau} = \int_0^1 \delta \nu^{\tau}_s \textrm{d} s$. Here, $\delta \nu^{\tau}_s$ is associated to the modelling of $\partial_s u_{s}^{\tau}$ already used in the previous step. We find that 
\begin{align}
\label{prop_1_part_2_15}
\delta M^{\tau} \lesssim \delta N N_0 +  \delta N_0.
\end{align}
Since we know from part \textit{i)} that $u^{\tau}_i \rightarrow u_i$ uniformly we can pass to the limit in \eqref{prop_1_part_2_13}.  In order to pass to the limit in the modelling we, furthermore, use that $v_{i\tau}(\cdot, a_i(\cdot)) \rightarrow v_i(\cdot, a_i(\cdot))$ and, by Step 1 of the proof of Lemma \ref{lemma:reconstruct_1}, also $\delta \nu^{\tau}  \rightarrow \delta \nu$ for some $\delta \nu$; both convergences are uniform in $x$.
\end{proof}

\subsection{Argument for Proposition \ref{ansatz_IVP}: Analysis of the ansatz for the ``initial boundary layer''}

\label{section:modelling_q}

In this section we investigate the modelling of $q$, which is defined in Definition \ref{q} --it culminates in the proof of Proposition \ref{ansatz_IVP}.

We start with two technical lemmas, which allow us to post-process the natural modelling of the ansatz $q$ after $\vi(\cdot, a_0, U_{int} - u)$ according to $\bar{a}$ --both of which will be proved in Section \ref{section:modelling_q_aux}. The first is used to exchange $\bar{a}$ in the natural modelling of $q$ for $a$ --using that $\bar{a} =a$ on $\partial \mathbb{R}^2_+$. Here is the statement:

\begin{lemma}[Post-processing of the Modelling]
\label{post_process}
Let $\alpha \in (0,1)$. We use the notation from Definition \ref{constant_soln}.

\begin{itemize}[leftmargin=.2in]
\item[i)] Assume that $a, a^{\prime} \in C^{\alpha}(\mathbb{R}^2)$ with $\| a \|_{\alpha}, \|a^{\prime}\|_{\alpha} \leq 1$ and $a^{\prime} = a$ on $\left\{ x_2 =0 \right\}$. Let $V_{int}(\cdot, a_0) \in C^{\alpha}(\mathbb{R})$ for $a_0 \in [\lambda,1]$ with $\lambda>0$.  For any points $x,y \in \mathbb{R}^2_+$ we then have the relation
\begin{align}
\label{Lemma_6_result_1}
\begin{split}
&\left| \vi(x, a(y), \vi_{int}(a(y))) - \vi(y, a(y), \vi_{int}(a(y)))  \right.\\
& \quad \quad \quad \quad \quad \quad \left. - \left(\vi(x, a^{\prime}(y), \vi_{int}(a^{\prime}(y))) - \vi(y, a^{\prime}(y), \vi_{int}(a^{\prime}(y)) ) \right) \right|\\
&\lesssim \| \vi_{int}\|_{\alpha,1} ([a]_{\alpha} + [a^{\prime}]_{\alpha}) d^{2\alpha}(x,y).
\end{split}
\end{align}
It follows that if  $\ui \in C^{\alpha}(\mathbb{R}^2_+)$ is modelled after $\vi(\cdot, a_0, \vi_{int}(a_0))$ according to $a$ with modelling constant $M$, then $\ui$ is modelled in the same way according to $a^{\prime}$. The new modelling constant $M^{\prime}$ satisfies 
\begin{align}
\label{new_M_2}
M^{\prime} \lesssim M + \| \vi_{int}\|_{\alpha,1}.
\end{align}

\item [ii)]Let $i =0,1$.  Assume that $a_i, a_i^{\prime} \in C^{\alpha}(\mathbb{R}^2)$ with $\| a_i \|_{\alpha}, \|a_i^{\prime}\|_{\alpha} \leq 1$ and $a_i^{\prime} = a_i$ on $\left\{ x_2 =0 \right\}$ and $V_{int,i}(\cdot, a_0) \in C^{\alpha}(\mathbb{R})$ for $a_0 \in [\lambda,1]$ with $\lambda>0$. We find that if $\ui \in C^{\alpha}(\mathbb{R}^2_+)$ is modelled after $(\vi(\cdot, a_0, \vi_{int,1}(a_0)), -\vi(\cdot, a_0, \vi_{int,0}(a_0)))$ according to $(a_1, a_0)$ with modelling constant $\delta M$, then $\ui$ is modelled in the same way according to $(a^{\prime}_1, a^{\prime}_0)$. The new modelling constant $\delta M^{\prime}$ satisfies
\begin{align}
\label{new_M}
\delta M^{\prime} \lesssim \delta M +   \|\vi_{int,1} - \vi_{int,0} \|_{\alpha,1} +\displaystyle\max_{i=0,1} \| \vi_{int,i}\|_{\alpha,2}( \| a_1 - a_0\|_{\alpha} + \| a_1^{\prime} - a_0^{\prime}\|_{\alpha}).
\end{align}
\end{itemize}
\end{lemma}

In our second technical lemma we show that the modelling of a function is preserved under the application of the heat semigroup. This, in particular, allows us to exchange $\vi(\cdot, a_0, U_{int} - u)$ for $\vi(\cdot, a_0, U_{int} - v(\cdot, a_0))$ in the natural modelling of $q$. Here is the statement:

\begin{lemma}[Propagation of the Modelling]
\label{propogation}
Let $\alpha \in (0,1)$. We use the notation from Definition \ref{constant_soln}.
\begin{itemize}[leftmargin=.2in]
\item[i)] Assume that $\ui\in C^{\alpha}(\mathbb{R}^2_+)$ is modelled after $\vi(\cdot, a_0, U - u)$ for $U,u \in C^{\alpha}(\mathbb{R})$ according to $a\in C^{\alpha}(\mathbb{R}^2)$ with $\|a\|_{\alpha}\leq 1$ and $a \in[\lambda, 1]$ for $\lambda>0$ with modelling constant $M$. If, furthermore, $u$ is modelled after $v$ according to $a$ on $\left\{ x_2 = 0\right\}$ with modelling constant $M_{\partial}$ and $\nu_{\partial} \in C^{2\alpha -1}(\mathbb{R})$,  then $\ui$ is modelled after $\vi(\cdot, a_0, U-v(a_0))$ according to $a$ with modelling constant $M^{\prime}$ bounded as 
\begin{align}
\label{new_modelling_constant}
M^{\prime} \lesssim M + M_{\partial} + \|\nu_{\partial}\|_{2\alpha-1} + \|U\|_{\alpha} + \|u \|_{\alpha} +  \| v\|_{\alpha,1}.
\end{align}

\item[ii)] Let $i = 0,1$. We will assume that $\ui \in C^{\alpha}(\mathbb{R}^2_+)$ is modelled after $(\vi(\cdot, a_0, U_0 - u_0),-\vi(\cdot, a_0, U_1 - u_1))$ according to $(a_0, a_1)$ for $U_i, u_i \in C^{\alpha}(\mathbb{R})$ and $a_i \in C^{\alpha}(\mathbb{R}^2)$ such that $\|a_i \|_{\alpha}\leq 1$ and $a_i \in [\lambda,1]$ for $\lambda>0$ with modelling constant $\delta M$. If, furthermore, $u_1 - u_0$ is modelled after $(v_1 , -v_0)$ according to $(a_1, a_0)$ on $\left\{ x_2 = 0\right\}$ with modelling constant $\delta M_{\partial}$ and $\delta  \nu_{\partial}\in C^{2\alpha -1}(\mathbb{R})$, then $\ui$ is modelled after $(\vi(\cdot, a_0, U_0 -v_0(a_0)),-\vi(\cdot, a_0, U_1 - v_1(a_0)))$ according to $(a_0, a_1)$ with modelling constant $\delta M^{\prime}$ bounded as 
\begin{align}
\begin{split}
\label{new_modelling_constant_2}
\delta M^{\prime} \lesssim & \delta M + \delta M_{\partial} + \|\delta \nu_{\partial}\|_{2\alpha -1}  + \| U_1 - U_0\|_{\alpha} + \| u_1 - u_0\|_{\alpha}\\
 & + \max_{i = 0,1} (\| U_i \|_{\alpha} + \| u_i \|_{\alpha} ) \|a_1 - a_0 \|_{\alpha} +  \|v_1 - v_0\|_{\alpha}.
\end{split}
 \end{align}
 \end{itemize}
\end{lemma}

We now complete this section by giving the proof of Proposition \ref{ansatz_IVP}:

\begin{proof} In this proof we drop the subscript $u$ on $V^{\prime}_{u}(\cdot, a_0)$ and $q_{u}$ and for $i =0,1$ write $V^{\prime}_i(\cdot, a_0)$ for $V^{\prime}_{u_i}(\cdot, a_0)$ and $q_i$ for $q_{u_i}$. Notice that this is a slight abuse of notation since we have defined the objects $V^{\prime}(\cdot, a_0)$ and $q$ already in \eqref{notation_inter} and Definition \ref{q} respectively. As already mentioned, the notation that we use in the current proposition only differs in the sense that it allows for more general $u$, whereas in the rest of this paper $u$ is always taken to be the solution of \eqref{prop1_5} in Proposition \ref{linear_forcing}. 

There are total of six steps, of which the first three correspond to $i)$ and the last three to $ii)$. Here is the argument:\\

\noindent $i)$ \textbf{Step 1:} \textit{(An intermediate modelling)} \quad We first show that $q$ is modelled after $\vi^{\prime}(\cdot, a_0)$ according to $a$ on $\mathbb{R}^2_+$ with modelling constant $M_{intermediate}$ bounded as 
\begin{align}
\label{intermediate_modelling_2.11}
M_{intermediate}  \lesssim  ( N_0^{int} + \| \upa\|_{\alpha}) [a]_{\alpha}.
\end{align}
To obtain this, take two points $x,y \in \mathbb{R}^2_+$ and write
\begin{align}
\label{intermediate_modelling_2.2}
\begin{split}
& |\vi^{\prime}(x, \bar{a}(x)) - \vi^{\prime}(y, \bar{a}(y)) - (\vi^{\prime}(x, a(y)) - \vi^{\prime}(y, a(y)) )|\\
 \leq & | \vi^{\prime}(x, \bar{a}(x)) - \vi^{\prime}(x, \bar{a}(y))| + |\vi^{\prime}(x, \bar{a}(y)) - \vi^{\prime}(y, \bar{a}(y))- (\vi^{\prime}(x, a(y)) - \vi^{\prime}(y, a(y)) )|.
\end{split}
\end{align}
By $v)$ of Lemma \ref{semigroup_bounds} applied to $\bar{a}$ and part $i)$ of Lemma \ref{post_process}, we bound the second term of \eqref{intermediate_modelling_2.2} by $( \| \uf_{int} \|_{\alpha} + \| \upa\|_{\alpha}) [a]_{\alpha}d^{2\alpha}(x,y)$. 
For the first term of \eqref{intermediate_modelling_2.2}, using Lemma \ref{semigroup_bounds} and \eqref{a_bound_match_up}, we find that
\begin{align}
\label{modelling_q_701}
\begin{split}
 | \vi^{\prime}(x, \bar{a}(x)) - \vi^{\prime}(x, \bar{a}(y))|
& \lesssim | \partial_{a_0} \vi^{\prime} (x , a_0)| \, | \bar{a}(y) - \bar{a}(x)|\\
 &  \lesssim  (\| \uf_{int}\|_{\alpha} + \|\upa\|_{\alpha}) [a]_{\alpha} x_2^{\frac{\alpha}{2}} \times  \begin{cases}
 (x_2^{-\frac{\alpha}{2}} +y_2^{-\frac{\alpha}{2}}  ) d^{2\alpha}(x,y),\\
 d^{\alpha}(x,y),\\
  \end{cases}
  \end{split}
\end{align}
where the notation on the right-hand side indicates that both bounds hold. We then post-process this as in part $i)$ of Lemma \ref{post_process}  and use  \hyperlink{B3}{(B3)}  to find that
\begin{align}
\label{modelling_2_new}
\begin{split}
 | \vi^{\prime}(x, \bar{a}(x)) - \vi^{\prime}(x, \bar{a}(y))| & \lesssim  (\|\uf_{int}\|_{\alpha} + \|\upa\|_{\alpha}) [a]_{\alpha}  d^{2\alpha}(x,y) \\
 & \leq  (N_0^{int}+ \|\upa\|_{\alpha}) [a]_{\alpha}  d^{2\alpha}(x,y) .
 \end{split}
\end{align}

\noindent \textbf{Step 2:} \textit{(Application of Lemma \ref{propogation})} \quad We have assumes that $\upa$ is modelled after $\vp$ according to $a$ on $\partial \mathbb{R}^2_+$ with modelling constant $M_{\partial}$ and with respect to $\nu_{\partial}$. By \eqref{new_modelling_constant} of Lemma \ref{propogation}, Lemma \ref{small_v} with \hyperlink{B1}{(B1)} , and \eqref{intermediate_modelling_2.11}, we obtain that $q$ has the claimed modelling with modelling constant bounded as 
\begin{align}
\label{modelling_for_q_1}
\begin{split}
M & \lesssim M_{intermediate} + M_{\partial} + \|\nu_{\partial}\|_{2\alpha-1} + \|u\|_{\alpha} + N_0 + N_0^{int}\\
&\lesssim  M_{\partial} + \|\nu_{\partial}\|_{2\alpha-1} + \|u\|_{\alpha} + N_0 + N_0^{int}.
\end{split}
\end{align}

For the modelling of $\tilde{q}$, let $\tilde{x} = (x_1, |x_2|)$ for $x \in \mathbb{R}^2$. Then, notice that for any $x,y \in \mathbb{R}^2$ we have that $d(x,y) \geq d(\tilde{x}, \tilde{y})$, which implies that $\tilde{q}$ is modelled after $\tilde{\vi}$ according to $\tilde{a}$. We then apply part $i)$ of Lemma \ref{post_process} to see that, since $\tilde{a} = a$ on $\left\{ x_2 = 0 \right\}$,  $\tilde{q}$ is also modelled according to $a$ and that the bound \eqref{modelling_for_q_1} still holds.\\

\noindent \textbf{Step 3:} \textit{(Bound for the $C^{\alpha}$-norm)} \quad For our proof of \eqref{holder_q_constant}, we let $x,y \in \mathbb{R}^2_+$ and write
\begin{align}
\label{q_intermediate_norm}
\begin{split}
|q(x) - q(y)| & \lesssim ([\vi^{\prime}]_{\alpha} + \| \vi^{\prime} \|_1 [a]_{\alpha}) d^{\alpha}(x,y) \lesssim (N_0^{int} + \|u\|_{\alpha}) d^{\alpha}(x,y),
\end{split}
\end{align}
where we have used Lemma \ref{semigroup_bounds} with $\bar{a}$ and $\vi^{\prime}(\cdot, a_0)$ and that $[a]_{\alpha}\leq 1$. Part $ii)$ of Lemma \ref{semigroup_bounds} gives that $\| q \|  \lesssim N_0^{int} + \|u \|$.
We remark that we have also used \hyperlink{B3}{(B3)}.\\

\noindent  $ii)$ \textbf{Step 4:} \textit{(An intermediate modelling)}  \quad We begin by showing that $q_{1} - q_{0}$ is modelled after $(\vi^{\prime}_1 , -\vi^{\prime}_0)$ according to $(a_{1}, a_{0})$ on $\mathbb{R}^2_+$ with modelling constant bounded by
\begin{align}
\label{intermediate_modelling_2}
\delta M_{intermediate} \lesssim \| a_1 - a_0 \|_{\alpha}(N_0^{int} + \|u\|_{\alpha})+\delta N_0^{int} +  \| \upa_{1} - \upa_{0}\|_{\alpha}.
\end{align} 
To see this, for any two points $x,y \in \mathbb{R}^2_+$, we apply the triangle inequality  and the definition of $q_i$ to write
\begin{align}
\begin{split}
&\big| q_{1}(x) - q_{0} (x)  - ( q_{1}(y) - q_{0} (y))   
 - (\vi^{\prime}_1(x, a_{1}(y)) -\vi^{\prime}_1(y, a_{1}(y)))  + (\vi^{\prime}_0(x, a_{0}(y)) - \vi^{\prime}_0(y, a_{0}(y)))  \big|\\
& \lesssim   \big|  \vi^{\prime}_1(x, \bar{a}_1(x)) -  \vi^{\prime}_0(x, \bar{a}_0(x)) - (\vi^{\prime}_1(x, \bar{a}_1(y)) - \vi^{\prime}_0(x,\bar{a}_0(y))) \big|\\
& \quad   +\big |\vi^{\prime}_1(x, \bar{a}_1(y))  - \vi^{\prime}_1(x, a_{1}(y)) - (\vi^{\prime}_0(x,\bar{a}_0(y))  - \vi^{\prime}_0(x, a_{0}(y)))  \\
& \quad \qquad   - (\vi^{\prime}_1(y, \bar{a}_1(y))  -\vi^{\prime}_1(y, a_{1}(y)) - (\vi^{\prime}_0(y, \bar{a}_0(y))  -\vi^{\prime}_0(y, a_{0}(y))) )\big|.
\end{split}
\end{align}
We treat the first term essentially as \eqref{modelling_q_701} and, in particular, bound it by
\begin{align*}
\begin{split}
&  | (\vi^{\prime}_1(x, \bar{a}_1(x)) -   \vi^{\prime}_0(x, \bar{a}_0(x))) -(\vi^{\prime}_1(x, \bar{a}_1(y)) -   \vi^{\prime}_0(x, \bar{a}_0(y)))|\\
& \lesssim  \|\vi^{\prime}_1 - \vi^{\prime}_0\|_1 |\bar{a}_s(x) - \bar{a}_s(y)| \\
 & \lesssim (\| \uf_{int,1} - \uf_{int,0}\|_{\alpha}  + \| \upa_{1} - \upa_{0}\|_{\alpha} ) [ a_s ]_{\alpha} d^{2\alpha}(x,y) \lesssim (\delta N_0^{int} +  \| \upa_{1} - \upa_{0}\|_{\alpha})d^{2\alpha}(x,y),
 \end{split}
\end{align*}
where $a_s$ is defined in \eqref{interpolate_defn}.  The second term is more involved, but was already treated in part $ii)$ of Lemma \ref{post_process}. Up to a multiplicative constant, it is bounded by 
\begin{align*}
( \| \uf_{int,1} - \uf_{int,0}\|_{\alpha}  + \| \upa_{1} - \upa_{0}\|_{\alpha}  + \max_i (\| \uf_{int,i}\|_{\alpha} + \| \upa_i \|_{\alpha} ) \|a_1 - a_0\|_{\alpha}) d^{2\alpha}(x,y).
\end{align*}

\noindent \textbf{Step 5:} \textit{(Application of Lemma \ref{propogation})} \quad Recall that we assume that $u_1 - u_0$ is modelled after $(v_1, -v_0)$ according to $(a_1, a_0)$ with modelling constant $\delta M_{\partial}$ and associated $\delta \nu_{\partial}$. The argument for the modelling of $\tilde{q}_1 - \tilde{q}_0$ is completed as in part $i)$, but instead using the second parts of Lemmas \ref{post_process} and \ref{propogation}.\\

\noindent \textbf{Step 6:} \textit{(Bound for the $C^{\alpha}$-norm)} \quad We first use Lemma \ref{semigroup_bounds} and \hyperlink{C3}{(C3)} to write
\begin{align}
\begin{split}
\label{difference_qs_infty}
 \|q_1- q_0\|  & \lesssim \| (\vi^{\prime}_1 -\vi^{\prime}_0) (\cdot , a_0) \| + \| \partial_{a_0} \vi^{\prime}_0 (\cdot, a_0) \| \|\overline{a_1 - a_0}\| )\\
 &  \lesssim\|a_1 - a_0 \|_{\alpha} (N_0^{int} + \|u_0\|)+ \| u_1 - u_0 \| + \delta N_0^{int},
 \end{split}
\end{align}
where $\overline{a_1 - a_0}$ solves \eqref{a_equn} with initial condition $a_1 - a_0$ and $(\vi^{\prime}_1 -\vi^{\prime}_0) (\cdot , a_0)$ solves \eqref{constant_coeff_ivp_body} with initial condition $U_{int,1} - u_1 - (U_{int,0} - u_0)$. Then, for two points $x,y \in \mathbb{R}^2_+$, we have that
\begin{align}
\label{stability_holder_norm_difference_1}
\begin{split}
& | (q_1 - q_0)(x) - (q_1 - q_0)(y)|\\
 & \leq |\vi^{\prime}_1(x,\bar{a}_1(x)) - \vi^{\prime}_0(x, \bar{a}_0(x))  - (\vi^{\prime}_1(x, \bar{a}_1(y)) - \vi^{\prime}_0(x, \bar{a}_0(y))) |\\
 & \quad + |\vi^{\prime}_1(x,\bar{a}_1(y)) - \vi^{\prime}_0(x, \bar{a}_0(y))  - (\vi^{\prime}_1(y, \bar{a}_1(y)) - \vi^{\prime}_0(y, \bar{a}_0(y))) |.
\end{split}
\end{align}
For the first term we let $\bar{a}_s$ be defined as $a_s$ in \eqref{interpolate_defn} and denote 
\begin{align}
V_s^{\prime} := s V_1^{\prime}(\cdot, a_0) + (1 - s) V_0^{\prime}(\cdot, a_0).
\end{align}
We then notice that
\begin{align}
\label{stability_holder_norm_difference_2}
\begin{split}
& \vi^{\prime}_1(x,\bar{a}_1(x)) -\vi^{\prime}_0(x, \bar{a}_0(x))  - (\vi^{\prime}_1(x, \bar{a}_1(y)) - \vi^{\prime}_0(x, \bar{a}_0(y)))\\
&=  \int_0^1 \partial_s ( \vi^{\prime}_s(x, \bar{a}_s(x)) - \vi^{\prime}_s(x, \bar{a}_s(y))) \, \textrm{d}s\\
& =    \int_0^1  \Big( \vi^{\prime}_1(x, \bar{a}_s(x)) -   \vi^{\prime}_0(x, \bar{a}_s(x)) -(\vi^{\prime}_1(x, \bar{a}_s(y)) -   \vi^{\prime}_0(x, \bar{a}_s(y)) )\\
& \, \quad \quad \quad \quad  \quad \quad+ ( \partial_{a_0} \vi^{\prime}_s (x, \bar{a}_s(x)) -  \partial_{a_0} \vi^{\prime}_s(x, \bar{a}_s(y))) \overline{a_1 - a_0} (x) \\
& \, \quad \quad \quad \quad \quad \quad \quad  \quad \quad+ \partial_{a_0} \vi^{\prime}_s(x, \bar{a}_s(y)) (\overline{a_1 - a_0} (x) - \overline{a_1 - a_0} (y)) \vphantom{\sum}  \Big)\, \textrm{d} s.
\end{split}
\end{align}
Using the bounds from Lemma \ref{semigroup_bounds} and \hyperlink{C3}{(C3)} we obtain
\begin{align}
& | \vi^{\prime}_1(x, \bar{a}_s(x)) -   \vi^{\prime}_0(x, \bar{a}_s(x))-(\vi^{\prime}_1(x, \bar{a}_s(y)) -   \vi^{\prime}_0(x, \bar{a}_s(y)))| \\
& \quad \lesssim ( \|u_1 - u_0\|_{\alpha}  + \delta N_0^{int}) d^{\alpha}(x,y),\\
& |( \partial_{a_0} \vi^{\prime}_s (x, \bar{a}_s(x)) -  \partial_{a_0} \vi^{\prime}_s(x, \bar{a}_s(y))) \overline{a_1 - a_0}(x)|  \lesssim  \|a_1 - a_0 \|(\|u_s \| + N_0^{int}) d^{\alpha}(x,y), \\
\textrm{and} \quad & |\partial_{a_0} \vi^{\prime}_s(x, \bar{a}_s(y)) (\overline{a_1 - a_0} (x) -\overline{a_1 - a_0} (y))| \lesssim [a_1 - a_0]_{\alpha} (\|u_s\| + N_0^{int}) d^{\alpha}(x,y).
\end{align}
Combining these estimates gives
\begin{align}
\label{stability_holder_norm_difference_3}
\begin{split}
& | \vi^{\prime}_1(x,\bar{a}_1(x)) -  \vi^{\prime}_0(x, \bar{a}_0(x))  - ( \vi^{\prime}_1(x, \bar{a}_1(y)) -  \vi^{\prime}_0(x, \bar{a}_0(y)))|\\
 & \lesssim  (\|u_1 - u_0\|_{\alpha} + \delta N_0^{int} + \|a_1 - a_0\|_{\alpha}  (\|u_s\|_{\alpha} + N_0^{int})) d^{\alpha}(x,y). 
\end{split}
\end{align}
A similar strategy can be used to bound the second term on the right-hand side of \eqref{stability_holder_norm_difference_1}. In particular, we write
\begin{align}
\label{stability_holder_norm_difference_4}
\begin{split}
& | \vi^{\prime}_1(x,\bar{a}_1(y)) -  \vi^{\prime}_0(x, \bar{a}_0(y))  - ( \vi^{\prime}_1(y, \bar{a}_1(y)) - \vi^{\prime}_0(y, \bar{a}_0(y)))|\\
&= \Big| \int_0^1 \partial_s (  \vi^{\prime}_s(x, \bar{a}_s(y)) -  \vi^{\prime}_s(y, \bar{a}_s(y))) \, \textrm{d}s \Big|\\
& \leq \int_0^1  ( | (\vi^{\prime}_1 -    \vi^{\prime}_0)(x, \bar{a}_s(y)) -( \vi^{\prime}_1-    \vi^{\prime}_0)(y, \bar{a}_s(y)) |  \\
& \quad  \quad \quad \quad \quad + | \partial_{a_0} \vi^{\prime}_s (x, \bar{a}_s(y)) -  \partial_{a_0} \vi^{\prime}_s(y, \bar{a}_s(y))| \, | \overline{a_1 - a_0} (y)| )\, \textrm{d} s\\
& \lesssim (\delta N_0^{int} +\|u_1 - u_0\|_{\alpha}  + \|a_1 - a_0 \|_{\alpha} (\|u_s\| + N_0^{int}) )d^{\alpha}(x,y).
\end{split}
\end{align}
Together  \eqref{stability_holder_norm_difference_1},  \eqref{stability_holder_norm_difference_3}, and \eqref{stability_holder_norm_difference_4} show that 
\begin{align}
\label{difference_holder_norm_prop_2}
[q_1 - q_0]_{\alpha} \lesssim \|a_1 - a_0 \|_{\alpha} (\max_{i=0,1} \|u_i\|_{\alpha} + N_0^{int}) + \|u_1 - u_0\|_{\alpha} + \delta N_0^{int}. 
\end{align}

\end{proof}

\subsection{Proofs of auxiliary lemmas for Proposition \ref{ansatz_IVP}} 
\label{section:modelling_q_aux}

We begin with the argument for Lemma \ref{post_process}:

\begin{proof}
To keep notation lean, in this proof we set $\vi_i(\cdot, a_0) :=  \vi(\cdot, a_0, \vi_{int,i}(a_0))$ and in part $i)$ drop the index $i$. This is technically in conflict with --more general than-- \eqref{vi_short}, which we use in the rest of this paper.\\


\noindent $i)$ First, we write
\begin{align}
\label{lemma6_1}
\begin{split}
&\left|\vi(x, a(y)) -\vi(y, a(y)) - (\vi(x, a^{\prime}(y)) -\vi(y, a^{\prime}(y)) ) \right|\\
& \lesssim \sup_{a_0 \in [\lambda,1]} \left| \partial_{a_0} (\vi(x, a_0) -\vi(y, a_0)) \right| \left|a(y) - a^{\prime}(y)\right|.
\end{split}
\end{align}
Notice that, since $a  = a^{\prime}$ on $\{x_2 = 0\}$, we have that 
\begin{align}
\label{a_bound_match_up}
|a(y) - a^{\prime}(y)| \lesssim ( [a]_{\alpha} + [a^{\prime}]_{\alpha}) y_2^{\frac{\alpha}{2}}.
\end{align}
Using \eqref{a_bound_match_up}, we then bound the right-hand side of \eqref{lemma6_1} in two ways:
\begin{align}
\label{lemma6_2}
\begin{split}
 &\sup_{a_0 \in [\lambda,1]} | \partial_{a_0} (\vi(x, a_0) -\vi(y, a_0)) | \, |a(y) - a^{\prime}(y)|\\
&  \lesssim
\| \vi_{int}\|_{\alpha, 1}( [a]_{\alpha} + [a^{\prime}]_{\alpha}) \times
  \begin{cases}
  y_2^{\frac{\alpha}{2}} (x_2^{-\frac{\alpha}{2}} + y_2^{-\frac{\alpha}{2}}  ) d^{2\alpha}(x,y)\\
 y_2^{\frac{\alpha}{2}} d^{\alpha}(x,y),\\
  \end{cases}
  \end{split}
\end{align}
where we use either \eqref{heat_kernel_bound_3.5} or \eqref{heat_kernel_bound_4} applied to $\vi(\cdot, a_0)$. We now consider two cases: $y_2 \leq 2 x_2$ and  $2x_2 \leq y_2$. For the first case we use the top estimate of \eqref{lemma6_2}, which can then easily be bounded by $\| \vi_{int}\|_{\alpha,1}( [a]_{\alpha} + [a^{\prime}]_{\alpha})d^{2\alpha}(x,y)$. In the second case, we have that $\frac{y_2}{2} \leq y_2 - x_2$, which allows to bound the bottom term of \eqref{lemma6_2} in the same way. (Both of these bounds are up to a multiplicative constant.)

Our modelling claim then follows from \eqref{defn_model_intro} and the triangle inequality. In particular, for $x,y \in \mathbb{R}^2_+$, the relation \eqref{Lemma_6_result_1} gives that
\begin{align*}
\begin{split}
& |\ui(x) - \ui(y) - (\vi(x, a^{\prime}(y)) -\vi(y, a^{\prime}(y))) - \nu(y)(x-y)_1|\\
 &\lesssim M d^{2\alpha}(x,y) + |\vi(x, a(y)) -\vi(y, a(y)) - (\vi(x, a^{\prime}(y)) -\vi(y, a^{\prime}(y)))|\\
& \lesssim (M + \|\vi_{int}\|_{\alpha,1})d^{2\alpha}(x,y).
\end{split}
\end{align*}

\vspace{.2cm}

\noindent $ii)$ The triangle inequality yields:
\begin{align}
\label{postprocess_modelling}
\begin{split}
& | \ui(x) - \ui(y) - (-1)^{i} (\vi_i(x, a^{\prime}_i(y)) -\vi_i(y, a^{\prime}_i(y))) - \nu(y)(x-y)_1|\\
& \lesssim  \delta M d^{2\alpha}(x,y)\\
& \quad \quad + |\vi_0(x, a_0(y)) -\vi_0(y, a_0(y)) - (\vi_0(x, a^{\prime}_0(y)) -\vi_0(y, a^{\prime}_0(y)))\\
&\quad \quad \quad \quad  - (\vi_1(x, a_1(y)) -\vi_1(y, a_1(y)) - (\vi_1(x, a^{\prime}_1(y)) -\vi_1(y, a^{\prime}_1(y))))|
\end{split}
\end{align}
for any points $x,y \in \mathbb{R}^2_+$. Letting 
\begin{align}
a_i^t = t a_i + (1-t) a_{i}^{\prime} \,  \textrm{ for }\, i = 0,1 \qquad  \textrm{ and }  \qquad a_s^t = s a_1^t + (1-s) a_0^t,
\end{align}
we then notice that
\begin{align}
\label{derivs_inter_coeff}
\begin{split}
 \partial_s a_s^t & = a_1^t - a_0^t,\\
\partial_t a_s^t & = s(a_1 - a_{1}^{\prime}) +(1-s) (a_0 - a_{0}^{\prime}),\\ 
\textrm{and} \quad  \partial_t \partial_s a_s^t  & = a_1 - a_0 - (a_{1}^{\prime} - a_{0}^{\prime}).
\end{split}
\end{align}
This new notation allows us to write
\begin{align}
\label{prop_2_stab_modelling_3}
\begin{split}
& |\vi_1(x, a_1(y))  -\vi_1(x, a_1^{\prime}(y)) - (\vi_0(x, a_0(y))  -\vi_0(x, a_{0}^{\prime}(y))) \\
& \quad  - \left(\vi_1(y, a_1(y))  -\vi_1(y, a^{\prime}_{1}(y)) - (\vi_0(y, a_0(y))  -\vi_0(y, a^{\prime}_{0}(y))) \right)|\\
& = \Big| \int_0^1 \int_0^1  \partial_s \partial_t(\vi_s(x, a_s^t(y))  -\vi_s(y, a_s^t(y))) \, \textrm{d} s \, \textrm{d} t \Big|\\
& =   \Big| \int_0^1 \int_0^1  \partial_s ( (\partial_{a_0}\vi_s(x, a_s^t(y)) - \partial_{a_0}\vi_s(y, a_s^t(y))) \partial_t a_s^t(y) )  \textrm{d}s  \, \textrm{d} t\Big|\\
& \leq   \int_0^1 \int_0^1\Big(  | \partial_{a_0} (\vi_1 -\vi_0)(x, a_s^t(y)) - \partial_{a_0} (\vi_1 -\vi_0)(y, a_s^t(y))  | \, | \partial_t a_s^t(y)| \\
& \quad  \quad \quad \quad \quad \quad + |\partial_{a_0}^2\vi_s (x, a_s^t(y)) - \partial_{a_0}^2\vi_s(y, a_s^t(y)) | \, |\partial_t a_s^t(y)| \, |  \partial_s a_s^t(y)|\\
&   \quad \quad \quad \quad \quad \quad \quad \quad \quad  \quad \quad  + |\partial_{a_0}\vi_s(x, a_s^t(y)) - \partial_{a_0}\vi_s(y, a_s^t(y))| \,  | \partial_t \partial_s a_s^t(y)|  \Big) \, \textrm{d} s  \, \textrm{d} t .
\end{split}
\end{align}
To finish we bound the three terms on the right-hand side. Using the relations \eqref{derivs_inter_coeff}, these terms are treated in the same manner as \eqref{lemma6_1} above. In particular, the first term can be bounded as
\begin{align*}
\begin{split}
& | \partial_{a_0} (\vi_1 -\vi_0)(x, a_s^t(y)) - \partial_{a_0} (\vi_1 -\vi_0)(y, a_s^t(y))  | \, | \partial_t a_s^t(y)|
\lesssim  \| \vi_{int, 1} - \vi_{int,0}\|_{\alpha,1}  d^{2\alpha}(x,y),
\end{split}
\end{align*}
where we have used \eqref{heat_kernel_bound_3.5}, \eqref{heat_kernel_bound_4}, and \eqref{a_bound_match_up}  applied to $a_i$ and $a_i^{\prime}$. For the second term we use the same strategy and, additionally, that $| \partial_s a_s^t(y)| \lesssim \|a_1 - a_0 \|  + \|a^{\prime}_1 - a^{\prime}_0 \|$. We obtain that
\begin{align*}
\begin{split}
& | \partial^2_{a_0}\vi_s(x, a_s^t(y)) - \partial^2_{a_0}\vi_s(y, a_s^t(y))|\, | \partial_t a_s^t(y)| \, |\partial_s a_s^t (y)|\\
& \lesssim   \| \vi_{int,s}\|_{\alpha,2}  (\|a_1 - a_0 \|  + \|a^{\prime}_1 - a^{\prime}_0 \|) d^{2\alpha}(x,y).
\end{split}
\end{align*}
For the last term we use that
 \begin{align*}
 |(a_1 - a_0 )(y)- (a_{1}^{\prime} - a_{0}^{\prime})(y)| \lesssim([a_1 - a_0]_{\alpha} + [a^{\prime}_1 - a^{\prime}_0]_{\alpha}  ) y_2^{\frac{\alpha}{2}}
 \end{align*}
and either \eqref{heat_kernel_bound_3.5} or \eqref{heat_kernel_bound_4}. We obtain the relation
\begin{align*}
\begin{split}
 |\partial_{a_0} \vi_s(x, a_s^t(y)) - \partial_{a_0} \vi_s(y, a_s^t(y))| \,  | \partial_t \partial_s a_s^t(y)| 
 \lesssim  \| \vi_{int, s} \|_{\alpha,1}([a_1 - a_0]_{\alpha} + [a^{\prime}_1 - a^{\prime}_0]_{\alpha}) d^{2\alpha}(x,y).
\end{split}
\end{align*}

\end{proof}

Here is the argument for Lemma \ref{propogation}:

\begin{proof} Mainly, we combine \eqref{defn_model_intro} with the heat kernel formulation given in \eqref{heat_kernel_representation}. We start with part $i)$: \\

\noindent  $i)$ \textbf{Step 1:} \textit{(Modelling according to $a_{tr}$)} \quad  We begin with an application of part $i)$ of Lemma \ref{post_process}. In particular,  if we define 
\begin{align}
\label{defn_a_trace}
a_{tr} (x) := a(x_1, 0),
\end{align}
for $x \in \R^2_+$, then $\ui$ is modelled after $\vi(\cdot, a_0, U-u)$ according to $a_{tr}$ with a modelling constant $M_{tr}$ bounded above as $M_{tr} \lesssim M + \|U\|_{\alpha} + \|u \|_{\alpha}$.\\

\noindent \textbf{Step 2:} \textit{(Use of the initial modelling)} The crucial step of our proof is showing that
\begin{align} 
\label{modelling_theorem_1}
\begin{split}
& | \vi(x, a_{tr}(y),  u - v(a_{tr}(y))) - \vi(y, a_{tr}(y), u - v(a_{tr}(y)))  - \nu^{\int} (y) (x-y)_1 |\\
& \lesssim (M_{\partial} + \|\nu_{\partial}\|_{2\alpha-1} + \|u \|_{\alpha} + \| v\|_{\alpha}) d^{2\alpha}(x,y),
\end{split}
\end{align}
for any points $x,y \in \mathbb{R}^2_+$, where $v(a_{tr}(y))$ is used as shorthand for $v(\cdot, a_{tr}(y))$  and 
\begin{align}
\label{defn_nu_int}
\nu^{\int}(y) := e^{-y_2} \int_{\mathbb{R}} \frac{1}{(4 \pi a_{tr}(y) y_2 )^{\frac{1}{2}}}\nu_{\partial}(s, 0) e^{\frac{-|y_1 - s|^2}{4 y_2 a_{tr}(y)}}\, \textrm{d}s.
\end{align}
 Once we have shown \eqref{modelling_theorem_1}, an easy application of the triangle inequality and Step 1 shows that $\ui$ is modelled after  $\vi(\cdot, a_0, U-v(a_0))$ according to $a_{tr}$ with modelling constant $M_{intermediate}$ bounded as 
\begin{align}
M_{intermediate} \lesssim M + M_{\partial} + \|\nu_{\partial}\|_{2\alpha-1} + \|U\|_{\alpha} + \|u \|_{\alpha} +  \| v\|_{\alpha}.
\end{align}
We remark that $\vi(\cdot, \cdot, \cdot)$ is linear in the third argument, since this is the initial condition.

To show \eqref{modelling_theorem_1} we first use the heat kernel representation \eqref{heat_kernel_representation} and \eqref{defn_nu_int} to write
\begin{align}
\label{set_up_post_process_modelling}
\begin{split}
& | \vi(x, a_{tr}(y),  u - v(a_{tr}(y)) )- \vi(y, a_{tr}(y), u - v(a_{tr}(y)))  - \nu_{\partial}^{\int} (y) (x-y)_1 |\\
& \lesssim  e^{-y_2}\Big| \int_{\mathbb{R}}\Big( u(x_1 - z (4x_2 a_{tr}(y))^{\frac{1}{2}},0) - u(y_1 - z (4y_2 a_{tr}(y))^{\frac{1}{2}},0) \\
& \quad \quad \quad \quad \quad - \big( v ((x_1 - z (4x_2 a_{tr}(y))^{\frac{1}{2}}, 0), a(y_1 - z (4y_2 a_{tr}(y))^{\frac{1}{2}},0) ) \\
&\quad \quad \quad  \quad \quad \quad \quad   - v((y_1 - z (4y_2 a_{tr}(y))^{\frac{1}{2}}, 0), a(y_1 - z (4y_2 a_{tr}(y))^{\frac{1}{2}}, 0)) \big)\\
&\quad \quad \quad  \quad \quad \quad \quad  \quad \quad  -  \nu_{\partial}(y_1 - z (4y_2 a_{tr}(y))^{\frac{1}{2}},0)(x-y)_1  \Big)e^{-z^2}\, \textrm{d}z\Big|\\
& \quad  + |e^{- x_2} - e^{-y_2}| ( \| u \| + \|v\|).
\end{split}
\end{align}
Notice that when $d(x,y) \leq 1$, since $\alpha \in (0,1)$, we may bound the second term using
\begin{align}
\label{lemma10_a}
|e^{- x_2} - e^{-y_2}|  \lesssim  |x_2 - y_2| \lesssim d^{2\alpha}(x,y).
\end{align}
The relation \eqref{lemma10_a} is trivial when $d(x,y) \geq 1$ since then  $|e^{- |x_2|} - e^{-|y_2|}| \leq 2$ and $d^{2\alpha} (x,y) \geq 1$. 

For the first term of \eqref{set_up_post_process_modelling} we first let $d(x,y) \geq 1$. Then the term can be bounded by: 
\begin{align}
\label{lemma10_b}
([u ]_{\alpha} + [v ]_{\alpha} + \| \nu_{\partial}\|) d^{\alpha}(x,y) \leq ([u ]_{\alpha} + [v]_{\alpha} + \| \nu_{\partial}\|) d^{2\alpha}(x,y)
\end{align}
The situation that $d(x,y) \leq 1$ is more involved and requires the modelling of $u$. We remark that the modelling and the triangle inequality allow us to bound the first term of \eqref{set_up_post_process_modelling} by 
\begin{align}
\label{redo_modelling}
\begin{split}
\hspace{-.4cm}M_{\partial} d^{2\alpha}(x,y)  \hspace{-.05cm} + \hspace{-.05cm}
 \Big|  \Big( (4x_2 a_{tr}(y))^{\frac{1}{2}} - (4y_2 a_{tr}(y))^{\frac{1}{2}} \Big) \hspace{-.1cm} \int_{\mathbb{R}} \nu_{\partial}(y_1 - z (4y_2 a_{tr}(y))^{\frac{1}{2}},0) z e^{-z^2}  \textrm{d}z \Big|.
 \end{split}
\end{align}
It remains to treat the second term of \eqref{redo_modelling}, which we do in four cases:\\

\noindent \textit{Case 1--} We assume that $y_2^{\frac{1}{2}} \geq d(x,y)$ and $y_2 \leq x_2$. Since $\sqrt{\cdot}$ is Lipschitz on $[ y_2, \infty)$ with constant $\frac{1}{2}y_2^{-\frac{1}{2}}$, we may write
 \begin{align}
\label{new_sqrtroot}
(4x_2 a_{tr}(y))^{\frac{1}{2}} - (4y_2 a_{tr}(y))^{\frac{1}{2}} \lesssim  (x_2 - y_2) y_2^{-\frac{1}{2}} .
\end{align}
Additionally, using that $\nu_{\partial} \in C^{2 \alpha -1}(\mathbb{R}^2)$ yields that
\begin{align*}
&\Big| \big( (4x_2 a_{tr}(y))^{\frac{1}{2}} - (4y_2 a_{tr}(y))^{\frac{1}{2}} \big)  \int_{\mathbb{R}} \nu_{\partial}(y_1 - z (4y_2 a_{tr}(y))^{\frac{1}{2}},0) z e^{-z^2} \, \textrm{d}z \Big|\\
& \lesssim    y_2^{-\frac{1}{2}}  |x_2 - y_2| \, \Big|  \int_{\mathbb{R}}    (\nu_{\partial}(y_1 - z (4y_2 a_{tr}(y))^{\frac{1}{2}},0)  - \nu_{\partial}(y_1, 0)) z e^{-z^2} \, \textrm{d}z \Big|\\
& \lesssim  y_2^{\frac{2\alpha -2}{2}} |x_2 - y_2|  [\nu_{\partial}]_{2\alpha -1} \int_{\mathbb{R}}|z|^{2\alpha} e^{-z^2} \, \textrm{d}z  \lesssim  [\nu_{\partial}]_{2\alpha -1} d^{2\alpha} (x,y).
\end{align*}

\noindent \textit{Case 2--} We assume that $x_2^{\frac{1}{2}} \geq d(x,y)$ and $x_2 \leq y_2$. Following the same recipe as in the previous case and adding in a couple of uses of the triangle inequality, we obtain
\begin{align*}
& \Big|\Big( (4x_2 a_{tr}(y))^{\frac{1}{2}} - (4y_2 a_{tr}(y))^{\frac{1}{2}}\Big) \int_{\mathbb{R}} \nu_{\partial}(y_1 - z (4y_2 a_{tr}(y))^{\frac{1}{2}},0) z e^{-z^2} \, \textrm{d}z \Big|\\
& \lesssim  y_2^{\frac{2\alpha -1}{2}} x_2^{-\frac{1}{2}} |y_2 - x_2|  [\nu_{\partial}]_{2\alpha -1}  \int_{\mathbb{R}}|z|^{2\alpha} e^{-z^2} \, \textrm{d}z \\
& \lesssim \Big( |y_2 - x_2|^{\frac{2\alpha -1}{2} }  \hspace{-.1cm} +  \hspace{-.05cm}x_2^{\frac{2\alpha -1}{2}}\Big) x_2^{-\frac{1}{2}} |y_2 - x_2|  [\nu_{\partial}]_{2\alpha -1}  \hspace{-.05cm}  \int_{\mathbb{R}}  \hspace{-.05cm} |z|^{2\alpha}e^{-z^2} \textrm{d}z  \lesssim   [\nu_{\partial}]_{2\alpha -1}  d^{2\alpha} (x,y).
\end{align*}

\noindent \textit{Case 3-- } We assume that $x_2^{\frac{1}{2}}  \leq d(x,y)$. Now we use the bound
\begin{align}
\label{new_sqrtroot_2}
| (4 x_2 a_{tr}(y))^{\frac{1}{2}} - (4y_2 a_{tr}(y))^{\frac{1}{2}}| \lesssim |x_2 - y_2|^{\frac{1}{2}}.
\end{align}
Using the same methods as in the previous cases, we find that 
\begin{align*}
\begin{split}
& \Big| \Big( (4x_2 a_{tr}(y))^{\frac{1}{2}} - (4y_2 a_{tr}(y))^{\frac{1}{2}}\Big) \int_{\mathbb{R}} \nu_{\partial}(y_1 - z (4 y_2 a_{tr}(y))^{\frac{1}{2}},0) z  e^{-z^2} \, \textrm{d}z \Big|\\
& \lesssim  \Big(|y_2 - x_2|^{\frac{2\alpha -1}{2}} + x_2^{\frac{2\alpha -1}{2}}\Big)  |x_2 - y_2|^{\frac{1}{2}}  [\nu_{\partial}]_{2\alpha -1} \int_{\mathbb{R}}|z|^{2\alpha}e^{-z^2} \, \textrm{d}z \lesssim    [\nu_{\partial}]_{2\alpha -1}   d^{2\alpha} (x,y).
\end{split}
\end{align*}

\noindent \textit{Case 4--} We assume that $y_2^{\frac{1}{2}} \leq d(x,y)$. Reusing \eqref{new_sqrtroot_2}, we obtain 
\begin{align*}
&\Big|  \Big( (4 x_2 a_{tr}(y))^{\frac{1}{2}} - (4 y_2 a_{tr}(y))^{\frac{1}{2}} \Big) \int_{\mathbb{R}} \nu_{\partial}(y_1 - z (4 y_2 a_{tr}(y))^{\frac{1}{2}},0) z e^{-z^2} \, \textrm{d}z \Big|\\
& \lesssim   y_2^{\frac{2\alpha -1}{2}} |x_2 - y_2|^{\frac{1}{2}} [\nu_{\partial}]_{2\alpha -1}  \int_{\mathbb{R}}|z|^{2\alpha}  e^{-z^2} \, \textrm{d}z \lesssim    [\nu_{\partial}]_{2\alpha -1}    d^{2\alpha} (x,y).
\end{align*}

\noindent \textbf{Step 3:} \textit{(Conclusion)} \quad We again apply part \textit{i)} of Lemma \ref{post_process}, but now to the modelling proven in the previous step to swap out $a_{tr}$ for $a$. We finally obtain that $\ui$ is modelled after $\vi(\cdot, a_0, U - v(\cdot, a_0))$ according to $a$ with modelling constant bounded as specified in \eqref{new_modelling_constant}.\\

\noindent $ii)$  Analogously to part \textit{i)}, we first notice that by part $ii)$ of Lemma \ref{post_process}, $\ui$ is modelled after $(\vi_1(\cdot, a_0, U_1 - u_1) , -\vi_0(\cdot, a_0, U_0 - u_0))$ according to $(a_{1,tr}, a_{0,tr})$. Here, we use the notation from \eqref{defn_a_trace}. The corresponding modelling constant is bounded as 
\begin{align}
\label{final_edit_modelling_1}
\delta M_{tr} \lesssim \delta M + \| U_1 - U_0\|_{\alpha} + \| u_1 - u_0\|_{\alpha} + \max_{i = 0,1} (\| U_i \|_{\alpha} + \| u_i \|_{\alpha} ) \|a_1 - a_0 \|_{\alpha}.
\end{align}

We then show that $\ui$ is modelled after $(\vi_1(\cdot, a_0, U_1 - v_1(a_0)) , -\vi_0(\cdot, a_0, U_0 - v_0(a_0)))$ according to $(a_{1,tr}, a_{0,tr})$, which, by the same strategy as in part $i)$, reduces to showing that 
\begin{align}
\label{modelling_theorem_1_2}
\begin{split}
& \big|  ( \vi_0(x, a_{0,tr}(y), u_0 - v_0(a_0))  - \vi_0(y, a_{0,tr}(y), u_0 - v_0(a_0))) \\
&  \quad  -(\vi_1(x, a_{1,tr}(y), u_1 - v_1(a_1) ) - \vi_1(y, a_{1,tr}(y),u_1 - v_1(a_1) )  - \delta \nu^{\int}(y)(x - y)_1 \big|  \\
 & \lesssim  \left( \delta M_{\partial} + \|\delta \nu_{\partial}\|_{2\alpha -1}  + \|u_1 - u_0 \|_{\alpha} + \|v_1 - v_0\|_{\alpha} \right) d^{2\alpha}(x,y)
\end{split}
\end{align}
for $x,y \in \mathbb{R}^2_+$. Of course, this is the analogue of \eqref{modelling_theorem_1} from part $i)$ and $\delta \nu^{\int}$ is defined as in \eqref{defn_nu_int}, but in terms of $\delta \nu_{\partial}$. The argument for \eqref{modelling_theorem_1_2} is completely analogous to that for \eqref{modelling_theorem_1} --we do not repeat the calculation.

Combining \eqref{final_edit_modelling_1} and \eqref{modelling_theorem_1_2} with the triangle inequality yields that the modelling constant corresponding to the intermediate modelling, \textit{i.e.} the modelling according to $(a_{1,tr},a_{0,tr})$ proven above, is bounded as
 \begin{align}
 \label{final_edit_modelling_2}
 \begin{split}
 \delta M_{intermediate} \lesssim & \delta M + \delta M_{\partial} + \|\delta \nu_{\partial}\|_{2\alpha -1}  + \| U_1 - U_0\|_{\alpha} + \| u_1 - u_0\|_{\alpha}\\
 & + \max_{i = 0,1} (\| U_i \|_{\alpha} + \| u_i \|_{\alpha} ) \|a_1 - a_0 \|_{\alpha} +  \|v_1 - v_0\|_{\alpha}.
\end{split}
 \end{align}
Applying part $ii)$ of Lemma \ref{post_process} we obtain the desired modelling.

\end{proof}

\subsection{Argument for Proposition \ref{linear_IVP}: Analysis of the linear problem with trivial forcing}
\label{section:linear_IVP}


In this section we correct the ansatz $q$ in order to solve \eqref{IC_2}. We first collect two technical lemmas:


\begin{lemma}
\label{lemma:norm}
Let $\alpha \in (0,1)$. If a regular distribution $f$ on $\mathbb{R}^2$ satisfies the relation 
\begin{align}
\label{growth_condition_ivp_1}
 |f(x)| \leq C_f |x_2|^{\frac{2\alpha -2}{2}},
\end{align}
for any $x \in \mathbb{R}^2$ and some $C_f \in \mathbb{R}$, then for $T>0$ and $j,l \geq 0$ we have that
\begin{align}
\label{equiv_seminorm_set_0}
 \| \partial_1^j\partial_2^l f_{T} \| \lesssim C_f (T^{\frac{1}{4}})^{2\alpha - 2 - j - 2l}
\end{align}
and
\begin{align}
\label{equiv_seminorm_set_3}
[f_T]_{\alpha} \lesssim C_f (T^{\frac{1}{4}})^{\alpha -2}.
\end{align}
If, additionally, we know that $f \equiv 0$ on $\mathbb{R}^2_-$, then we have that
\begin{align}\label{j002}
 \| \partial_1^j\partial_2^l f_{T} \|_{\mathbb{R}^2_L} \lesssim C_f L^{-\delta}(T^{\frac{1}{4}})^{2\alpha - 2 - j - 2l + 2\delta}
\end{align}
and
\begin{align}
\label{equiv_seminorm_set_2}
[f_T]_{\alpha; \mathbb{R}^2_L} \lesssim C_f   L^{-\delta}  (T^{\frac{1}{4}})^{\alpha-2+2\delta},
\end{align}
for any $\delta,L>0$ --where the uniform constants may depend additionally on $\delta$.
\end{lemma}

\noindent We also need a lemma that combines the methods of Section \ref{section:modelling_q} with those of Lemma \ref{semigroup_bounds}. In particular, we show the following bound:
\begin{lemma}
\label{semigroup_bound_modelling}
Let $\alpha \in (\frac{2}{3},1)$. Assume that $u \in C^{\alpha}(\mathbb{R}^2)$ is modelled after $v$ according to $a \in C^{\alpha}(\mathbb{R}^2)$ on $\left\{x_2 = 0 \right\}$,  satisfying $\| a \|_{\alpha} \leq 1$ and $a \in [\lambda, 1]$ for $\lambda > 0$, with modelling constant $M_{\partial}$. Then, for any $x \in \mathbb{R}^2_+$, we find that 
\begin{align}
| E_{tr} \partial_1^2 \vi (x, a_0, u-v(a_0) ) |  \lesssim M_{\partial} |x_2|^{\frac{2\alpha -2 }{2}},
\end{align}
where $E_{tr}$ denotes evaluation of the parameter $a_0$ at $a_{tr}(x)$, which we have defined in \eqref{defn_a_trace}.

\end{lemma}

\noindent Both Lemmas \ref{lemma:norm} and \ref{semigroup_bound_modelling} are proven in Appendix  \ref{technical_lemmas_2}.

Using these technical tools, we now give the main argument of this section:\\

\begin{proof}[Proposition \ref{linear_IVP}] The idea of this proof is to correct the ansatz $q$ defined in Definition \ref{q}. This proof has eight steps, of which the first four correspond to $i)$.\\

\noindent  $i)$ \textbf{Step 1:} \textit{(Regularity for the forcing of the equation solved by $w$)} \quad In this step we show that, for any point $x \in \mathbb{R}^2$, the bound 
\begin{align}
\label{Step_2_conclusion}
| (\partial_2 q - a  \partial_1^2 q +q)^E(x)| \lesssim N(N_0^{int} + N_0)  |x_2|^{\frac{2\alpha -2}{2}}
\end{align}
holds, where we use Definition \ref{extensions}. By the bounds in Lemma \ref{lemma:norm} and the equivalence in Lemma \ref{equivnorm}, we interpret this as information on the $C^{2\alpha-2}$-norm. To obtain \eqref{Step_2_conclusion}, we first notice that on $\mathbb{R}^2_+$ the expression $\partial_2 q - a  \partial_1^2 q +q$ is classical since $q$ is smooth for positive times. Applying Leibniz' rule we find that 
\begin{align}
\label{step1_prop1_1}
\begin{split}
& ( \partial_2 q - a  \partial_1^2q +q) (x) \\
& = \partial_2 \vi^{\prime}( x ,  \bar{a}(x)) +  \partial_{a_0} \vi^{\prime}(x,  \bar{a}(x)) \partial_2 \bar{a}(x)   - a \partial_1^2\vi^{\prime}(x, \bar{a}(x))  -  2 a  \partial_1 \partial_{a_0}  \vi^{\prime}(x,  \bar{a}(x))  \partial_1 \bar{a}(x)\\ 
& \quad - a\partial_{a_0} \vi^{\prime}(x, \bar{a}(x))\partial_1^2\bar{a}(x) - a\partial_{a_0}^2\vi^{\prime}(x,  \bar{a}(x))  (  \partial_1 \bar{a}(x))^2  + \vi^{\prime}(x, \bar{a}(x)).
\end{split}
\end{align}
Notice that we have $\partial_2 \vi^{\prime}( x ,  \bar{a}(x))  = \bar{a}(x)   \partial_1^2 \vi^{\prime}( x ,  \bar{a}(x)) - \vi^{\prime}(x, \bar{a}(x))$ due to \eqref{constant_coeff_ivp_body} and $ \partial_2\bar{a}= \partial_1^2 \bar{a}$ from \eqref{a_equn}. Plugging in these identities, we obtain
\begin{align}
\label{defn_g}
\begin{split}
 & (\partial_2 q - a  \partial_1^2 q +q)(x)\\
 & = (  \bar{a} - a)(x)  \partial_1^2 \vi^{\prime}(x, \bar{a}(x)) + \partial_{a_0} \vi^{\prime}(x, \bar{a}(x)) (1 - a(x))\partial_1^2 \bar{a}(x)  \\
& \quad - 2 a(x)  \partial_{1} \partial_{a_0}\vi^{\prime}(x, \bar{a}(x)) \partial_1 \bar{a}(x) - a(x)\partial_{a_0}^2\vi^{\prime}(x,  \bar{a}(x)) (  \partial_1 \bar{a}(x))^2.
\end{split}
\end{align}

To complete this step we apply Lemma \ref{semigroup_bounds} to $\bar{a}$ to find that 
\begin{align}
\label{bound_bar_a}
\| \bar{a} \| \lesssim \| a \| \leq 1 \quad \textrm{ and } \quad [\bar{a}]_{\alpha} \lesssim [a]_{\alpha},
\end{align}
the second of which can be post-processed to give
\begin{align}
\label{a_difference}
|a(x)  - \bar{a}(x)| \lesssim [a]_{\alpha}|x_2|^{\frac{\alpha}{2}}.
\end{align}
Combining the above bounds with further applications of Lemma \ref{semigroup_bounds} to either $\bar{a}$ or $\vi^{\prime}(\cdot, a_0)$, we find that
\begin{align}
|( a  - \bar{a})(x)  \partial_1^2 \vi^{\prime}(x, \bar{a}(x)) | & \lesssim  [a]_{\alpha}([\uf_{int}]_{\alpha} + [\upa]_{\alpha} ) |x_2|^{\frac{2\alpha -2}{2}},\\
|  \partial_{a_0} \vi^{\prime}(x, \bar{a}(x)) (1 - a(x)) \partial_1^2 \bar{a}(x) | & \lesssim  [a]_{\alpha} ([\uf_{int}]_{\alpha} + [\upa]_{\alpha} )   |x_2|^{\frac{2\alpha -2}{2}},\\
| a(x)  \partial_{a_0} \partial_1 \vi^{\prime}(x, \bar{a}(x)) \partial_1\bar{a}(x)| & \lesssim [a]_{\alpha} ([\uf_{int}]_{\alpha} + [\upa]_{\alpha} ) |x_2|^{\frac{2 \alpha -2}{2}},\\
\textrm{and} \quad  | a(x) \partial_{a_0}^2 \vi^{\prime}(x, \bar{a}(x)) (  \partial_1 \bar{a}(x))^2 |
 & \lesssim  [a]_{\alpha} (\|\uf_{int}\| + \|\upa\| )  |x_2|^{\frac{2\alpha -2}{2}}.
\end{align}
We remark that for the second estimate above it is important that the initial condition of $\vi^{\prime}(\cdot, a_0)$ does not depend on $a_0$. These estimates, \eqref{edit_edit_3}, and the assumptions \hyperlink{B2}{(B2)}, \hyperlink{B3}{(B3)}, and $[a]_{\alpha} \leq N$ give \eqref{Step_2_conclusion}.\\

\noindent \textbf{Step 2:} \textit{(Construction of the correction $w$)} \quad 
 We now show that there exists $w \in C^{2 \alpha} (\mathbb{R}^2_+)$ solving 
\begin{align}
\begin{split}
\label{w_equation}
(\partial_2 - a \diamond \partial_1^2 +1 ) w  & = -( \partial_2q - a \partial_1^2q + q )  \hspace{2.5cm} \textrm{in} \quad  \mathbb{R}^2_+,\\
w & = 0 \hspace{5cm}  \textrm{on}  \quad \partial \mathbb{R}^2_+. 
\end{split}
\end{align}
In fact, we construct the solution $w$ of \eqref{w_equation} as a $C^{2\alpha}$-solution of 
\begin{align}
\label{correction_equn_ws}
(\partial_2 - a \diamond \partial_1^2 +1 ) w  & = - ( \partial_2q - a \partial_1^2q + q )^E  \quad \quad \quad \textrm{in} \quad  \mathbb{R}^2
\end{align}
and then show that $w|_{\mathbb{R}^2_-} =0$. The construction of the correction $w$ follows a similar procedure as part $i)$ of Proposition \ref{linear_forcing}.\\

\noindent \textit{Step 2.1:(A specific form of the singular product)} \quad Let $u \in C^{2\alpha}(\mathbb{R}^2)$ and satisfy $\partial_1^2 u \in C^{\alpha}(\Omega)$ for $\Omega \subseteq \mathbb{R}^2$. Using a trivial version of the argument from Step 1 of Proposition \ref{linear_forcing}, we find that the singular product $a \diamond \partial_1^2 u$ obtained using the trivial modelling of $u$ via Lemma \ref{lemma:reconstruct_2} coincides with the classical product on $\Omega$. In particular, this follows from the uniqueness in Lemma \ref{lemma:reconstruct_2}.\\

\noindent \textit{Step 2.2: (H\"{o}lder bounds for the right-hand side of \eqref{correction_equn_ws})} \quad Let $g = -(\partial_2q - a \partial_1^2q +q)^E$. We now estimate $\| g_{\tau} \|_{\alpha; \mathbb{R}^2_{L} }$ and $\| g_{\tau} \|_{\alpha; \mathbb{R}^2}$ for any $L\in(0,1)$ and $\tau >0$.  

We first bound $[ g_{\tau} ]_{\alpha; \mathbb{R}^2_{L} }$ for which we use \eqref{equiv_seminorm_set_2} of Lemma \ref{lemma:norm} with $\delta = \frac{\alpha +2}{2}$ and \eqref{Step_2_conclusion} to obtain
\begin{align*}
[g_{\tau}]_{\alpha; \mathbb{R}^2_{L}} \lesssim N(N_0 + N_0^{int}) (\tau^{\frac{1}{4}})^{2\alpha}
 L^{-\frac{\alpha +2}{2}} .
\end{align*}
To bound the corresponding $L^{\infty}$-norm, we use \eqref{j002} (again with $\delta =\frac{\alpha +2}{2}$), which gives that
\begin{align*}
\| g_{\tau} \|_{\mathbb{R}^2_L} \lesssim  N(N_0 + N_0^{int})  (\tau^{\frac{1}{4}})^{3\alpha} L^{-\frac{\alpha +2}{2}} .
\end{align*} 
For our estimate on $[g_{\tau}]_{\alpha; \mathbb{R}^2}$, we again use  \eqref{Step_2_conclusion}, but now in combination with \eqref{equiv_seminorm_set_3}; we find that  
\begin{align*}
[g_{\tau}]_{\alpha; \mathbb{R}^2} \lesssim N(N_0 + N_0^{int}) (\tau^{\frac{1}{4}})^{\alpha -2 }.
\end{align*}
For the $L^{\infty}$-norm $\| g_{\tau}\|$, we use \eqref{equiv_seminorm_set_0} to write 
\begin{align*}
\| g_{\tau} \|_{\mathbb{R}^2} \lesssim N(N_0 + N_0^{int}) (\tau^{\frac{1}{4}})^{2\alpha -2}.
\end{align*}

\vspace{.3cm}

\noindent \textit{Step 2.3: (Analysis of the regularized problem)}  \quad Let $\tau \in (0,1)$. From the last step we know that $g_{\tau} \in C^{\alpha}(\mathbb{R}^2)$, which means that there exists $w^{\tau} \in C^{\alpha+2}(\mathbb{R}^2)$ solving
\begin{align}
\label{approx_equation_prop_1}
(\partial_2 - a  \partial_1^2 +1 ) w^{\tau}  & = g_{\tau} && \textrm{in} \quad  \mathbb{R}^2.
\end{align}

Similar to Proposition \ref{linear_forcing}, we would now like to pass to the limit $\tau \rightarrow 0$ with an application of Lemma \ref{KrylovSafonov}. For this application we set $I =1$, $f_1(\cdot, a_0) =0 $, $\sigma_1 =0$, and $a = a$. We first check the condition \eqref{Lemma5_3}. Convolving \eqref{approx_equation_prop_1} with $\psi_T$, we obtain that $w^{\tau}$ solves 
\begin{align}
\label{approx_equation_prop_1_convolved}
(\partial_2 - a  \partial_1^2 +1 ) w^{\tau}_T  = g_{\tau +T} + (a \partial_1^2 w^{\tau})_T - a \partial_1^2 w^{\tau}_T \quad \quad \quad    \textrm{in} \quad  \mathbb{R}^2.
\end{align}
A calculation similar to  \eqref{corollary_3_calc_1}, taking \eqref{Step_2_conclusion} as input, yields that
\begin{align}
\label{approx_soln_1}
\|  g_{\tau  } \|_{2\alpha -2} \lesssim N (N_0 + N_0^{int}).
\end{align}
Furthermore, since $w^{\tau} \in C^{\alpha+2} (\mathbb{R}^2)$, we may apply Step 2.1, which is then combined with \eqref{Lemma4_4} of Lemma \ref{lemma:reconstruct_2}  to give that
\begin{align}
\label{approx_soln_2}
\sup_{T\leq1} (T^{\frac{1}{4}})^{2 - 2\alpha} \|(a \partial_1^2 w^{\tau})_T - a \partial_1^2 w^{\tau}_T \|  = \| [a, (\cdot) ] \diamond \partial_1^2 w^{\tau} \|_{2\alpha -2} \lesssim [a]_{\alpha} [w^{\tau}]_{2\alpha} .
\end{align}
Together  \eqref{approx_equation_prop_1_convolved}, \eqref{approx_soln_1}, and \eqref{approx_soln_2} yield that \eqref{Lemma5_3} is satisfied with $K = C([w^{\tau}]_{2\alpha}  [a]_{\alpha} + N(N_0 + N_0^{int} ))$ for some large enough constant $C \in \R$.

Applying Lemma \ref{KrylovSafonov} we find that 
\begin{align}
\label{modelling_correction}
[w^{\tau}]_{2\alpha} + \|w^{\tau}\|_{\alpha} \lesssim [a]_{\alpha} [w^{\tau}]_{2\alpha} + N(N_0 + N_0^{int} ),
\end{align}
which, after we use that $[a]_{\alpha}\ll 1$, gives
\begin{align}
\label{whole_space_holder_prop_2}
[w^{\tau}]_{2\alpha}  + \|w^{\tau}\|_{\alpha}  \lesssim N(N_0 + N_0^{int} ).
\end{align}

\noindent \textit{Step 2.4: (Passing to the limit in the regularization)} \quad We now pass to the limit $\tau \rightarrow 0$ in the sequence of approximate solutions $w^{\tau}$. Using the convention \eqref{definition_beta_12.2}, in which we define the $C^{2\alpha}$-seminorm, we see that \eqref{whole_space_holder_prop_2} allows us to apply the Arzel\`{a}-Ascoli theorem in $C^{2\alpha}(\mathbb{R}^2)$, which implies that up to a subsequence $w^{\tau} \rightarrow w$ uniformly. In order to pass to the limit in \eqref{approx_equation_prop_1}, just like in Step 3 of Proposition \ref{linear_forcing}, we first notice that $ g_{\tau} \rightharpoonup  g$ and $ \partial_2 w_{\tau} \rightharpoonup  \partial_2 w$ distributionally.  It is still necessary to show that $a \partial_1^2 w^{\tau} \rightharpoonup a \diamond \partial_1^2 w$, which follows from \eqref{Lemma4_3} of Lemma \ref{lemma:reconstruct_2}. To avoid repetition we again reference \cite[Proof of Proposition 3.8, Steps 9 and 10]{OW}. Notice lastly that since the bound \eqref{whole_space_holder_prop_2} is preserved under taking the limit $\tau \rightarrow 0$, we obtain \eqref{modelling_estimate_ivp_w}.

In order to see that $w$ satisfies the initial condition of \eqref{w_equation} we use the estimates from Step 2.2. In particular, the classical Schauder estimate for \eqref{approx_equation_prop_1} (see \textit{e.g.} \cite[Theorem 8.10.1]{Kr}) yields
\begin{align*} 
\| w_{\tau} \|_{\alpha+2 ;  \mathbb{R}^2_{L}} \lesssim  \| g_{\tau}\|_{\alpha; \mathbb{R}^2_L} \lesssim  N(N_0 + N_0^{int})L^{-\frac{\alpha +2}{2}} (\tau^{\frac{1}{4}})^{2\alpha};
\end{align*} 
and passing to the limit $\tau \rightarrow 0$ implies that $w \equiv 0$ on $\mathbb{R}^2_{L}$ for every $L>0$.\\

\noindent \textbf{Step 3:} \textit{(Uniqueness)}  \quad In this step we show that the correction $w$ solving \eqref{w_equation} such that $w \equiv 0$ on $\mathbb{R}^2_-$ is unique. To see this, we assume that we have two such solutions $w$ and $w^{\prime}$ and subtract them. We then use the same argument as in Step 4 of the proof of Proposition \ref{linear_forcing}, to obtain that 
\begin{align*}
(\partial_2 - a \diamond \partial_1^2 + 1)  (w - w^{\prime}) & = 0 &&  \textrm{in} \quad \mathbb{R}^2_+\\
w - w^{\prime}& = 0 && \textrm{on} \quad \partial \mathbb{R}^2_+.
\end{align*}
By Step 2.1 we have that $a \diamond \partial_1^2 (w- w^{\prime}) = a \partial_1^2 (w- w^{\prime}) \equiv 0$ on $\mathbb{R}^2_-$. In particular, we find that $w - w^{\prime}$ solves 
\begin{align*}
(\partial_2 - a \diamond \partial_1^2 + 1)  (w - w^{\prime}) & = 0 && \textrm{in} \quad \mathbb{R}^2,
\end{align*}
which we can then take as input into Lemma \ref{KrylovSafonov}. The proof of our claim then proceeds exactly as in Step 4 of Proposition \ref{linear_forcing} by showing that $\|w - w^{\prime}\|_{\alpha} =0$. \\

\noindent \textbf{Step 4:} \textit{(Conclusion)} \quad To conclude, we check that $\ui = q + w$ solves \eqref{IC_2}.  In this step it is important to keep in mind the notations from Definition \ref{constant_soln} and Definition \ref{q} and, additionally, the shorthand given in \eqref{vi_short} and \eqref{notation_inter}. Notice that because $q = \uf_{int} - \upa$  and $w =0$ on $\partial \mathbb{R}^2_+$, the desired boundary condition holds. Furthermore, by \eqref{correction_equn_ws} we have that
\begin{align}
\begin{split}
& \partial_2(\tilde{q} + w) - a \diamond \partial_1^2 \tilde{q} - a \diamond \partial_1^2 w + (\tilde{q} + w) \\
& = (\partial_2\tilde{q} - a \diamond \partial_1^2 \tilde{q} + \tilde{q}) - (\partial_2 q - a  \partial_1^2 q + q)^E  \quad \qquad \textrm{in} \quad \mathbb{R}^2.
\end{split}
\end{align}
To finish we show that 
\begin{align}
\label{products_linear_prop_2}
a \diamond \partial_1^2 \tilde{q} + a \diamond \partial_1^2 w  = a \diamond \partial_1^2(\tilde{q} + w)
\end{align}
and
\begin{align}
\label{classical_product_rhs}
(\partial_2\tilde{q} - a\diamond \partial_1^2 \tilde{q} + \tilde{q}) - (\partial_2 q - a \partial_1^2 q + q)^E \equiv 0 \quad \qquad \textrm{in} \quad \mathbb{R}^2_+.
\end{align}

For \eqref{products_linear_prop_2} we first notice that, since $w \in C^{2\alpha}(\mathbb{R}^2)$, it follows from Proposition \ref{ansatz_IVP} that $\tilde{q} + w$ is modelled after $\tilde{\vi}(\cdot, a_0)$ according to $a$. So, the product on the right-hand side of \eqref{products_linear_prop_2} is defined via Lemma \ref{lemma:reconstruct_2} with this modelling. The first product on the left-hand side is defined using the same modelling and the second product on the left-hand side is defined via the trivial modelling. Just like in Step 4 of Proposition \ref{linear_forcing}, we find that the triangle inequality and \eqref{Lemma4_3} of Lemma \ref{lemma:reconstruct_2} may be combined to give 
\begin{align}
\lim_{T \rightarrow 0} \| (a \diamond \partial_1^2 (\tilde{q} + w))_T - (a \diamond \partial_1^2 \tilde{q})_T  - (a \diamond \partial_1^2 w)_T \| =0,
\end{align}
which implies \eqref{products_linear_prop_2}.

To show \eqref{classical_product_rhs}, we prove that $a \diamond \partial_1^2 \tilde{q}$ is the classical product on $\mathbb{R}^2_+$. We first notice that $\partial_1^2 \tilde{q}$ satisfies \eqref{assumption_G_1}, which implies that the product $a \partial_1^2 \tilde{q}$ is well-defined in a distributional sense. For this calculation fix a point $x \in \mathbb{R}^2$; then, we may write
\begin{align}
\partial_1^2 \tilde{q}(x) = & \partial_1^2 \tilde{\vi}^{\prime}(x, \tilde{\bar{a}}(x)) + 2 \partial_1  \partial_{a_0} \tilde{\vi}^{\prime}(x, \tilde{\bar{a}}(x)) \partial_1 \tilde{\bar{a}}(x) \\
&+ \partial_{a_0}^2 \tilde{\vi}^{\prime}(x, \tilde{\bar{a}}(x))  (\partial_1 \tilde{\bar{a}} (x))^2 + \partial_{a_0} \tilde{\vi}^{\prime}(x, \tilde{\bar{a}}(x))  \partial_1^2 \tilde{\bar{a}}(x).
\end{align}
Applying Lemma \ref{semigroup_bounds} and using \eqref{edit_edit_3} along with the assumptions \hyperlink{B2}{(B2)}, \hyperlink{B3}{(B3)}, and $[a]_{\alpha} \leq N \leq 1$ yields that 
\begin{align}
 |\partial_1^2 \tilde{\vi}^{\prime}(x, \tilde{\bar{a}}(x)) | \lesssim (N_0 + N_0^{int}) |x_2|^{\frac{\alpha -2}{2}}
 \end{align}
 and 
\begin{align}
\label{better_decay}
\begin{split}
 & | \partial_1  \partial_{a_0} \tilde{\vi}^{\prime}(x, \tilde{\bar{a}}(x)) \partial_1 \tilde{\bar{a}}(x) | + |\partial_{a_0}^2 \tilde{\vi}^{\prime}(x, \tilde{\bar{a}}(x))  (\partial_1 \tilde{\bar{a}}(x) )^2| + | \partial_{a_0} \tilde{\vi}^{\prime}(x, \tilde{\bar{a}}(x))  \partial_1^2 \tilde{\bar{a}}(x)| \quad \\
 & \lesssim N (N_0 + N_0^{int}) |x_2|^{\frac{2\alpha -2}{2}}.
\end{split}
\end{align}
So, indeed $\partial_1^2 \tilde{q}$ satisfies \eqref{assumption_G_1}.

As now  $a \partial_1^2 \tilde{q}$ has a well-defined classical meaning, it makes sense to write
\begin{align}
\label{prop2_last_edit}
\begin{split}
 \lim_{T\rightarrow 0} \| (a \diamond \partial_1^2 \tilde{q})_T -  (a \partial_1^2 \tilde{q})_T\|
\lesssim &  \lim_{T\rightarrow 0} \| [ a , (\cdot)_T] \diamond \partial_1^2 \tilde{q} - E [a , (\cdot)_T] \partial_1^2 \tilde{\vi}(\cdot, a_0) \|  \\
& +   \lim_{T\rightarrow 0} \| [ a, (\cdot)_T] \partial_1^2 \tilde{q} - E [a, (\cdot)_T] \partial_1^2 \tilde{\vi}(\cdot, a_0) \|,
\end{split}
\end{align} 
where $E$ denotes evaluation of a function of $(x, a_0)$ at $(x,a(x))$. Notice that by Lemma \ref{lemma:reconstruct_2}, the first term on the right-hand side of \eqref{prop2_last_edit} vanishes. We will now show that the second term also vanishes, which finishes our argument for \eqref{classical_product_rhs}. To this end, notice that by \eqref{better_decay} we have that 
\begin{align}
\begin{split}
& \int_{\mathbb{R}^2} | a(x) - a(x-y) | \Big( |\partial_1  \partial_{a_0} \tilde{\vi}^{\prime}(x-y, \tilde{\bar{a}}(x-y)) \partial_1 \tilde{\bar{a}}(x-y) |  \\
& \hspace{3.5cm}  + |\partial_{a_0}^2 \tilde{\vi}^{\prime}(x-y, \tilde{\bar{a}}(x-y))  (\partial_1 \tilde{\bar{a}}(x-y) )^2 | \\
&  \hspace{4cm} + | \partial_{a_0} \tilde{\vi}^{\prime}(x-y, \tilde{\bar{a}}(x-y))  \partial_1^2 \tilde{\bar{a}}(x-y) | \Big) |\psi_T(y)| \, \textrm{d}y  \lesssim (T^{\frac{1}{4}})^{3\alpha -2},
\end{split}
\end{align}
for any $x \in \mathbb{R}^2$. To finish, we use that $[a]_{\alpha}\leq 1$ and the triangle inequality to write 
\begin{align}
\label{new_calc}
\begin{split}
& \hspace{-.3cm} \int_{\mathbb{R}^2} |a(x) - a(x-y)| | \partial_1^2 \tilde{\vi}^{\prime}(x-y, \tilde{\bar{a}}(x-y)) - E \partial_1^2 \tilde{\vi}(x-y, a_0) | |\psi_T(y)| \, \textrm{d} y\\
\lesssim &   \int_{\mathbb{R}^2} d^{\alpha}(0,y) | \partial_1^2 \tilde{\vi}(x-y, a(x))  - \partial_1^2 \tilde{\vi}(x-y, a(x-y)) | | \psi_T(y)| \, \textrm{d} y  \\
&+  \int_{\mathbb{R}^2} d^{\alpha}(0,y) | \partial_1^2 \tilde{\vi}(x-y, a(x-y))  - \partial_1^2 \tilde{\vi}(x-y, a_{tr}(x-y)) | | \psi_T(y)| \, \textrm{d} y \\
& +  \int_{\mathbb{R}^2} d^{\alpha}(0,y)| \partial_1^2 \tilde{\vi}^{\prime}(x-y, \tilde{\bar{a}}(x-y))  - \partial_1^2 \tilde{\vi}^{\prime}(x-y, a(x-y)) |  \, \textrm{d} y \\
& +  \int_{\mathbb{R}^2} d^{\alpha}(0,y)| \partial_1^2 \tilde{\vi}^{\prime}(x-y, a(x-y))  - \partial_1^2 \tilde{\vi}^{\prime}(x-y, a_{tr}(x-y)) |  \, \textrm{d} y \\
&+   \int_{\mathbb{R}^2} d^{\alpha}(0,y) | \partial_1^2 \tilde{\vi}(x-y, a_{tr}(x-y), u - v(a_{tr}(x-y))) | | \psi_T(y)| \,  \textrm{d} y.\\
\end{split}
\end{align}
Notice that by Lemma \ref{semigroup_bounds} and \eqref{a_bound_match_up}, the first four terms of \eqref{new_calc} are uniformly bounded (in $x$) by $(N_0 + N_0^{int}) (T^{\frac{1}{4}})^{3\alpha -2}$. For the last term we use Lemma \ref{semigroup_bound_modelling}, which gives a uniform bound of $M (T^{\frac{1}{4}})^{3\alpha -2}$, where $M$ is associated to the modelling of $u$ after $v$ according to $a$ that follows from Proposition \ref{linear_forcing}. Combining all of these observations we find that the second term of \eqref{prop2_last_edit} also vanishes as $T\rightarrow 0$.\\

\noindent $ii)$ \textbf{Step 5:} \textit{(Interpolation of the data)} \quad We linearly interpolate the data as in Proposition \ref{linear_forcing}. Notice that $a_s$ and $f_s$ have already been defined in Step 5 of Proposition \ref{linear_forcing} and that $v_s(\cdot, a_0)$ and $u_s$ corresponding to $f_s$ and $a_s$ have also been introduced. We now additionally let  
\begin{align}
\label{interpolate_Prop_2}
\quad \uf_{int,s} := s \uf_{int,1}  + (1-s)  \uf_{int,0}
\end{align}
for $s \in [0,1]$.  By Definition \ref{q}, these conventions induce the notation
\begin{align}
\label{interpolate_q_prop2}
q_s := \vi^{\prime}_s(\cdot, \bar{a}_s(\cdot)),
\end{align}
where $\vi^{\prime}_s(\cdot,a_0) = \vi(\cdot, a_0, \uf_{int,s} - u_s)$ and $\bar{a}_s$ solves \eqref{a_equn} with initial condition $a_s$.\\

\noindent \textbf{Step 6:} \textit{(A continuous curve of corrections $w_s^{\tau}$ and an equation for $\partial_s w_{s}^{\tau})$} \quad In analogue to \eqref{Step_2_conclusion}, the bounds from Lemma \ref{semigroup_bounds} yield
\begin{align}
\label{bound_holder_seminorm_interpolate}
|(\partial_2 q_s - a_s \partial_1^2 q_s + q_s )^E (x)| \lesssim N(N_0^{int} + N_0) |x_2|^{\frac{2 \alpha -2}{2}}
\end{align}
for  any point $x \in \mathbb{R}^2$. Feeding \eqref{bound_holder_seminorm_interpolate} into the machinery that we have developed in part $i)$, we find that there exists $w_s \in C^{2\alpha}(\mathbb{R}^2)$ solving \eqref{w_equation} with right-hand side $-(\partial_2 q_s - a_s \partial_1^2 q_s +q_s)$ and coefficient $a_s$ and that  $w_s$ actually solves
\begin{align}
\label{approx_equation_stab_1}
(\partial_2 - a_s \diamond \partial_1^2  + 1) w_s & =  g_s && \textrm{in} \quad \mathbb{R}^2,
\end{align}
where $g_s := -(\partial_2 q_s - a_s \partial_1^2 q_s +q_s)^E$. This solution $w_s$ is obtained by taking the limit in $C^{2\alpha}(\mathbb{R}^2)$ of the sequence of regularized solutions $w_s^{\tau}$ of 
\begin{align}
\label{approx_equation_stab_2}
(\partial_2 - a_s \diamond \partial_1^2  + 1) w^{\tau}_s & =  g_{s \tau} && \textrm{in} \quad \mathbb{R}^2.
\end{align}
By the same arguments as in Step 4, we find that $\ui_s = q_s + w_s$ solves \eqref{IC_2} with coefficients $a_s$ and initial condition $\uf_{int,s} - \upa_s$.

Since Step 2.1 implies that when $\tau>0$ the singular product in \eqref{approx_equation_stab_2} is the classical product, we may differentiate \eqref{approx_equation_stab_2} with respect to $s$  and find that $\partial_s w_{s}^{\tau}$ solves 
\begin{align}
\label{approx_equation_stab_3}
(\partial_2 - a_s \partial_1^2  + 1) \partial_s w^{\tau}_s & =   (\partial_s g_s)_{\tau} + \partial_s a_s \partial_1^2 w^{\tau}_s && \textrm{in} \quad \mathbb{R}^2.
\end{align}
By similar arguments as in Step 2, the right-hand side of \eqref{approx_equation_stab_3} is of class $C^{\alpha}$, which implies that $\partial_s w_{s}^{\tau} \in C^{\alpha +2}(\mathbb{R}^2)$. In particular, $\partial_s w_{s}^{\tau}$ is trivially modelled.\\

\noindent \textbf{Step 7:} \textit{(Estimates for $\partial_s w^{\tau}_s$)} \quad  Let $\tau \in (0,1)$. We apply Lemma \ref{KrylovSafonov} to $\partial_s w_{s}^{\tau}$ with the inputs $I = 2$, $f_1 (\cdot, a_0) = \partial_s g_s$, $f_2(\cdot, a_0) = \partial_s a_s \partial_1^2  w_{s}^{\tau}(\cdot, a_0)$, and  $\sigma_1 = \sigma_2 = 0$.  First, we check that $\partial_s w^{\tau}_{s}$ is an approximate solution in the sense of \eqref{Lemma5_3}. To begin, we convolve \eqref{approx_equation_stab_3} with $\psi_T$, which gives
\begin{align*}
( \partial_2 - a_s  \partial_1^2 +1 )\partial_s w^{\tau}_{sT} =  \partial_s ( g_{s \tau})_T + (\partial_s a_s \partial_1^2 w^{\tau}_{s} )_T  -  [ a_s , (\cdot)_T ] \partial_1^2 \partial_s  w^{\tau}_{s}  \quad \quad \quad \textrm{in} \quad  \mathbb{R}^2.
\end{align*}
The crux of the proof of part $ii)$ of Proposition \ref{linear_IVP} is showing that
\begin{align}
\label{interpolate_approx_soln}
\begin{split}
& \sup_{T\leq1}(T^{\frac{1}{4}})^{2-2\alpha} \|   \partial_s (g_{s \tau})_T + (\partial_s a_s \partial_1^2 w_{s}^{\tau} )_T  -  [a_s , (\cdot)_T ] \partial_1^2 \partial_s w^{\tau}_{s} \|\\
 & \lesssim  [a_s]_{\alpha}  [\partial_s  w_{s}^{\tau} ]_{2\alpha} +   \delta N (N_0 + N_0^{int})  + \delta N_0 + \delta N_0^{int},\\
\end{split}
\end{align}
which we split into three steps. The eventual application of Lemma \ref{KrylovSafonov} then comes in Step 7.4. \\

\noindent\textit{Step 7.1:} \quad We start by showing that  
 \begin{align}
\label{prop_2_stability_6}
 \| \partial_s g_{s\tau} \|_{2\alpha -2} \lesssim \delta N (N_0 + N_0^{int}) + \delta N_0 + \delta N_0^{int}.
\end{align}
In particular, for any $x \in \mathbb{R}^2$ the identity 
\begin{align}
\label{prop_2_stablility_6.1}
\begin{split}
\partial_s g_s(x) = & - (( \partial_2 - a_s \partial_1^2 + 1) ( \vi^{\prime}_1(x, \bar{a}_s(x)) - \vi^{\prime}_0(x, \bar{a}_s(x))) )^E\\
& - (( \partial_2 - a_s \partial_1^2 + 1)  (\overline{a_1 - a_0}) \partial_{a_0}\vi^{\prime}_s(x,\bar{a}_s(x)))^E\\
& + ((a_1 - a_0) \partial_1^2 \vi^{\prime}_s(x,\bar{a}_s(x)))^E
\end{split}
\end{align}
holds. The first term may be treated like \eqref{corollary_3_calc_1} and \eqref{approx_equation_prop_1_convolved}, using the linearity of the equations \eqref{constant_coeff_ivp_body} and \eqref{a_equn} along with the assumption \hyperlink{C3}{(C3)} and \eqref{modelling_difference_prop_1}. In conjunction with \eqref{monotonicity} and $N \leq 1$, we obtain
\begin{align}
\label{prop_2_stablility_8}
\begin{split}
 & \| ((\partial_2 - a_s \partial_1^2 +1)( \vi^{\prime}_1(\cdot, \bar{a}_s(\cdot)) - \vi^{\prime}_0(\cdot, \bar{a}_s(\cdot)))^E)_{\tau }\|_{2\alpha -2} \\
 & \lesssim N_0 \delta N + \delta N_0 + \delta N_0^{int}. \quad
 \end{split}
\end{align}
Treating the second and third terms of \eqref{prop_2_stablility_6.1} is more involved. Applying Leibniz' rule, for any $x \in \mathbb{R}^2_+$, we have that 
 \begin{align}
 \begin{split}
&( \partial_2 - a_s \partial_1^2 + 1)  (\overline{a_1 - a_0}) \partial_{a_0}\vi^{\prime}_s(x,\bar{a}_s(x)) - (a_1 - a_0) \partial_1^2 \vi^{\prime}_s(x,\bar{a}_s(x))\\
&=  \partial_{a_0}\vi^{\prime}_s(x,\bar{a}_s(x)) \partial_2  (\overline{a_1 - a_0}) +  (\overline{a_1 - a_0})  \partial_2  \partial_{a_0}\vi^{\prime}_s(x,\bar{a}_s(x))\\
& +  (\overline{a_1 - a_0})  \partial^2_{a_0}\vi^{\prime}_s(x,\bar{a}_s(x))  \partial_2 \bar{a}_s  - a_s  \partial_{a_0}\vi^{\prime}_s(x,\bar{a}_s(x)) \partial_1^2  (\overline{a_1 - a_0}) \\
& - 2 a_s  \partial_1 (\overline{a_1 - a_0}) (\partial_1 \partial_{a_0} \vi^{\prime}_s(x, \bar{a}_s(x)) + \partial_{a_0}^2 \vi^{\prime}_s(x, \bar{a}_s(x)) \partial_1 \bar{a}_s)\\
&- a_s (\overline{a_1 - a_0}) \Big(   \partial_1^2 \partial_{a_0} \vi^{\prime}_s(x, \bar{a}_s(x)) + 2  \partial_1 \partial_{a_0}^2 \vi^{\prime}_s(x, a_s(x)) \partial_1 \bar{a}_s(x) \\
& \quad \quad \quad \quad  \quad \quad \quad  \quad \quad \quad  +  \partial_{a_0}^3 \vi^{\prime}_s(x, \bar{a}_s(x)) (\partial_1 \bar{a}_s)^2 + \partial_{a_0}^2 \vi^{\prime}_s(x, \bar{a}_s(x)) \partial_1^2 \bar{a}_s \Big)\\
&+  (\overline{a_1 - a_0}) \partial_{a_0}\vi^{\prime}_s(x,\bar{a}_s(x)) \\
& -  (a_1 - a_0)\Big( -\partial_1^2 \vi^{\prime}_s(x,\bar{a}_s(x)) + 2   \partial_1 \partial_{a_0}  \vi^{\prime}_s(x,\bar{a}_s(x)) \partial_1 \bar{a}_s\\
&  \quad \quad \quad  \quad \quad \quad \quad \quad \quad \quad + \partial_{a_0}^2 \vi^{\prime}_s(x,\bar{a}_s(x)) (\partial_1 \bar{a}_s)^2 + \partial_{a_0} \vi^{\prime}_s(x,\bar{a}_s(x)) \partial_1^2 \bar{a}_s   \Big).
\end{split}
\end{align}
We re-work this identity by using the following relations:
\begin{align*}
\partial_2 (\overline{a_1 - a_0})  & =  \partial_1^2 (\overline{a_1 - a_0}) - (\overline{a_1 - a_0}),\\
 \quad  \partial_2 \bar{a}_s  & =   \partial_1^2 \bar{a}_s -\bar{a}_s,  \\
  \textrm{and} \, \,\partial_2 \partial_{a_0} \vi^{\prime}_s(x, \bar{a}_s(x)) 
 & =  \bar{a}_s(x) \partial_1^2 \partial_{a_0} \vi^{\prime}_s(x, \bar{a}_s(x)) - \partial_{a_0} \vi^{\prime}_s(x, \bar{a}_s(x))+ \partial_1^2 \vi^{\prime}_s(x, \bar{a}_s(x)).
\end{align*} 
These are plugged-in to obtain that
\begin{align*}
\begin{split}
& ( \partial_2 - a_s \partial_1^2 + 1)  (\overline{a_1 - a_0}) \partial_{a_0}\vi^{\prime}_s(x,\bar{a}_s(x)) - (a_1 - a_0) \partial_1^2 \vi^{\prime}_s(x,\bar{a}_s(x))\\
&=  (1 - a_s )  \partial_{a_0}\vi^{\prime}_s(x,\bar{a}_s(x))  \partial_1^2 (\overline{a_1 - a_0}) +  (\overline{a_1 - a_0}) (\bar{a}_s(x) - a_s(x))\partial_1^2 \partial_{a_0} \vi^{\prime}_s(x, \bar{a}_s(x)) \\
&+  (a_1 - a_0 - \overline{a_1 - a_0}  ) \partial_1^2 \vi^{\prime}_s(x, \bar{a}_s(x)) +  (\overline{a_1 - a_0})  \partial^2_{a_0}\vi^{\prime}_s(x,\bar{a}_s(x))  ( \partial_1^2 \bar{a}_s -\bar{a}_s)\\
& - 2 a_s  \partial_1 (\overline{a_1 - a_0}) ( \partial_1 \partial_{a_0} \vi^{\prime}_s(x, \bar{a}_s(x)) + \partial_{a_0}^2 \vi^{\prime}_s(x, \bar{a}_s(x)) \partial_1 \bar{a}_s)\\
& - a_s (\overline{a_1 - a_0}) \Big( 2  \partial_1 \partial_{a_0}^2 \vi^{\prime}_s(x, a_s(x)) \partial_1 \bar{a}_s(x) +  \partial_{a_0}^3 \vi^{\prime}_s(x, a_s(x)) (\partial_1 \bar{a}_s)^2 + \partial_{a_0}^2 \vi^{\prime}_s(x, a_s(x)) \partial_1^2 \bar{a}_s \Big)\\
& -  (a_1 - a_0) \Big(   2  \partial_1 \partial_{a_0}  \vi^{\prime}_s(x,\bar{a}_s(x)) \partial_1 \bar{a}_s    + \partial_{a_0}^2\vi^{\prime}_s(x,\bar{a}_s(x)) (\partial_1 \bar{a}_s)^2+ \partial_{a_0} \vi^{\prime}_s(x,\bar{a}_s(x)) \partial_1^2 \bar{a}_s  \Big).
\end{split}
\end{align*}

Each term on the right-hand side of the above expression is now treated separately. In particular, using the bounds from Lemma \ref{semigroup_bounds}, that the initial condition of $\vi^{\prime}_s(\cdot, a_0)$ does not depend on $a_0$, the relation \eqref{a_difference}, and the linearity of the equations \eqref{constant_coeff_ivp_body} and \eqref{a_equn}, we obtain the following estimates:
\begin{align}
& | (1 - a_s(x) )  \partial_{a_0}\vi^{\prime}_s(x,\bar{a}_s(x))  \partial_1^2 (\overline{a_1 - a_0})(x)|\lesssim (1 + \|a_s\|) [U_{int,s} - u_s]_{\alpha} [a_1 - a_0]_{\alpha}  x_2^{\frac{2\alpha -2}{2}},\\
& |  (\overline{a_1 - a_0})(x) (\bar{a}_s- a_s)(x) \partial_1^2 \partial_{a_0} \vi^{\prime}_s(x, \bar{a}_s(x)) 
|  \lesssim \|a_1 - a_0\| [a_s]_{\alpha}[U_{int,s} - u_s]_{\alpha}  x_2^{\frac{2\alpha -2}{2}}, \\
& | (\overline{a_1 - a_0} -  (a_1 - a_0) )(x) \partial_1^2 \vi^{\prime}_s(x, \bar{a}_s(x))|
\lesssim  [a_1 - a_0]_{\alpha} [U_{int,s} - u_s]_{\alpha} x_2^{\frac{2\alpha -2}{2}},\\
& | \overline{a_1 - a_0}(x) \partial_{a_0}^2  \vi^{\prime}_s(x, \bar{a}_s(x)) ( \partial_1^2 \bar{a}_s -\bar{a}_s)(x)|
\lesssim    \| a_1 - a_0\| [U_{int,s} - u_s]_{\alpha}  x_2^{\frac{\alpha}{2}} ( [a_s]_{\alpha} x_2^{\frac{\alpha -2}{2}} \hspace{-.1cm}+ \hspace{-.1cm} \|a_s\|),\\
& | a_s(x)  \partial_1 \overline{a_1 - a_0}(x) \partial_1 \partial_{a_0}  \vi^{\prime}_s(x, \bar{a}_s(x))|  \lesssim \|a_s\| [a_1 - a_0]_{\alpha} [U_{int,s} - u_s]_{\alpha} x_2^{\frac{2\alpha -2}{2}}\\
& |  \partial_1 \overline{a_1 - a_0}(x) \partial_{a_0}^2  \vi^{\prime}_s(x, \bar{a}_s(x)) \partial_1 \bar{a}_s(x))|  \lesssim    [a_s]_{\alpha}  [a_1 - a_0]_{\alpha} \|U_{int,s} - u_s\| x_2^{\frac{2\alpha -2}{2}},\\
& | a_s(x) \overline{a_1 - a_0}(x)  \partial_1 \partial_{a_0}^2 \vi^{\prime}_s(x, a_s(x)) \partial_1 \bar{a}_s(x)|\lesssim \|a_s\| [a_s]_{\alpha}\|a_1 - a_0\| [U_{int,s} - u_s]_{\alpha} x_2^{\frac{2\alpha -2}{2}},\\
& |  a_s(x) \overline{a_1 - a_0}(x)  \partial_{a_0}^3 \vi^{\prime}_s(x, a_s(x)) (\partial_1 \bar{a}_s(x))^2|   \lesssim  \|a_s\| \|a_1 - a_0\| \|U_{int,s} - u_s\|  [a_s]_{\alpha}^2  x_2^{\frac{2\alpha -2}{2}},\\
&  |  a_s(x) \overline{a_1 - a_0}(x) \partial_{a_0}^2 \vi^{\prime}_s(x, a_s(x)) \partial_1^2 \bar{a}_s(x)| \lesssim  \|a_s\| \|a_1 - a_0\| [U_{int,s} - u_s] _{\alpha} [a_s]_{\alpha} x_2^{\frac{2\alpha -2}{2}}, \\
& |(a_1 - a_0)(x)  \partial_1 \partial_{a_0}  \vi^{\prime}_s(x,\bar{a}_s(x)) \partial_1 \bar{a}_s(x)| \lesssim \|a_1 -a_0\|  [U_{int,s} - u_s] _{\alpha}  [a_s]_{\alpha} x_2^{\frac{2\alpha -2}{2}},\\
& |(a_1 - a_0)(x) \partial_{a_0}^2 \vi^{\prime}_s(x,\bar{a}_s(x)) (\partial_1 \bar{a}_s(x))^2 | \lesssim \|a_1 - a_0\|  [a_s]_{\alpha}^2  \|U_{int,s} - u_s\| x_2^{\frac{2\alpha -2}{2}}, \\
\textrm{and} & \quad  | (a_1 - a_0)(x)  \partial_{a_0} \vi^{\prime}_s(x,\bar{a}_s(x)) \partial_1^2 \bar{a}_s(x) |   \lesssim \|a_1 -a_0\|  [U_{int,s} - u_s] _{\alpha}  [a_s]_{\alpha} x_2^{\frac{2\alpha -2}{2}}.
\end{align}
Combining these estimates with the assumptions \hyperlink{C2}{(C2)},  \hyperlink{C3}{(C3)}, and \hyperlink{C4}{(C4)}, along with the previous estimate \eqref{edit_edit_3}, we find that the second and third terms of \eqref{prop_2_stablility_6.1} are bounded as $\delta N (N_0 + N_0^{int})  (x_2^{\frac{2\alpha -2}{2}} + x_2^{\frac{\alpha}{2}})$. Applying Lemma \ref{lemma:norm} and using \eqref{prop_2_stablility_8}, we then obtain \eqref{prop_2_stability_6}. In our application of Lemma \ref{lemma:norm}, we remark that the term $|x_2|^{\frac{\alpha}{2}}$ is not disturbing as $(T^{\frac{1}{4}})^{2 \alpha -2 } \leq (T^{\frac{1}{4}})^{\alpha}$ when $T \in (0,1)$. \\

\noindent \textit{Step 7.2:} \quad To continue checking \eqref{interpolate_approx_soln} we use the triangle inequality to write
\begin{align}
\label{prop_2_stability_19}
\| \partial_s a_s \partial_1^2 w_{s}^{\tau} \|_{2\alpha -2} 
\leq   \| [\partial_s a_s, (\cdot) ] \partial_1^2 w_{s}^{\tau}\|_{2\alpha -2}  +   \sup_{T\leq 1 }(T^{\frac{1}{4}})^{2-2\alpha} \| \partial_s a_s \partial_1^2 (w_{s}^{\tau})_T\|. \quad 
\end{align}
The first term is treated with \eqref{Lemma4_4} of Lemma \ref{lemma:reconstruct_2}, the analogue of \eqref{whole_space_holder_prop_2} for $w_s^{\tau}$, and assumption  \hyperlink{C2}{(C2)}, which yield that 
\begin{align}
\label{prop_2_stability_20}
\| [\partial_s a_s, (\cdot)] \partial_1^2 w_{s}^{\tau}\|_{2\alpha -2}  \lesssim \delta N [ w_{s}^{\tau}]_{2\alpha}
 \lesssim \delta N  (N_0 + N_0^{int}).
\end{align}
The second term of \eqref{prop_2_stability_19} is also handled using \eqref{whole_space_holder_prop_2}. In particular, for any $x \in \mathbb{R}^2$ we can use \eqref{moment_bound}, that $\psi_T$ is an even Schwartz function, that $N\leq 1$, and assumption \hyperlink{C4}{(C4)} to obtain
\begin{align*} 
 & |\partial_s a_s (\partial_1^2 w_{s}^{\tau})_T (x)| \lesssim   \| a_0 - a_1 \|  \Big| \int_{\mathbb{R}^2} (w_{s}^{\tau}(y) - w_{s}^{\tau}(x) - \partial_1 w_{s}^{\tau}(x) (y-x)_1) \partial_1^2 \psi_T(y-x) \, \textrm{d}y \Big|\\
& \lesssim  \| a_0 - a_1 \| [w_{s}^{\tau}]_{2\alpha} (T^{\frac{1}{4}})^{2\alpha -2} \lesssim  \delta N  (N_0 + N_0^{int}) (T^{\frac{1}{4}})^{2\alpha -2}.
\end{align*}

\noindent \textit{Step 7.3:} \quad To finish checking \eqref{interpolate_approx_soln}, we again use \eqref{Lemma4_4} of Lemma \ref{lemma:reconstruct_2} for
\begin{align}
\label{prop_2_stability_23}
\begin{split}
\| [a_s, (\cdot)] \partial_1^2 \partial_s w_{s}^{\tau}\|_{2\alpha -2} \lesssim  [a_s]_{\alpha}[ \partial_s w_{s}^{\tau}]_{2\alpha}.
\end{split}
\end{align}
This completes the argument for \eqref{interpolate_approx_soln}.\\

\noindent \textit{Step 7.4:}\quad  Having shown \eqref{interpolate_approx_soln} and using  $[a_s]_{\alpha} \ll 1$, we can then apply Lemma \ref{KrylovSafonov} to find that
\begin{align}
\label{Prop_2_ stability_KS_result}
 \| \partial_s w^{\tau}_{s} \|_{\alpha}  + [ \partial_s w^{ \tau}_{s} ]_{2\alpha} \lesssim  \delta N (N_0 + N_0^{int})+ \delta N_0 + \delta N_0^{int}.
\end{align}

\vspace{.4cm}

\noindent \textbf{Step 8:} \textit{(Conclusion)} \quad Just as in Step 8 of Proposition \ref{linear_forcing}, the bound \eqref{Prop_2_ stability_KS_result} may be integrated-up to give:
\begin{align}
\label{prop_2_stability_24}
  \| w^{\tau}_{0} - w^{\tau}_{1} \|_{\alpha} + [w^{\tau}_{0} - w^{\tau}_{1} ]_{2\alpha}
  \lesssim  \delta N (N_0 + N_0^{int})+ \delta N_0 + \delta N_0^{int}.
\end{align}
Passing to the limit $\tau \rightarrow 0$, we find that the bound \eqref{prop_2_stability_24} holds also for $w_0 - w_1$. 
\end{proof}

\subsection{Proof of Theorem \ref{linear_theorem}}
\label{section:proof_theorem_1}

\begin{proof} As already advertised, the proof of Theorem \ref{linear_theorem} consists of combining Propositions \ref{linear_forcing}, \ref{ansatz_IVP}, and \ref{linear_IVP}.\\

\noindent \textit{i)} From Proposition \ref{linear_forcing} we have a unique solution $\upa \in C^{\alpha}(\mathbb{R}^2)$ of \eqref{forcing_theorem_1}  that is modelled after $v$ according to $a$. In Proposition \ref{ansatz_IVP} we take $u$ to be this solution of \eqref{forcing_theorem_1}. By \eqref{edit_edit_3}, \eqref{modelling_q_constant}, \eqref{holder_q_constant}, and \eqref{nu_bound_1} from the proof of Lemma \ref{lemma:reconstruct_1} --where $M_{\partial}$ and $\nu_{\partial}$ correspond to the modelling of $u$-- we find that 
\begin{align*}
M_{q} + \|q\|_{\alpha} \lesssim N_0 + N_0^{int}.
\end{align*}
An application of Proposition \ref{linear_IVP} then gives a unique $w \in C^{2\alpha}(\mathbb{R}^2)$ such that $w \equiv 0$ on $\mathbb{R}^2_-$ and $\ui = q + w$ solves \eqref{IVP_theorem_1}. The desired solution $\uf \in C^{\alpha}(\mathbb{R}^2_+)$ of \eqref{theorem_1_main_ivp} is then given by the restriction of $\uf = \upa + \ui$. 
 
To check that $\upa + \ui$ in fact satisfies \eqref{theorem_1_main_ivp}, we show that
 \begin{align}
  \label{theorem_1_3}
a \diamond \partial_1^2 \uf =  a \diamond \partial_1^2 \upa + a \diamond \partial_1^2 \ui.
 \end{align}
The argument for \eqref{theorem_1_3} has already been used in Step 1 of Proposition \ref{linear_forcing} and Step 2.1 of Proposition \ref{linear_IVP}. In particular, Lemma \ref{lemma:reconstruct_2} and the triangle inequality yield
\begin{align*}
& \lim_{T \rightarrow 0 } \|  (a \diamond \partial_1^2  \uf )_T - (a \diamond \partial_1^2 \upa)_T  +  (a \diamond \partial_1^2 \ui)_T \| = 0.
\end{align*}
The relations \eqref{theorem_1_i_holder} and \eqref{theorem_1_i_modelling} are a consequence of \eqref{edit_edit_3}, \eqref{modelling_q_constant}, and \eqref{modelling_estimate_ivp_w}.\\

\noindent \textit{ii)} We now use the results of part \textit{ii)} of Propositions \ref{linear_forcing}, \ref{ansatz_IVP}, and \ref{linear_IVP}. In particular, for $u_0$ and $u_1$ in part $ii)$ of Proposition \ref{ansatz_IVP} we take the solutions from part $ii)$ of Proposition \ref{linear_forcing}. Using \eqref{modelling_difference_prop_1}, \eqref{modelling_q_constant_difference}, \eqref{holder_q_constant_difference}, and \eqref{nu_bound_1} --where $\delta M_{\partial}$ and $\delta \nu_{\partial}$ correspond to the modelling of $u_1 - u_0$-- we find that 
\begin{align*}
M_{q_1 - q_0} + \|q_1 - q_0\|_{\alpha} \lesssim \delta N( N_0 + N_0^{int}) +  \delta N_0 + \delta N_0^{int}.
\end{align*}
Then relations \eqref{theorem_1_ii_holder} and  \eqref{theorem_1_ii_modelling} are immediate from  the above bound, \eqref{modelling_difference_prop_1}, and \eqref{modelling_prop_2_ii}.  
\end{proof}

\section{Proof of Theorem \ref{theorem_nonlinear}: Analysis of the quasilinear problem}
\label{section:treat_nonlinear}

\begin{proof}
We work under the assumptions of part $ii)$. The main idea of the proof is to do a contraction mapping argument for
\begin{align}
\label{map}
\left( u_i^*,  w_i^*, a_i^* \right) \mapsto ( q_i^*,    a_i:= a(u_i^* + w_i^* + \tilde{q}_i^*),\{  a_i \diamond \partial_1^2 v_i (\cdot, a_0)  \}) \stackrel{Thm. \ref{linear_theorem}}{\longmapsto}(u_i, w_i, a_i ), \qquad
\end{align}
\noindent where $u_i^* \in C^{\alpha}(\mathbb{R}^2)$ is modelled after $v_i$ according to $a_i^*\in C^{\alpha}(\mathbb{R}^2)$ and $w_i^* \in C^{2\alpha}(\mathbb{R}^2)$ such that $w_i^* \equiv 0$ on $\mathbb{R}^2_-$. We make the additional assumption that $a_i^* = a(U_{int,i})$ on $\left\{ x_2 = 0\right\}$ and $a_i^*, u_i^*,$ and $w_i^*$ are $x_1$-periodic. We, furthermore, use the convention
\begin{align}
\label{q_star}
q_i^*:= \vi(\cdot, \overline{a}_i, \uf_{int,i} -u_i^*),
\end{align}
where we have made use of Definition \ref{constant_soln} and $\overline{a}_i$ solves \eqref{a_equn} with the initial condition $a(U_{int,i})$. We also use the notation given in \eqref{vi_short} and that in Definition \ref{extensions} to denote even-reflection.\\

\noindent \textbf{Step 1:} \textit{(Application of Lemma 3.2 of \cite{OW}; see Section \ref{section:reconstruct})} \quad Let $i =0,1$. We introduce the notation
\begin{align}
\begin{split}
\label{M_star_definition}
M^*  := & \displaystyle\max_{i=0,1} \left( M_{u^*_i} +  [w_i^*]_{2\alpha} + \| u^*_i\|_{\alpha} + \| w_i^*\|_{\alpha}  \right) + N_0 + N_0^{int}\\
\textrm{and} \quad  \delta M^*  := & M_{u_1^* - u^*_0} +  [w_1^* - w_0^*]_{2\alpha} + \| u^*_1 - u^*_0 \|_{\alpha}   + \| w^*_1 - w^*_0\|_{\alpha} \\
 & + (\displaystyle\max_{i=0,1} \|u_i^*\|_{\alpha} + N_0 + N_0^{int}) \| a^*_1 - a^*_0 \|_{\alpha} + \delta N_0 + \delta N_0^{int},
 \end{split}
 \end{align}
where $M_{u^*_i} $ corresponds to the modelling of $u^*_i$ after $v_i$ according to $a_i^*$ and $M_{u_1^* - u^*_0}$ is associated to the modelling of $u^*_1 - u^*_0$ after $( v_1, -v_0 )$ according to $(a^*_1, a^*_0)$.

We also define:
\begin{align}
U_i^* := & u_i^* + w_i^* + \tilde{q}_i^*, \label{U_star_definition}\\
\tilde{M}  := & \displaystyle\max_{i=0,1} \left( M_{a_i} + [a_i]_{\alpha} \right) + N_0 + N_0^{int} ,  \quad  \textrm{and}\label{M_tilde_definition}\\
\delta \tilde{M}  := &M_{a_1 - a_0} + \| a_1- a_0\|_{\alpha} \\
& + (N_0 + N_0^{int})( \| a^{\prime}(\uf^*_1) -  a^{\prime}(\uf^*_0) \|_{\alpha} + \|a^*_1 - a^*_0\|_{\alpha} )  + \delta N_0 + \delta N_0^{int} , \label{M_tilde_diff_definition}
\end{align}
where $M_{a_i}$ corresponds to the modelling of $a_i$ after $\tilde{\vi}_i + v_i$ according to $a^*_i$ and $\mu_i = a^{\prime}(\uf^*_i)$ and $M_{a_1 - a_0}$ is associated to the modelling of $a_1 - a_0$ after  $(\tilde{\vi}_1 + v_1, \tilde{\vi}_0 + v_0 )$ according to $(a^*_1,a^*_0)$ and $(\mu_1, -\mu_0)$.

Using the bounds from \cite[Lemma 3.2]{OW} and the assumptions on the nonlinearity $a$, we then find that 
\begin{align}
& a_i \in [\lambda, 1] \textrm{ and } [a_i]_{\alpha} \ll1 \quad  \textrm{ if }  \quad  \displaystyle\max_{i=0,1}(\|u^*_i\|_{\alpha} + \|w^*_i\|_{\alpha}) \ll 1 \textrm{ and } N_0, N_0^{int} \ll 1,  \label{smooth_transform_1}\\
& \tilde{M} \lesssim M^* \hspace{2.6cm} \textrm{if } \quad \displaystyle\max_{i=0,1}(\|u^*_i\|_{\alpha} + \|w^*_i\|_{\alpha}) \ll 1 \textrm{ and } N_0, N_0^{int} \ll 1, \label{smooth_transform_2} \\
\textrm{and} \quad & \delta \tilde{M} \lesssim \delta M^* \hspace{2.3cm} \textrm{if } \quad M^* \leq 1. \label{smooth_transform_3}
\end{align}

For \eqref{smooth_transform_1} we notice that
\begin{align}
\label{holder_a_fixed}
[a_i]_{\alpha} \lesssim \|a^{\prime}\| ([u_i^*]_{\alpha} + [w_i^*]_{\alpha} + [\tilde{q}_i]_{\alpha})
 \end{align}
and $\|a^{\prime}\| \leq 1$. So, since $a_i \in [\lambda, 1]$ is clear as $a \in [\lambda,1]$,  \eqref{smooth_transform_1} holds if $[u_i^*]_{\alpha}+ [w_i^*]_{\alpha} + [\tilde{q}^*_i]_{\alpha} \ll 1$. By \eqref{holder_q_constant} and the notation \eqref{q_star} we know that $[\tilde{q}^*_i]_{\alpha}  \ll 1$ if $N_0^{int}+\|u_i^*\|_{\alpha}\ll1$.

For \eqref{smooth_transform_2} and \eqref{smooth_transform_3} we first observe that $U_i^*$ (defined in \eqref{U_star_definition}) is modelled after $\tilde{\vi}_i + v_i$ according to $a_i^*$. This follows from the assumed modelling of $u_i^*$ and noticing that $q_i^*$ is modelled after $\tilde{\vi}_i$ according to $a_i^*$ with modelling constant $M_{q^*_i}$ bounded as 
\begin{align}
\label{q_modelling_bound_theorem_2}
M_{q^*_i} \lesssim M_{u_i^*} + \|u_i^*\|_{\alpha} + N_0 + N_0^{int},
\end{align}
which we must still show. In particular, the modelling of $q_i^*$ and the bound \eqref{q_modelling_bound_theorem_2} follow from taking $a = a_i$ in part $i)$ of Proposition \ref{ansatz_IVP} --this yields that $q_i^*$ is modelled after $\tilde{\vi}_i$ according to $a_i$ with corresponding modelling constant $M_{intermediate}$ bounded as 
\begin{align}
M_{intermediate} \lesssim M_{u_i^*} + \|u^*_i\|_{\alpha} + N_0+N_0^{int}, 
\end{align} 
where we have additionally used \eqref{nu_bound_1} and \hyperlink{B1}{(B1)}. Using Lemma \ref{post_process}, since $a(U_{int,i}) = a_i^*$ on $\partial \R^2_+$, we obtain \eqref{q_modelling_bound_theorem_2}. From \eqref{holder_q_constant} we also obtain that
\begin{align}
\label{q_holder_theorem_2}
\quad \| q^*_i\|_{\alpha}   \lesssim \|u^*_i\|_{\alpha} + N_0^{int}.
\end{align}
We can then combine \eqref{q_modelling_bound_theorem_2} and \eqref{q_holder_theorem_2} with the bound \eqref{new_modelling}; $\|a^{\prime}\|,\|a^{\prime \prime}\| \leq 1$; and the assumptions of \eqref{smooth_transform_2} to write
\begin{align}
\label{Ma_modelling}
\begin{split}
M_{a_i} & \lesssim  \| a^{\prime}\| M_{U_i^*}  + \| a^{\prime \prime}\|[U_i^*]^2\\
& \lesssim \| a^{\prime}\| (M_{u_i^*}  + [w_i^*]_{2\alpha}+ M_{q^*_i}) + \| a^{\prime \prime}\|([u^*_i]_{\alpha} + [w_i^*]_{\alpha} + [q^*_i]_{\alpha})^2\\
& \lesssim  M_{u_i^*}  + [w_i^*]_{2\alpha} + \|u_i^*\|_{\alpha} + [w_i^*]_{\alpha} + N_0 + N_0^{int}.
\end{split}
\end{align}
Another application of \eqref{holder_a_fixed} and \eqref{q_holder_theorem_2} yields \eqref{smooth_transform_2}.

The bound \eqref{smooth_transform_3} requires the use of both \eqref{new_modelling_2} and \eqref{Calpha_norm_new}. First, however, we collect the bounds stemming from \eqref{modelling_q_constant_difference} and \eqref{holder_q_constant_difference}. In particular, we first notice that by \eqref{modelling_q_constant_difference}  we have that
\begin{align}
\label{q_modelling_difference_theorem_2_2}
\begin{split}
&M_{q_1^* - q_0^*} 
 \lesssim M_{u_1^* - u_0^*} + \|u_1^* - u_0^*\|_{\alpha} + \| a^*_1 - a^*_0\|_{\alpha} (\displaystyle\max_{i=0,1} \|u_i^*\|_{\alpha} + N_0^{int}) + \delta N_0^{int} + \delta N_0,
\end{split}
\end{align}
where this corresponds to the modelling of $q_1^* - q_0^*$ after $(\tilde{\vi}_1, -\tilde{\vi}_0)$ according to $(a_1^*, a_0^*)$. Applying \eqref{holder_q_constant_difference}, we obtain
\begin{align}
\label{q_holder_difference_theorem_2}
\|q_1^* - q_0^*\|_{\alpha} & \lesssim  \|u_1^* - u_0^*\|_{\alpha} + \| a^*_1 - a^*_0\|_{\alpha} (\displaystyle\max_{i=0,1} \|u_i^*\|_{\alpha} + N_0^{int}) + \delta N_0^{int}.
\end{align}
Combining \eqref{q_modelling_bound_theorem_2}, \eqref{q_holder_theorem_2}, \eqref{q_modelling_difference_theorem_2_2}, \eqref{q_holder_difference_theorem_2}, \eqref{new_modelling_2}, the assumption $M^* \leq 1$, and that $\|a^{\prime}\|$, $\|a^{\prime \prime}\|$, $\|a^{\prime \prime \prime}\|\leq 1$, we obtain
\begin{align}
\begin{split}
  M_{a_1 -a_0}
& \lesssim \|a^{\prime}\| M_{U_1^* - U_0^*} + \|U_1^* - U_0^* \|_{\alpha}(\|a^{\prime \prime} \| \displaystyle\max_{i=0,1} [U_i^*]_{\alpha} +  \frac{1}{2} \|a^{\prime \prime \prime} \| \displaystyle\max_{i=0,1} [U_i^*]^2_{\alpha} + \|a^{\prime \prime} \|  \displaystyle\max_{i=0,1} M_{U_i^*} )\\
& \lesssim  M_{u_1^* - u^*_0} + [w_1^* - w_0^*]_{2\alpha} + \|u_1^* - u_0^*\|_{\alpha} + \|w_1^* - w_0^*\|_{\alpha} + \delta N_0^{int} + \delta N_0\\
&   \quad  + \|a^*_1 - a^*_0\|_{\alpha}(  \displaystyle\max_{i=0,1} \|u_i^*\|_{\alpha}+ N_0^{int}). 
\end{split}
\end{align}
Using \eqref{Calpha_norm_new}, \eqref{q_holder_difference_theorem_2}, and $M^* \leq 1$ we find that 
\begin{align}
\label{coefficients_diff_thm_2}
\begin{split}
&  \| a_1-  a_0 \|_{\alpha} + \| a^{\prime}(\uf^*_1) -  a^{\prime}(\uf^*_0) \|_{\alpha} \\
 & \lesssim  (\|a^{\prime}\| + \|a^{\prime\prime}\| +( \|a^{\prime \prime} \|+  \|a^{\prime \prime\prime} \|  )\| \max_{i=0,1} [ \uf^*_i ]_{\alpha}) \| \uf^*_1- \uf^*_0\|_{\alpha}\\
 & \lesssim   \| u_1^* - u_0^*\|_{\alpha}+ \| w_1^* - w_0^*\|_{\alpha} + \delta N_0^{int} +\|a^*_1 - a^*_1\|_{\alpha} (\displaystyle\max_{i=0,1} \|u_i^*\|_{\alpha} + N_0^{int}).
 \end{split}
\end{align}
Combining the last two computations, we obtain \eqref{smooth_transform_3}.\\

\noindent \textbf{Step 2:} \textit{(Application of Corollary \ref{post_process_reconstruct_1})} \quad  In this step we apply Corollary \ref{post_process_reconstruct_1}. For $i,j = 0,1$, we obtain families of distributions $\left\{ a_i \diamond \partial_1^2  \vp_j (\cdot, a_0)\right\}_{a_0 \in [\lambda,1]}$, satisfying 
\begin{align}
& \hspace{-.2cm}  \| [a_i, (\cdot)] \diamond \partial_1^2 v_j(\cdot, a_0) \|_{2\alpha -2,2} \lesssim N_0 (N_0^{int} + N_0 + M_{a_i})  \lesssim N_0 \tilde{M}, \label{cor_3_use_1}\\
& \hspace{-.2cm} \| [a_i, (\cdot)] \diamond \partial_1^2 v_1(\cdot, a_0) -  [a_i, (\cdot)] \diamond \partial_1^2 v_0(\cdot, a_0)\|_{2\alpha -2,1} \lesssim \delta N_0 ( N_0 + N_0^{int}+ M_{a_i} )  \lesssim \delta N_0 \tilde{M},  \label{cor_3_use_2}\\
& \hspace{-.2cm}  \textrm{and }   \| [a_1, (\cdot)] \diamond \partial_1^2 v_i(\cdot, a_0) - [a_0, (\cdot)] \diamond \partial_1^2  v_i(\cdot, a_0)\|_{2\alpha -2,1}\\
& \quad \quad \lesssim N_0 \Big(M_{a_1 - a_0}  + (N_0 + N_0^{int})(\|a^{\prime}(\uf^*_1) -a^{\prime}(\uf^*_0)\|_{\alpha} + \| a^*_1 - a^*_0\|_{\alpha} )+\delta N_0 + \delta N_0^{int} \Big) \\
& \quad \quad \lesssim N_0 \delta \tilde{M}.\label{cor_3_use_3}
\end{align}
Notice that \eqref{cor_3_use_1} follows from \eqref{conclusion_lemma_2_1}, \eqref{cor_3_use_2} follows from \eqref{Lemma_3_3_1_1}, and \eqref{cor_3_use_3} follows from \eqref{conclusion_lemma_2_ii} via the additional ingredient of either the definition \eqref{M_tilde_definition} or \eqref{M_tilde_diff_definition}.\\

\noindent \textbf{Step 3:} \textit{(Application of Theorem \ref{linear_theorem})} \quad  As indicated in \eqref{map},  for $i =0,1$, we now apply Theorem \ref{linear_theorem} with $a_i := a(\uf^*_i)$, initial condition $U_{int,i}$, and forcing $f_i$. We use the convention that $U_i = u_i + q_i + w_i$ and the notation
\begin{align}
\begin{split}
M & := \displaystyle\max_{i=0,1} (M_{u_i} + [w_i]_{2\alpha} + \|u_i\|_{\alpha} + \| w_i \|_{\alpha}) + N_0 + N_0^{int}\\
\delta M & := M_{u_1 - u_0} + [w_1 - w_0]_{2\alpha} + \| u_1 - u_0\|_{\alpha} + \| w_1 - w_0\|_{\alpha}\\
& \quad \quad \quad   + (N_0 + N_0^{int}) \|a_1 - a_0\|_{\alpha} + \delta N_0 + \delta N_0^{int}.
\end{split}
\end{align}

To apply the first part of Theorem \ref{linear_theorem} we work under the assumption that $M^* \ll 1$: Assumptions \hyperlink{B1}{(B1)} and \hyperlink{B3}{(B3)} are verified as they are adopted into the assumptions on the inputs $f_i$ and $U_{int,i}$; the assumption \hyperlink{B2}{(B2)} is verified by $a_i$ using \eqref{smooth_transform_1} and the assumption $M^* \ll 1$; and the existence of the appropriate offline products in \hyperlink{B4}{(B4)} is guaranteed by \eqref{cor_3_use_1} of the previous step with $N = \tilde{M}$, which via \eqref{smooth_transform_2} satisfies $\tilde{M} \lesssim M^* \ll1$ and, therefore, $\tilde{M} \leq1$. The relations \eqref{theorem_1_i_holder} and \eqref{theorem_1_i_modelling} then give that 
\begin{align}
 M & \lesssim  N_0 + N_0^{int} \quad \quad \quad  \textrm{if} \quad M^* \ll1 \label{thm_2_3_1}.
\end{align}

To apply the second part of Theorem \ref{linear_theorem} we again work under the assumption that $M^*\ll 1$. The conditions \hyperlink{C1}{(C1)} and \hyperlink{C3}{(C3)} are again automatically verified since they have been adopted into the assumptions of Theorem \ref{theorem_nonlinear}. For the assumptions \hyperlink{C2}{(C2)} and \hyperlink{C4}{(C4)} we set $\delta N = \delta \tilde{M}$, which is a valid choice for \hyperlink{C4}{(C4)} by \eqref{cor_3_use_3}, and notice that by \eqref{coefficients_diff_thm_2} we have that 
\begin{align}
\| a_1 - a_0\|_{\alpha} \lesssim \delta \tilde{M}.
\end{align} 
By \eqref{theorem_1_ii_holder} and \eqref{theorem_1_ii_modelling} we obtain
\begin{align}
 \label{thm_2_3_2}
\begin{split}
 \delta M & \lesssim (N_0 + N_0^{int}) \delta \tilde{M} + \delta N_0 + \delta N_0^{int}\\
 & \lesssim   (N_0 + N_0^{int}) \delta M^* + \delta N_0 + \delta N_0^{int}  \quad \quad \quad \textrm{ if } \quad  M^* \ll1,
\end{split}
\end{align}
where we have additionally used \eqref{smooth_transform_3}.\\

\noindent \textbf{Step 4:} \textit{(Fixed-point argument)}\quad  We now let $U_{int,1} = U_{int,0}$ and $f_0 = f_1$, which implies that $\delta N_0 = \delta N_0^{int} = 0$.  We will perform a fixed-point argument for the map given in \eqref{map} in the space of triples $(u_i^*, w_i^*, a_i^*)$ as described following \eqref{map} and, furthermore, satisfying 
\begin{align}
\label{condition_set}
M^* \leq \epsilon
\end{align}
for some $\epsilon>0$. By \eqref{thm_2_3_1} we see that the set defined through \eqref{condition_set} is mapped to itself under \eqref{map} for $\epsilon \ll1$. Using the same argument as in \cite{OW}, we find that 
\begin{align}
&d((u_1, w_1 , a_1), (u_0, w_0, a_0))\\
&:= M_{u^*_1 - u^*_0} + [w_1 - w_0]_{2\alpha} + \| u^*_1 - u^*_0\|_{\alpha}  + \|w_1 - w_0 \|_{\alpha} + (N_0 + N_0^{int}) \| a_1 - a_0\|_{\alpha}
\end{align}
defines a distance function under which the set defined by \eqref{condition_set} is complete and closed. By \eqref{thm_2_3_2} with $\delta N_0 = \delta N_0^{int} = 0$, we obtain that $\delta M \lesssim (N_0 + N_0^{int}) \delta M^*$, which translates into:
\begin{align}
\label{contraction}
d((u_1, w_1 , a_1), (u_0, w_0, a_0))\lesssim (N_0 + N_0^{int}) d((u_1^*, w_1^* , a^*_1), (u^*_0, w^*_0, a^*_0)).
\end{align}
In other words, the map given by \eqref{map} is a contraction on the space defined by \eqref{condition_set}.\\

\noindent \textbf{Step 5:} \textit{(Conclusion)} \quad We first conclude part $i)$. Notice that the fixed point  $(u,w,a)$ of the map \eqref{map} found in the previous step satisfies the claim in part $i)$ of this theorem. For the uniqueness part of our claim, assume that the triplet $(u,w,a)$ satisfies part $i)$ of Theorem \ref{theorem_nonlinear} and notice that then it is clearly a fixed-point of \eqref{map}. To finish we must check that this triplet is in the set defined by \eqref{condition_set}. Notice that thanks to \eqref{smallness_assumption}, we know that \eqref{smooth_transform_1} and \eqref{smooth_transform_2} hold, and we may use \eqref{prop1_5} and \eqref{modelling_estimate_ivp_w} of Propositions \ref{linear_forcing} and \ref{linear_IVP} respectively to obtain that 
\begin{align}
\label{new_new_new}
M_u + [w]_{2\alpha} + \|u\|_{\alpha} + \|w\|_{\alpha} \lesssim N_0 + N_0^{int}.
\end{align} 
So, since $N_0, N^{int}_0 \ll 1$ and $M = M^*$ for a fixed point, we find that indeed $(u,w,a)$ satisfies \eqref{condition_set}. Furthermore, the a priori bounds contained in \eqref{thm_2_conclusion} follows from \eqref{thm_2_3_1}.

Moving on part $ii)$, assume that we have two triplets $(u_i,w_i,a_i)$ corresponding to two solutions in part $i)$. Each $(u_i,w_i,a_i)$ is a fixed point of its own map \eqref{map} corresponding to $f_i$ and $U_{int,i}$. Since we are dealing with fixed points we have that $M^* = M$ and $\delta M^* = \delta M$. By \eqref{new_new_new} we know that $M^* \ll 1$ when $N_0, N_0^{int}\ll1$, which means that we may apply \eqref{thm_2_3_2} to obtain \eqref{thm_2_conclusion_2}. 

\end{proof}

\section{Construction of the new ``offline'' products}
\label{products_proof}

We now give the proofs of Lemma \ref{lemma:new_reference_products_1} and Corollary \ref{new_family_1}; as well as of Lemma \ref{lemma:new_reference_products_2} and Corollary \ref{new_family_2}.

\vspace{-.2cm}

\subsection{Proofs of Lemma \ref{lemma:new_reference_products_1} and Corollary \ref{new_family_1}: First type of new ``offline'' products}

\label{products_proof_1}

We begin with the proofs of Lemma \ref{lemma:new_reference_products_1} and Corollary \ref{new_family_1}:

\begin{proof}[Lemma \ref{lemma:new_reference_products_1}]  Since $\frac{2-\alpha}{2} < 1$, the bound \eqref{assumption_G_1} ensures that $\partial_1^2 G \in L^1_{\rm{loc}}(\R^d)$ --whereby, for $F \in  L^{\infty}(\mathbb{R}^2)$, the product $F \partial_1^2 G$ is well-defined defined as a regular distribution. In particular, for any test function $\varphi \in C^{\infty}_0(\mathbb{R}^2)$, $F \partial_1^2 G(\varphi) = \langle \varphi, F\partial_1^2 G \rangle$ --where we recall that $\langle \cdot, \cdot \rangle$ denotes the $L^2(\R^2)$ inner-product.  In order to obtain \eqref{assumption_reconstruction_2}, we fix $x \in \mathbb{R}^2$ and use \eqref{defn_convolution_w_dist} and \eqref{assumption_G_1} as
\begin{align}
\label{edit_thesis_1}
\begin{split}
   &| [F, (\cdot)_T] \partial_1^2 G  (x) | = |F (\partial_1^2 G)_T (x) - \langle F\partial_1^2 G, \psi_T(x - \cdot) \rangle|\\
& =  \Big| \int_{\mathbb{R}^2} (F(x) - F(y)) \psi_T(x-y)  \partial_1^2 G(y) \, \textrm{d} y \Big| \\
& \lesssim  C(G) [ F ]_{\alpha} \Big(  \int_{\mathbb{R}} |\psi_T(x-y)| \,  d^{\alpha}(x,y)\, ( |y_2|^{\frac{\alpha - 2 }{2}} +|y_2|^{\frac{2\alpha - 2 }{2}} ) \, \textrm{d} y \Big)\\
 & \lesssim C(G)  \left[ F \right]_{\alpha} (T^{\frac{1}{4}})^{2\alpha -2} \\
 &\quad  \times \Big(\int_{-1}^1   \int_{\mathbb{R}}  |\psi_1(\hat{x}-\hat{y})| \,  d^{\alpha}(\hat{x},\hat{y}) \,  (|\hat{y}_2|^{\frac{\alpha - 2 }{2}}  + |\hat{y}_2|^{\frac{2\alpha - 2 }{2}}  )\, \textrm{d}\hat{y}_1 \,  \textrm{d}\hat{y}_2+ \int_{\mathbb{R}^2}  |\psi_1(\hat{x}-\hat{y})|  d^{\alpha}(\hat{x},\hat{y}) \,  \textrm{d}\hat{y} \Big).
\end{split}
\end{align} 
Here we have used the change of variables \eqref{standard_rescaling} (\textit{i.e.} $(\hat{x}_1, \hat{x}_2) = ( x_1  T^{-\frac{1}{4}}, x_2 T^{-\frac{1}{2}})$) and that $T \leq 1$. To handle the first term on the right-hand side of \eqref{edit_thesis_1} we use that 
\begin{align*}
p( \cdot ) & =   \int_{\mathbb{R}} |\psi_1(x_1, \cdot )| (|x_1|^{\alpha} + | \cdot |^{\frac{\alpha}{2}}) \, \textrm{d}x_1 \in L^{\infty}(\mathbb{R}),
\end{align*}
which follows from $\psi_1$ being a Schwartz function. Using this, we then have that 
\begin{align}
\begin{split}
 & \int_{-1}^1   \int_{\mathbb{R}}  |\psi_1(x-y)|  d^{\alpha}(x,y) ( |y_2|^{\frac{\alpha - 2 }{2}} +  |y_2|^{\frac{2\alpha - 2 }{2}}) \, \textrm{d}y_1 \,  \textrm{d}y_2\\
&\lesssim  \|p\| \int_{-1}^1  ( |y_2|^{\frac{\alpha - 2 }{2}} +  |y_2|^{\frac{2\alpha - 2 }{2}}) \,  \textrm{d}y_2 < \infty.
\end{split}
\end{align}
For the second term on the right-hand side of \eqref{edit_thesis_1} using that $\psi_1$ is a Schwartz function, we obtain the desired \eqref{assumption_reconstruction_2}. 
\end{proof}

We now apply Lemma \ref{lemma:new_reference_products_1} to obtain the first type of new ``offline'' products:

\begin{proof}[Corollary \ref{new_family_1}]  \quad\\

\noindent $i)$ For $i = 0,1,2$, let $G_i = \partial^i_{a_0}\tilde{\vi}(\cdot, a_0)$ in Lemma \ref{lemma:new_reference_products_1}. By Lemma \ref{semigroup_bounds}, each $C(G_i)$ is bounded by $[\uf_{int} -v(\cdot, a_0)]_{\alpha,2}$. Applying \eqref{assumption_reconstruction_2} and Lemma \ref{small_v}, yields \eqref{assumption_reconstruction_2_new_2}.\\

\noindent $ii)$ For $i = 0,1,2$, let $F = \partial_{a_0}^i v(\cdot, a_0)$ in part $i)$.  The result of Lemma \ref{small_v} yields \eqref{assumption_reconstruction_2_new_2.5}. 

\end{proof}

\subsection{Proofs of Lemma \ref{lemma:new_reference_products_2} and Corollary \ref{new_family_2}: Second type of new ``offline'' products}

\label{products_proof_2}

We now prove Lemma \ref{lemma:new_reference_products_2} and Corollary \ref{new_family_2}.

\begin{proof}[Lemma \ref{lemma:new_reference_products_2}] We begin by symbolically applying Leibniz' rule: 
\begin{align}
\label{definition_reference_1}
\begin{split}
 G \partial_1^2 F
``="&  \partial_1^2 (F G)  - 2 \partial_1 F \partial_1G - F \partial_1^2 G\\
``="& - 2 (\partial_1 (F \partial_1 G )  - F\partial_1^2 G   )+ \partial_1^2 (FG)  - F \partial_1^2G\\
`` = " & \partial_1^2 (FG) - 2 \partial_1 ( F \partial_1 G ) + F \partial_1^2G.
\end{split}
\end{align} 
This heuristic calculation motivates the definition 
\begin{align}
\label{definition_Step_1}
 G  \diamond \partial_1^2 F :=  \partial_1^2 ( F G) - 2 \partial_1 (  F \partial_1G ) +  F  \partial_1^2 G.
 \end{align}
Thanks to \eqref{assumption_G_2}, $F  \partial_1^2 G, F \partial_1 G,$ and $FG \in L^1_{\rm{loc}} (\R^2)$ --for $\varphi \in C^{\infty}_0(\R^2)$, we have that 
\begin{align*}
G  \diamond \partial_1^2 F(\varphi) = \langle FG, \partial_1^2 \varphi \rangle  + 2 \langle F \partial_1 G, \partial_1 \varphi \rangle + \langle F \partial_1^2 G, \varphi \rangle. 
\end{align*}
Notice also that the operation $\diamond$ as defined in \eqref{definition_Step_1} is clearly bilinear.

We now check \eqref{reconstruction_assumption_2_2}. Let $x \in \mathbb{R}^2$ and use \eqref{defn_convolution_w_dist} and \eqref{definition_Step_1} to write
\begin{align}
&| [G , ( \cdot )_T  ]  \diamond \partial_1^2 F (x) | \\
& = G(x) ( \partial_1^2 F)_T \hspace{-.05cm} - \hspace{-.05cm} \langle FG,  \partial_1^2 \psi_T(x - \cdot)\rangle\hspace{-.05cm} - \hspace{-.05cm} 2  \langle F \partial_1 G, \partial_1 \psi_T(x - \cdot) \rangle \hspace{-.05cm} - \hspace{-.05cm} \langle F\partial_1^2 G, \psi_T(x - \cdot)\rangle  \\
&=  \Big| \int_{\mathbb{R}^2} (G(x) -G(y)) (F( y) -  F(x)) \partial_1^2 \psi_T(x-y) \textrm{d}y  \nonumber \\
& \hspace{1.5cm} - 2 \int_{\mathbb{R}^2}  (F(y) -  F(x))\partial_1G(y)  \partial_1 \psi_T(x-y) \, \textrm{d}y \nonumber\\
& \hspace{2.5cm}  - \int_{\mathbb{R}^2}  (F(y) -  F(x)) \partial_1^2 G(y) \psi_T(x_1-y) \, \textrm{d}y\label{rhs_ref} \\
& \hspace{3.5cm} + \int_{\mathbb{R}^2} (G(x) - G(y))   F(x) \partial_1^2 \psi_T(x-y)\, \textrm{d}y \nonumber \\
& \hspace{4.5cm}  - 2 \int_{\mathbb{R}^2} F(x) \partial_1 G(y)  \partial_1 \psi_T(x-y) \, \textrm{d}y \nonumber\\
&  \hspace{5.5cm} - \int_{\mathbb{R}^2} F(x) \partial_1^2 G(y) \psi_T(x-y) \, \textrm{d}y\Big|.\nonumber
\end{align}
The terms on the right-hand side of \eqref{rhs_ref} are then treated separately. The first term is easily handled using \eqref{moment_bound} as
\begin{align*}
\begin{split}
&\Big| \int_{\mathbb{R}^2} (G(x) - G(y)) (F(y) -  F(x)) \partial_1^2 \psi_T(x-y)\, \textrm{d}y  \Big|\\
&\leq  \left[ G \right]_{\alpha} \left[ F\right]_{\alpha} \int_{\mathbb{R}^2} |\partial_1^2 \psi_T(x-y)| \, d^{2\alpha}(x,y) \, \textrm{d} y \lesssim [G]_{\alpha}  \left[ F \right]_{\alpha}  (T^{\frac{1}{4}})^{2\alpha -2}.
\end{split}
\end{align*}
For the second term, we additionally use \eqref{assumption_G_2} to write
\begin{align}
 \label{diamond1_prime_3}
\begin{split}
&\Big| \int_{\mathbb{R}^2}  (F(y) -  F(x))\partial_1G(y)  \partial_1 \psi_T(x-y) \, \textrm{d}y \Big| \\ 
& \lesssim  C(G) \left[ F\right]_{\alpha} \int_{\mathbb{R}^2} |x_2 |^{\frac{\alpha -1}{2}} d^{\alpha}(x,y) |\partial_1\psi_T(x - y) | \, \textrm{d} y   \lesssim C(G)  \left[ F\right]_{\alpha}  (T^{\frac{1}{4}})^{2 \alpha -2}.
\end{split}
\end{align}
The third term is treated as
\begin{align}
\label{diamond1_prime_4}
\begin{split}
&\Big| \int_{\mathbb{R}^2}  (F(y) -  F(x))\partial_1^2 G(y)  \psi_T(x-y) \, \textrm{d}y \Big|\\
&\lesssim C(G) [F]_{\alpha}  \int_{\mathbb{R}^2} |x_2 |^{\frac{\alpha -2}{2}} d^{\alpha}(x,y) | \psi_T(x - y)| \, \textrm{d} y \lesssim  C(G) \left[ F \right]_{\alpha} (T^{\frac{1}{4}})^{2\alpha-2}.
\end{split}
\end{align}
We lastly notice that the last three terms of \eqref{rhs_ref} cancel each other. 

\end{proof}



We now apply Lemma \ref{lemma:new_reference_products_2} with Lemmas \ref{semigroup_bounds} and \ref{small_v}  to obtain the second type of new ``offline'' products:

\begin{proof}[Proof of Corollary \ref{new_family_2}]

\noindent \textit{i)} In part \textit{i)} of Lemma \ref{lemma:new_reference_products_2} we set $G_l =\partial_{a_0}^l \tilde{\vi}_i(\cdot, a_0)$ and $F= \partial^k_{a_0^{\prime}}v_j(\cdot, a^{\prime}_0)$ for $l,k  = 0,1,2$. Notice that, for $l = 0,1,2$, we may apply Lemma \ref{semigroup_bounds} to obtain that \eqref{assumption_G_2} is satisfied and each of the corresponding $C(G_l)$ is bounded by $[ U_{int,i}-v_i(\cdot, a_0)]_{\alpha,2}$. Combining \eqref{reconstruction_assumption_2_2} and Lemma \ref{small_v}, we obtain the desired \eqref{assumption_reconstruction_2_new_1}. Of course, here we have also used the relation 
\begin{align}
\partial^k_{a_0^{\prime}} \partial_{a_0}^l ( \tilde{\vi}_i(\cdot, a_0) \diamond \partial_1^2 v_j(\cdot, a^{\prime}_0)) =  \partial_{a_0}^l \tilde{\vi}_i(\cdot, a_0) \diamond \partial_1^2  \partial^k_{a_0^{\prime}} v_j(\cdot, a^{\prime}_0), 
\end{align}
which follows from the definition \eqref{definition_Step_1}.\\

\noindent \textit{ii)} This is an immediate consequences of the triangle inequality, part i), and the assumption \eqref{assumption_offline_products_appendix_1}.\\

\noindent \textit{iii)} Let $k,l =0,1$ and $i,j = 0,1$. We start by showing \eqref{reference_cross_lemma_new}. Notice that by the definitions \eqref{reference_now_in_lemma} and \eqref{definition_Step_1}, we have that 
\newpage
\begin{align*}
 &\partial_{a_0}^l \partial_{a_0^{\prime}}^k ( [  (\tilde{\vi}_0 +v_0)(\cdot, a_0), \, (\, \cdot \,)_T ] \diamond  \partial_1^2 v_j(\cdot, a_0^{\prime}) -  [ (\tilde{\vi}_1 +v_1)(\cdot, a_0) , \, (\, \cdot \,)_T ] \diamond  \partial_1^2 v_j(\cdot, a^{\prime}_0) )\\
 & =   [  \partial_{a_0}^l( \tilde{\vi}_0 - \tilde{\vi}_1) (\cdot, a_0), \, (\, \cdot \,)_T ] \diamond  \partial_1^2 \partial_{a_0^{\prime}}^k v_j(\cdot, a^{\prime}_0)\\
 & \quad +  \partial_{a_0}^l \partial_{a_0^{\prime}}^k (\left[  v_0(\cdot, a_0), \, (\, \cdot \,)_T \right] \diamond  \partial_1^2 v_j(\cdot, a^{\prime}_0) -  \left[ v_1(\cdot, a_0) , \, (\, \cdot \,)_T \right] \diamond  \partial_1^2 v_j(\cdot, a^{\prime}_0) ).
\end{align*}
The relation \eqref{reference_cross_lemma_new} then follows from the triangle inequality, the assumption \eqref{assumption_offline_products_appendix_2}, and the bilinearity of the singular product in Lemma \ref{lemma:new_reference_products_2}. In particular, in Lemma \ref{lemma:new_reference_products_2}, we take $G = \tilde{\vi}_0(\cdot, a_0) -  \tilde{\vi}_1(\cdot, a_0)$, for which Lemma \ref{semigroup_bounds} gives that $C(\tilde{\vi}_0(\cdot, a_0) -  \tilde{\vi}_1(\cdot, a_0)) \lesssim [ U_{int,0} - U_{int,1} + v_0(\cdot, a_0) - v_1(\cdot, a_0) ]_{\alpha,1}$, and $F(\cdot, a_0) = v_j(\cdot, a_0)$, to which we apply Lemma \ref{small_v}. 

Obtaining \eqref{reference_cross_lemma_2_new} is essentially the same. Again, by \eqref{reference_now_in_lemma} and \eqref{definition_Step_1} we can write
\begin{align*}
 & \partial_{a_0}^l \partial_{a_0^{\prime}}^k  ([  (\tilde{\vi}_i +v_i)(\cdot, a_0), \, (\, \cdot \,)_T ] \diamond  \partial_1^2 v_1(\cdot, a^{\prime}_0) -  [ (\tilde{\vi}_i +v_i)(\cdot, a_0) , \, (\, \cdot \,)_T ] \diamond  \partial_1^2 v_0(\cdot, a^{\prime}_0) )\\
& =  [ \partial_{a_0}^l \tilde{\vi}_i (\cdot, a_0) , \, (\, \cdot \,)_T ] \diamond  \partial_1^2 \partial_{a_0^{\prime}}^k (v_1 - v_0)(\cdot, a^{\prime}_0)  \\
 & \quad + \partial_{a_0^{\prime}}^k \partial_{a_0}^l (\left[  v_i(\cdot, a_0), \, (\, \cdot \,)_T \right] \diamond  \partial_1^2  v_1(\cdot, a^{\prime}_0) -  \left[ v_i(\cdot, a_0) , \, (\, \cdot \,)_T \right] \diamond  \partial_1^2 v_0(\cdot, a^{\prime}_0) ).
\end{align*}
The relation \eqref{reference_cross_lemma_2_new} is then obtained via the triangle inequality, the assumption \eqref{assumption_offline_products_appendix_3}, and the bilinearity of the singular product from Lemma \ref{lemma:new_reference_products_2} --along with Lemmas \ref{semigroup_bounds} and \ref{small_v}. \\ 

\noindent $iv)$ Let $i,j=0,1$. Then, this follows from the triangle inequality, Lemma 2, and part $i)$ of Corollary \ref{new_family_1} with $F = \partial_{a_0}^l(\tilde{\vi}_i + v_i)(\cdot, a_0)$ for $l = 0,1,2$. 
\end{proof} 

\subsection{Proof of the reconstruction lemmas}
\label{section:reconstruct_proof}
We give abbreviated proofs for Lemmas \ref{lemma:reconstruct_1} and \ref{lemma:reconstruct_2} as well as Corollary \ref{post_process_reconstruct_1}. Since many of the arguments in \cite{OW} see no change on their passage to our setting, we only address issues that see variations. For more details, see \cite[Lemmas 3.3 and 3.5]{OW}.

\begin{proof}[Lemma \ref{lemma:reconstruct_1}] The proof has four steps:\\

\noindent \textbf{Step 1:} \textit{(Bound for $\nu$)} \quad Using the same argument as in Step 1 of the proof of Lemma \ref{KrylovSafonov} one obtains
\begin{align}
\| \nu \|_{2\alpha -1}  & \lesssim M + N \label{nu_bound_1_old}.
\end{align}


\noindent  \textbf{Step 2:} \textit{(Dyadic decomposition)} \quad For $T \geq \tau >0$ such that $T = 2^n \tau$ for some $n \in \mathbb{N}$, one can show that
\begin{align}
\label{dyadic_decomp_1}
\begin{split}
& (  U h_T -   E_{diag} \left[   w, \left(\cdot\right)_T \right] \diamond h  - \nu \left[x_1, (\cdot)_T \right] h) - (  U h_{\tau} -   E_{diag} \left[   w , \left(\cdot \right)_{\tau} \right] \diamond h - \nu \left[x_1, (\cdot)_{\tau} \right] h)_{T - \tau}\\
& =  \displaystyle\sum_{ \footnotesize{t = \tau 2^i \textrm{ for } 
 0 \leq i \leq n}} \Big(   \left(\left[  U  , \left(\cdot\right)_t \right] -  E_{diag}\left[  w  , \left( \cdot \right)_t \right] - \nu\left[ x_1, (\cdot)_{t} \right] \right)h_t  \\
& \quad \quad \quad \quad \quad \quad \quad \quad \quad \quad \quad \quad   - \left[ \nu, \left(\cdot \right)_t\right] \left[  x_1  , \left( \cdot \right)_t \right] h - \left[E_{diag}, (\cdot)_t \right] \left[w, (\cdot)_t \right] \di h   \Big)_{T-2t}.
\end{split}
\end{align}
This dyadic decomposition follows from the semigroup property \eqref{semigroup}.\\

\noindent \textbf{Step 3:} \textit{(Use of the modelling)} \quad Using the dyadic decomposition from the previous step, one finds that 
\begin{align}
\label{Lemma3_9}
\begin{split}
& \big\|   U h_T -   E_{diag} \left[   w , \left(\cdot\right)_T \right] \diamond h - \nu \left[x_1, (\cdot)_T \right] h \\
& \qquad -  (  U h_{\tau} -   E_{diag} \left[   w , \left(\cdot\right)_{\tau} \right] \diamond h - \nu \left[x_1, (\cdot)_{\tau} \right] h)_{T - \tau} \big\|\\
& \lesssim  (M +  N) N_0 (T^{\frac{1}{4}})^{3 \alpha -2}
\end{split}
\end{align}
for $\tau < T  \leq 1$ such that $T$ is a dyadic multiple of $\tau$. In particular, \eqref{Lemma3_9} is obtained from \eqref{dyadic_decomp_1} after using \eqref{monotonicity} and the three estimates:
\begin{align}
\|  ([  U  , \left(\cdot\right)_t ] -  E_{diag}[ w  , \left( \cdot \right)_t ] - \nu[ x_1, (\cdot)_{t} ] )h_t  \| & \lesssim M N_0 (t^{\frac{1}{4}})^{3 \alpha -2}, \label{three_estimates_1}\\
\| [ \nu, \left(\cdot \right)_t] [  x_1  , \left( \cdot \right)_t ] h \| & \lesssim (M + N) N_0(t^{\frac{1}{4}})^{3 \alpha -2}, \label{three_estimates_2}\\
 \textrm{ and } \| [E_{diag}, (\cdot)_t ] [w, (\cdot)_t ] \di h \| & \lesssim N N_0 (t^{\frac{1}{4}})^{3 \alpha -2}.\label{three_estimates_3}
\end{align}
Notice that in this step, in order to make the geometric series on the right-hand side of \eqref{dyadic_decomp_1} converge, it is necessary that $\alpha \in (\frac{2}{3},1)$. The three estimates are proven using the assumptions  \eqref{Lemma2_0} - \eqref{Lemma2_3}. We remark that the proof of the first estimate requires the use of \cite[Lemma A.2]{OW}, which says that 
\begin{align}
\|[x_1, (\cdot)]h \|_{\alpha -2} \lesssim \| h\|_{\alpha -2}
\end{align}
and the bounds on $\nu$ from Step 1.\\

\noindent  \textbf{Step 4:} \textit{(Conclusion)} \quad To conclude, we introduce the notation $\mathcal{F}^{\tau} = U h_{\tau} - E_{diag}\left[w, (\cdot)_{\tau} \right]\di h - \nu \left[x_1, (\cdot)_{\tau} \right] h$. Now, \eqref{Lemma3_9} becomes
\begin{align}
\label{Lemma3_13.5}
\sup_{T \leq 1} (T^{\frac{1}{4}})^{2 - 3 \alpha} \| \mathcal{F}^T - (\mathcal{F}^{\tau})_{T - \tau} \| \lesssim (M + N ) N_0,
\end{align}
where the supremum is still taken over $T$ that are dyadic multiples of $\tau$. By the assumptions \eqref{Lemma2_2} and \eqref{Lemma2_2.5}, the bound \eqref{nu_bound_1_old}, and \cite[Lemma A.2 ]{OW} we obtain
 \begin{align}
 \label{lemma_2_estimate_edit_1}
 \begin{split}
\sup_{T \leq 1} (T^{\frac{1}{4}})^{2 - \alpha} \|\mathcal{F}^T\|
  = &  \sup_{T \leq 1} (T^{\frac{1}{4}})^{2 - \alpha} \| U h_{T} - E_{diag}\left[w, (\cdot)_{T} \right]\di h - \nu \left[x_1, (\cdot)_{T} \right] h \|\\
 \lesssim & ( \|U \|  + \| \nu\|) \| h_{T} \|_{\alpha-2} + N N_0 \\
 \lesssim & (\|U \|  +  M +N)  N_0 .
\end{split}
\end{align}
Combining this with \eqref{monotonicity} and \eqref{Lemma3_13.5}, the triangle inequality yields that
\begin{align}
\label{Lemma3_14}
\| \mathcal{F}^{\tau} \|_{\alpha-2} \lesssim  (\|U \|  + M +N) N_0.
\end{align}
By Lemma \ref{equivnorm}, we may (up to a subsequence) pass to the limit $\tau \rightarrow 0$ using the statement of Arzel\`{a}-Ascoli. In particular, we define $U \diamond h$ such that $ \mathcal{F}^{\tau} \rightharpoonup U \diamond h$. The bound \eqref{Dec_2019_1} follows from taking the limit $\tau \rightarrow 0$ in \eqref{Lemma3_9} and using the lower semicontinuity of the $L^{\infty}$-norm with respect to weak-$*$ convergence. \hfill  

\end{proof}
\vspace{.2cm}

As we have seen, in order to apply Lemma \ref{lemma:reconstruct_1} in the proof of Theorem \ref{theorem_nonlinear}, we use Corollary \ref{post_process_reconstruct_1}.  The proof of Corollary \ref{post_process_reconstruct_1} is essentially the same as that for \cite[Corollary 3.4]{OW}, but relies on modelling information in terms of $\tilde{\vi} + v$ as opposed to $v_{\textrm{OW}}$, where the subscript is included because of the massive term in \eqref{periodic_mean_free}. The different modelling information, however, does not change the character of the calculations as the equation solved by $\vi(\cdot, a_0)$ is  linear and we have access to Lemma  \ref{semigroup_bounds} and Corollary \ref{new_family_2}. While the proof is straightforward, various choices for the distribution $h$ and the family $\left\{ w(\cdot, x) \right\}_x$ in Lemma \ref{lemma:reconstruct_1} are made, it is computationally intensive. To avoid excessive repetition we, therefore, only give an abbreviated proof below.\\

\begin{proof}[Corollary \ref{post_process_reconstruct_1}] This is a corollary of Lemma \ref{lemma:reconstruct_1} and comes down to choosing appropriate families $\left\{w(\cdot, x) \right\}$, indexed by $x \in \R^2$, and distributions $h$ to which to apply the lemma. Here, we will use the full barrage of Lemmas \ref{equivnorm}, \ref{semigroup_bounds}, and \ref{small_v} and Corollary \ref{new_family_2} without further notice.\\

\noindent \textit{i)} First, for $a_0^{\prime} \in [\lambda,1]$, we set 
\begin{align}
w(\cdot,x) & = \sigma_i(x) (\tilde{\vi}_i + v_i)(\cdot, a_i(x)),\\
h & = \partial_1^2 v_j(\cdot, a^{\prime}_0),\\
\textrm{and} \quad w(\cdot, x) \di h & =  \sigma_i(x) (v_i+ \tilde{\vi}_i)(\cdot, a_i(x)) \diamond \partial_1^2 v_j(\cdot, a^{\prime}_0).
\end{align}
Using similar calculations to those in \cite{OW}, one finds that the assumptions \eqref{Lemma2_0}-\eqref{Lemma2_3} hold for $N = [U_{int,i}]_{\alpha} + [f_i]_{\alpha -2}$ and $N_0 = [f_j]_{\alpha -2}$.  We may then apply Lemma \ref{lemma:reconstruct_1}, which yields a distribution $\uf \diamond \partial_1^2 v_j(\cdot, a^{\prime}_0) \in C^{\alpha -2}(\mathbb{R}^2)$ satisfying \eqref{conclusion_lemma_2_1}, but without the parameter derivatives included in the norm. 

In order to obtain the full bound \eqref{conclusion_lemma_2_1}, we must also control the indicated parameter derivatives. (In this case, we must consider two parameter derivatives.) To do this, for any $a_0^-, a_0^+ \in [\lambda,1]$, we first set 
\begin{align}
\begin{split}
w(\cdot,x)  & = \sigma_i(x) (v_i+ \tilde{\vi}_i) (\cdot, a_i(x)),\\
 h & = \partial_1^2  v_j(\cdot, a_0^+) -  \partial_1^2 v_j(\cdot, a_0^-),\\
 \textrm{and} \quad w(\cdot,x) \diamond h  & =  \sigma_i(x)\big(  (v_i+ \tilde{\vi}_i) (\cdot, a_i(x))\diamond \partial_1^2 v_j(\cdot, a_0^+) -  (v_i+ \tilde{\vi}_i) (\cdot, a_i(x))\diamond \partial_1^2  v_j(\cdot, a_0^-) \big).
\end{split}
\end{align}
Again, the assumptions of Lemma \ref{lemma:reconstruct_1} are checked-- of course, the family of $w(\cdot, x)$ has not changed from the previous scenario. We find that \eqref{Lemma2_0}-\eqref{Lemma2_3} hold for $N_0 = [f_j]_{\alpha -2} |a_0^+ - a_0^-|$ and $N = [U_{int,i}]_{\alpha} + [f_i]_{\alpha -2}$. Lemma \ref{lemma:reconstruct_1} then yields a distribution $\uf \diamond (\partial_1^2 v_j (\cdot, a_0^+) -  \partial_1^2 v_j(\cdot, a_0^-)) \in C^{\alpha -2}(\mathbb{R}^2)$ satisfying
\begin{align}
\label{Cor_1_deriv_1_2}
\| [ \uf , (\cdot)_T]  \diamond (\partial_1^2 v_j(\cdot, a_0^+) -  \partial_1^2 v_j(\cdot, a_0^-)) \|_{2\alpha -2} \lesssim   |a_0^+ - a_0^-| (N_0 + N_0^{int})N_0^{int}.
\end{align}
To finish showing that \eqref{conclusion_lemma_2_1} holds for the norm $\| \cdot \|_{2\alpha -2,1}$, we notice that, due to the built-in linearity of the definition for $w(\cdot,x) \diamond h(\cdot)$ and the uniqueness in Lemma \ref{lemma:reconstruct_1}, the identity
\begin{align*}
U \diamond (\partial_1^2 v_j(\cdot, a_0^+) -  \partial_1^2 v_j(\cdot, a_0^-)) = U \diamond \partial_1^2 v_j(\cdot, a_0^+) -  U \diamond \partial_1^2 v_j(\cdot, a_0^-)
\end{align*}
holds. Plugging this into \eqref{Cor_1_deriv_1_2}, we may deduce \eqref{conclusion_lemma_2_1} for one parameter derivative. 

To obtain \eqref{conclusion_lemma_2_1} for the norm $\| \cdot \|_{2\alpha -2,2}$, we set 
\begin{align}
w(\cdot,x) & = \sigma_i(x)(v_i+ \tilde{\vi}_i)(\cdot, a_i(x)),\\
h & = (\partial_1^2 v_j(\cdot, a_0^{++}) - \partial_1^2 v_j(\cdot, a_0^{+-}) ) -(\partial_1^2 v_j(\cdot, a_0^{-+}) -  \partial_1^2 v_j(\cdot, a_0^{--})),\\
 \textrm{and} \quad w(\cdot, x) \diamond h  & = \sigma_i(x) \Big(  (v_i+ \tilde{\vi}_i) (\cdot, a_i(x)) \diamond \partial_1^2 v_j(\cdot, a_0^{++})  \\
 & \hspace{2cm} - (v_i+ \tilde{\vi}_i) (\cdot, a_i(x))  \diamond \partial_1^2 v_j(\cdot, a_0^{+-}) \\
 & \hspace{2.5cm} -  \big( (v_i+ \tilde{\vi}_i) (\cdot, a_i(x))  \diamond \partial_1^2 v_j(\cdot, a_0^{-+}) \\
 & \hspace{3cm} - (v_i+ \tilde{\vi}_i) (\cdot, a_i(x)) \diamond \partial_1^2 v_j(\cdot, a_0^{--}) \big) \Big)
  \end{align}
for any $a_0^{++}, a_0^{+-}, a_0^{-+}, a_0^{--} \in [\lambda, 1]$ such that $|a_0^{++}- a_0^{+-}| =  |a_0^{-+} - a_0^{--}| $. To finish, we again check the assumptions of Lemma \ref{lemma:reconstruct_1}; we conclude that \eqref{Lemma2_0}-\eqref{Lemma2_3}  are satisfied for $N_0 = [f_j]_{\alpha -2}|a_0^{++} - a_0^{+-}| \, |a_0^{-+} - a_0^{--}| $ and $N =  [U_{int,i}]_{\alpha} + [f_i]_{\alpha -2}$. One then completes the argument as for one parameter derivative above. 

To obtain \eqref{Lemma_3_3_1_1}, one first sets 
\begin{align}
w(\cdot,x) & = \sigma_i(x) (v_i + \tilde{\vi}_i)(\cdot, a_i(x)),\\
 h & = (\partial_1^2 v_1 -  \partial_1^2 v_0)(\cdot, a^{\prime}_0),\\
 \textrm{and} \quad w(\cdot,x) \diamond h  & =  \sigma_i(x) \Big((v_i + \tilde{\vi}_i)(\cdot, a_i(x)) \diamond \partial_1^2 v_1(\cdot, a^{\prime}_0)- (v_i+ \tilde{\vi}_i)(\cdot, a_i(x))  \diamond \partial_1^2 v_0(\cdot, a^{\prime}_0) \Big),
 \end{align}
for $a_0^{\prime} \in [\lambda, 1]$. We notice that \eqref{Lemma2_0}-\eqref{Lemma2_3} hold for $N_0 = [f_1 - f_0]_{\alpha}$ and $N = [U_{int,i}]_{\alpha} + [f_i]_{\alpha -2}$, which yields the relation \eqref{Lemma_3_3_1_1}, but without the control of the indicated parameter derivatives. Notice that we have used the uniqueness claim of Lemma \ref{lemma:reconstruct_1} in order to make the identification
\begin{center}
$\uf \diamond (\partial_1^2 v_1- \partial_1^2 v_0)(\cdot, a^{\prime}_0)  = \uf \diamond \partial_1^2 v_1  (\cdot, a^{\prime}_0)   -  \uf \diamond \partial_1^2 v_0 (\cdot, a^{\prime}_0). $
\end{center}

Following the recipe that we have already introduced above, in order to upgrade our previous result to the full \eqref{Lemma_3_3_1_1}, we, for $a_0^+, a_0^- \in [\lambda, 1]$, set 
\begin{align}
w(\cdot,x)  & = \sigma_i(x) (v_i + \tilde{\vi}_i)(\cdot, a_i(x)),\\
h & = (\partial_1^2 v_1 -  \partial_1^2 v_0)(\cdot, a_0^+) - (\partial_1^2 v_1 -  \partial_1^2 v_0)(\cdot, a_0^-),\\
\textrm{and} \quad 
 w(\cdot,x) \diamond h  & =  \sigma_i(x) \Big(  (v_i + \tilde{\vi}_i)(\cdot, a_i(x)) \diamond \partial_1^2 v_1(\cdot, a_0^+)\\
 & \hspace{2cm} - (v_i+ \tilde{\vi}_i)(\cdot, a_i(x))  \diamond \partial_1^2 v_0(\cdot, a_0^+)  \\
& \hspace{2.5cm}-  \big((v_i + \tilde{\vi}_i)(\cdot, a_i(x)) \diamond \partial_1^2 v_1(\cdot, a_0^-) \\
& \hspace{3cm} - (v_i+ \tilde{\vi}_i)(\cdot, a_i(x))  \diamond \partial_1^2 v_0(\cdot, a_0^-)\big) \Big).
 \end{align}
Noticing that \eqref{Lemma2_0}-\eqref{Lemma2_3} hold for $N_0 = [f_1 - f_0]_{\alpha -2 } |a_0^+ - a_0^-|$ and $N = [U_{int,i}]_{\alpha} + [f_i]_{\alpha -2}$, one finishes the argument as already indicated above.\\

\noindent \textit{ii)} For $a^{\prime}_0 \in [\lambda, 1]$, we now set 
\begin{align}
w(\cdot,x) & = \sigma_1(x) (v_1 + \tilde{\vi}_1)(\cdot, a_1(x)) - \sigma_0(x)(v_0 + \tilde{\vi}_0) (\cdot, a_0(x)),\\
 h & =\partial_1^2 v_j(\cdot, a^{\prime}_0),\\
 \textrm{ and } \quad  w(\cdot, x) \diamond  h  & =  \sigma_1(x) (v_1 + \tilde{\vi}_1)(\cdot, a_0(x))\diamond \partial_1^2 v_j(\cdot, a^{\prime}_0)  -  \sigma_0(x) (v_0 + \tilde{\vi}_0) (\cdot, a_0(x))\diamond \partial_1^2 v_j(\cdot, a^{\prime}_0).
 \end{align}
Again, one checks \eqref{Lemma2_0}-\eqref{Lemma2_3}, which are seen to hold for $N_0 = [f_j]_{\alpha -2}$ and $N = (\|a_1 - a_0 \|_{\alpha} + \| \sigma_1 - \sigma_0 \|_{\alpha} )(\max_{i=0,1}[f_i]_{\alpha -2} + \max_{i=0,1} [U_{int,i}]_{\alpha}) + [f_1 - f_0]_{\alpha -2} + [U_{int,1} - U_{int,0}]_{\alpha}$ --applying Lemma \ref{lemma:reconstruct_1} and additionally using the uniqueness to make the identification 
\begin{align}
(U_1 - U_0) \diamond \partial_1^2 v_j(\cdot, a^{\prime}_0) = U_1 \diamond \partial_1^2 v_j(\cdot, a^{\prime}_0) - U_0 \diamond \partial_1^2 v_j(\cdot, a^{\prime}_0),
\end{align}
yields  \eqref{conclusion_lemma_2_ii} without a parameter derivative. To obtain the statement for the parameter derivative, one takes $h = \partial_1^2 v_i (\cdot, a_0^+) - \partial_1^2 v_i (\cdot, a_0^-)$ and proceeds as in the previous part. \hfill  
\end{proof}

To finish this section we give the argument for the second reconstruction lemma, which we only summarize and for more details point the reader to the proof of \cite[Lemma 3.5]{OW}.\\

\begin{proof}[Lemma \ref{lemma:reconstruct_2}]  The proof has three steps:\\

\noindent \textbf{Step 1:} \textit{(Dyadic decomposition)} \quad Just as in Step 2 of the proof of Lemma \ref{lemma:reconstruct_1}, we have the dyadic decomposition:
\begin{align*}
\begin{split}
&(  F  \partial_1^2 \uf_T -  \sigma_i  E \left[   F  , \left(\cdot\right)_T \right] \diamond \partial_1^2 w_i) - (  F \partial_1^2 \uf_{\tau} - \sigma_i E \left[  F  , \left(\cdot\right)_{\tau} \right] \diamond \partial_1^2 w_i)_{T - \tau} \\
& = \displaystyle\sum_{ \footnotesize{t = \tau 2^i \textrm{ for } 
 0 \leq i \leq n}
} \Big( \left[  F  , \left(\cdot\right)_t \right] \partial_1^2 \uf_t -  \sigma_iE\left[  F  , \left( \cdot \right)_t \right] \partial_1^2  w_{it} \\
&  \hspace{3.2cm}   -\sigma_i [ E, \left(\cdot \right)_t] \left[  F  , \left( \cdot \right)_t \right] \diamond \partial_1^2  w_i   - [\sigma_i, (\cdot)_t] E[F, (\cdot)_t] \diamond \partial_1^2  w_i  \Big)_{T-2t}, 
\end{split}
\end{align*}
for $T, \tau >0$ such that $T = 2^n \tau$ for some $n \in \mathbb{N}$. Again, the proof of this identity only relies on the semigroup property \eqref{semigroup}.\\

\noindent \textbf{Step 2:} \textit{(Use of the modelling)} \quad We upgrade Step 1 to the following estimate:
\begin{align}
\label{combine_step3}
\begin{split}
&\|  F   \partial_1^2 \uf_T - \sigma_i E \left[   F  , \left(\cdot\right)_T \right] \diamond \partial_1^2 w_i - (  F   \partial_1^2 \uf_{\tau} - \sigma_i E \left[   F  , \left(\cdot\right)_{\tau} \right] \diamond \partial_1^2 w_i)_{T - \tau})\|\\
& \lesssim \left( \left[ F \right]_{\alpha} M + \| \sigma_i \|_{\alpha}(1 + [a]_{\alpha})N N_i  \right)(T^{\frac{1}{4}})^{3 \alpha -2},\\
\end{split}
\end{align}
which holds for $T = 2^n \tau$ for $\tau>0$ and $n \in \mathbb{N}$. The argument for \eqref{combine_step3} relies on the following three relations:
\begin{align}
\| \left[ F, (\cdot)_t \right] \partial_1^2 \uf_t - \sigma_i E\left[ F, (\cdot )_t \right] \partial_1^2 w_{it} \| & \lesssim \left[ F \right]_{\alpha} M(t^{\frac{1}{4}})^{3 \alpha -2}\label{lemma4_2_1},\\
 \|\sigma_i \left[E , (\cdot)_t \right] \left[ F, (\cdot)_t \right] \di \partial_1^2 w_i  \| & \lesssim \| \sigma_i \| \left[ a \right]_{\alpha} N N_i (t^{\frac{1}{4}})^{3 \alpha -2}, \label{lemma4_2_2}\\
\textrm{ and } \quad  \| [\sigma_i, (\cdot)_t] E[F, (\cdot)_t] \diamond \partial_1^2  w_i \| & \lesssim \left[ \sigma_i \right]_{\alpha} N N_i (t^{\frac{1}{4}})^{3 \alpha -2} \label{lemma4_2_3_new},
\end{align}
which hold for any $t>0$, and that $\alpha \in (\frac{2}{3},1)$.\\

\noindent \textbf{Step 3:} \textit{(Conclusion)} \quad To conclude we use the notation $\mathcal{F}^{\tau} =   F   \partial_1^2 \uf_{\tau}  - \sigma_i E \left[  F  , (\, \cdot \, )_{\tau} \right] \diamond \partial_1^2 w_i$. Notice that 
\begin{align*}
 \| \mathcal{F}^{T} \|  \leq & \| F \| \| \partial_1^2 \uf_{T} \| + \| \sigma_i \| \|  \left[  F  , (\, \cdot \, )_{\tau} \right] \diamond \partial_1^2 w_i \| \lesssim (\| F\|   [\uf]^{loc}_{\alpha} + \| \sigma_i\| N N_i) (T^{\frac{1}{4}})^{\alpha -2},
\end{align*}
where we have bounded $\| \partial_1^2 \uf_T \|$ as in \eqref{Lemma5_14.1} of Lemma \ref{KrylovSafonov} and have used the assumption \eqref{Lemma4_2}. Combining this with \eqref{combine_step3} and Step 2 of Lemma \ref{KrylovSafonov} we obtain
\begin{align}
\label{compactness_lemma_4}
\| \mathcal{F}^{\tau}\|_{\alpha -2} \lesssim   \| F \|_{\alpha} (M  + \| \sigma_i \| N_i)+ N N_i (1 + [a]_{\alpha}) \| \sigma_i\|_{\alpha}.
\end{align}
As in the proof of Lemma \ref{lemma:reconstruct_2}, we can then use compactness in $C^{\alpha -2}(\mathbb{R}^2)$ along with Lemma \ref{equivnorm} in order to define $F \diamond \partial_1^2 U$ as the weak limit (along a subsequence) of the $\left\{ \mathcal{F}^{\tau} \right\}_{\tau}$ as $\tau \rightarrow 0$.
\hfill  
\end{proof}

\section{Proof of Lemma \ref{equivnorm}}
\label{section:equivnorm}

In this section we give an argument for Lemma \ref{equivnorm} that is motivated by \cite[Lemma 5]{IO}.  

\begin{proof}

Observe that it suffices to show
\begin{align}
\label{equiv_step2_1}
\displaystyle \left[ f \right]_{\alpha -2} \lesssim  \| f \|_{\alpha -2}.
\end{align}
Indeed, for the converse inequality, we decompose $f = \partial_1^2 f^1 + \partial_2 f^2 + f^3$ in a way that is near optimal in the sense of Definition \ref{negative_norm}. For such a triple $(f^1, f^2, f^3)$ the use of \eqref{alpha_kernel_bound} then yields for $T\le 1$
\begin{align*}
\| f_T \|   & =  \| (\partial_1^2 f^1 + \partial_2 f^2 + f^3)_T \|\\
& \lesssim (T^{\frac{1}{4}})^{\alpha-2}([f^1]_{\alpha} + [f^2]_\alpha + \|f^3\| ) \lesssim (T^{\frac{1}{4}})^{\alpha-2} [f]_{\alpha-2},
\end{align*}
as desired. Hence, we may concentrate on \eqref{equiv_step2_1}.

As a technical tool, we make use of the convolution kernel $e^{-T} \psi_T$ that is associated to the semigroup of the operator $\mathcal{A}:= \partial_1^4 - \partial_2^2 + 1$. We use the notational convention that $f \ast e^{-T} \psi_T = f_T^m$ and, as always, $f \ast \psi_T = f_T$. \\

\noindent \textbf{Step 1:} \textit{(Bound for the $C^{\alpha}$- seminorm)} \quad We first show that, for $\alpha \in  (0,1)$, it holds that
\begin{align}
\label{equiv_step1}
\left[ f \right]_{\alpha} \lesssim  \displaystyle\sup_{T \leq 1} (T^{\frac{1}{4}})^{-\alpha} \| T \mathcal{A} f_T \|.
\end{align}
 For this, we first notice that due to homogeneity, we may assume that 
\begin{align*}
\displaystyle\sup_{T \leq 1}(T^{\frac{1}{4}})^{- \alpha } \| T \mathcal{A} f_T \| = 1.
\end{align*}
Notice that, due to the semigroup property and \eqref{monotonicity}, the above normalization ensures that
\begin{align*}
 \| T \mathcal{A} f^m_T \| = T e^{-T} \| (\mathcal{A} f_{1})_{T-1} \| \lesssim T e^{-T} \|\mathcal{A} f_1 \|  \lesssim 1,
 \end{align*}
for $T > 1$. Combining the two estimates, we find that 
\begin{align}
\label{equiv_step1_0.1}
\displaystyle\sup_{T >0}(T^{\frac{1}{4}})^{- \alpha } \| T \mathcal{A} f^m_T \| \lesssim 1.
\end{align}
Together with the semigroup property of $e^{-T}\psi_T$ and the moment bound \eqref{moment_bound}, this yields, for $j,l \geq 0$ and $T>0$, that 
\begin{align}
\label{equiv_step1_1}
\|\partial_1^j \partial_2^l \mathcal{A} f_T^m \| & = e^{-\frac{T}{2}}\|\partial_1^j \partial_2^l \mathcal{A} (f^m_{\frac{T}{2}})_{\frac{T}{2}} \|\\
&  \lesssim e^{-\frac{T}{2}} (T^{\frac{1}{4}})^{-j - 2 l } \| \mathcal{A} f^m_{\frac{T}{2}} \| \lesssim e^{-\frac{T}{2}} (T^{\frac{1}{4}})^{-j - 2 l + \alpha -4}.
\end{align}

By definition $e^{-T} \psi_T$ is a smooth solution of $(\partial_T + \mathcal{A})e^{-T} \psi_T = 0$ and the moment bounds for $\psi_T$, furthermore, imply that $f^m_T$ is a smooth solution of $(\partial_T + \mathcal{A}) f^m_T = 0$. Fixing $j,l \geq 0$ and using \eqref{equiv_step1_1} allows us to write
\begin{align*}
 \begin{split}
\| \partial_1^j \partial_2^l (f^m_t - f^m_T) \|& = \Big\| \int^T_t \partial_1^j \partial_2^l \mathcal{A} f^m_s  \, \textrm{d}s \Big\| \\
& \lesssim  \int^{T}_t  e^{-\frac{s}{2}}  (s^{\frac{1}{4}})^{-j - 2l +\alpha -4} \, \textrm{d} s  \lesssim (T^{\frac{1}{4}})^{-j - 2l + \alpha} + (t^{\frac{1}{4}})^{-j - 2l + \alpha}
 \end{split}
\end{align*}
for all $0 < t <T$. In the case that $j = l = 0$ this yields that
\begin{align}
\label{equiv_step1_3}
\| f^m_t - f^m_T \| \lesssim  (T^{\frac{1}{4}})^{\alpha},
\end{align}
which implies that \eqref{equiv_step1_3} holds also for $t=0$.

Fixing a point $x\in\R^2$, $0 < t <T$  and $j, l \geq 0$ such that $j + l \geq1$, we then use the triangle inequality to write 
\begin{align*}
 |\partial_1^j \partial_2^l f^m_t(x) | & \leq | \partial_1^j \partial_2^l ( f^m_t - f^m_T)(x) |+ | \partial_1^j \partial_2^l f^m_T (x)|\\
& \lesssim (  (t^{\frac{1}{4}})^{-j - 2l + \alpha}  +  (T^{\frac{1}{4}})^{-j - 2l + \alpha} ) + e^{-T}|\partial_1^j \partial_2^l f_T(x)|,
\end{align*}
which after letting $T \rightarrow \infty$ gives
\begin{align}
\label{equiv_step1_5.1}
\| \partial_1^j \partial_2^l f^m_t \|\lesssim   (t^{\frac{1}{4}})^{-j - 2l + \alpha}. 
\end{align}
 
To finish the argument for \eqref{equiv_step1}, we fix $T>0$ and two distinct points $x,y \in \R^2$. We then write
\begin{align*}
|f^m_T(y) - f^m_T(x)| \leq \|\partial_1 f^m_T \|d(y,x) + \|\partial_2 f^m_T \| d^2(y,x),
\end{align*}
which we combine with \eqref{equiv_step1_3} for $t = 0$ and \eqref{equiv_step1_5.1} to obtain
\begin{align}
\label{equiv_step1_7}
\begin{split}
|f(y) - f(x) | &\lesssim \| f- f^m_T \|+ \| \partial_1 f^m_T \|d(y,x) + \| \partial_2 f^m_T\| d^2(y,x)  \\
&\lesssim  (T^{\frac{1}{4}})^{\alpha} + (T^{\frac{1}{4}})^{\alpha-1} d(y,x) + (T^{\frac{1}{4}})^{\alpha -2} d^2(y,x).
\end{split}
\end{align}
This we may further process by setting $T^{\frac{1}{4}} = d(y,x)$, which yields $|f(y) - f(x)| \lesssim d^{\alpha}(y,x)$.\\

\noindent \textbf{Step 2:} \textit{(A specific decomposition of $f$)} \quad  Assume that $\|f \|_{\alpha -2} =1$. Using this and the properties \eqref{semigroup} and \eqref{monotonicity}, we obtain the relation
\begin{align}
\label{equiv_step5_1.7}
\| f^m_T\| = e^{-T} \| (f_{1})_{T-1} \| \lesssim e^{-T} \| f_1\|\lesssim e^{-T}
\end{align}
for $T>1$. In this step we find that these observations are enough to show that 
\begin{align}
\label{definition_solution_u}
u = \int_0^{\infty} f^m_T \, \textrm{d}T
\end{align}
is a distributional solution of
\begin{align}
\label{equation_for_u}
\mathcal{A}(u) & = f \qquad \textrm{in} \quad \mathbb{R}^2.
\end{align}

We first show that, for any $t\in(0,1)$, the function 
\begin{align}
\label{ut_defn}
u^t :=  \int_0^{\infty}  f^m_{t+T} \, \textrm{d}T
\end{align}
satisfies $\mathcal{A} u^t = f^m_t$. To see this, we recall from Step 1 that $f^m_{t+T}$ solves $(\partial_T + \mathcal{A})f^m_{t+T} = 0$ on $\mathbb{R}^2$, which allows us to write 
\begin{align}
\label{combine_1_1}
\int_{0}^{\infty} \partial_T f^m_{t+T} \, \textrm{d}T = -  \int_0^{\infty}  \mathcal{A} f^m_{t+T} \, \textrm{d}T.
 \end{align}
Using that $t>0$, we process the left-hand side of \eqref{combine_1_1} as
\begin{align}
\label{combine_1_3}
\int_{0}^{\infty} \partial_T f^m_{t+T} dT  = - f_t^m,
\end{align}
where we have  used that $\| f^m_{T}\|\to 0$ as $T\to\infty$ by \eqref{equiv_step5_1.7}. For the term on the right-hand side of \eqref{combine_1_1}, we use that $\| f\|_{\alpha -2}=1$ and \eqref{equiv_step5_1.7} to obtain
\begin{align}
\label{uniform_bound_norms}
\begin{split}
 \int_{0}^{\infty} | \partial_1^j \partial_2^l f^m_{t+T}| \, \textrm{d}T 
& \lesssim (t^{\frac{1}{4}})^{-j - 2l} \int_{0}^{\infty} \| f^m_T \| \textrm{d}T\\ & \lesssim (t^{\frac{1}{4}})^{-j - 2l} \Big(\int_{0}^{1}(T^{\frac{1}{4}})^{\alpha -2} \, \textrm{d}T + \int_1^{\infty} e^{-T} \, \textrm{d}T \Big)< \infty,
\end{split}
\end{align}
which means that
\begin{align}
\label{combine_1_2}
 \int_0^{\infty}  \mathcal{A} f^m_{t+T} \, \textrm{d}T =   \mathcal{A}\Big(  \int_0^{\infty}  f^m_{t+T} \, \textrm{d}T\Big).
\end{align}
In particular, combining \eqref{combine_1_1} and \eqref{combine_1_2} we end up with 
\begin{align*}
\mathcal{A} \Big(  \int_0^{\infty}  f^m_{t+T} \, \textrm{d}T\Big) = f^m_t.
\end{align*}
To show that $u^t \rightarrow u$ uniformly as $t\rightarrow 0$, we can directly estimate the difference as 
\begin{align}
\| u^t - u \| = \Big\| \int_0^t f_T^m dT \Big\| \lesssim t^{\frac{\alpha +2}{4}},
\end{align}
where we have again used that $\| f \|_{\alpha -2} =1$.\\

\noindent \textbf{Step 3:} \textit{(Argument for \eqref{equiv_step2_1})} \quad By homogeneity, we may assume that $\| f \|_{\alpha -2} =1$. Using the decomposition 
\begin{align*}
f = \mathcal{A}(u) = \partial_1^2 (\partial_1^2 u) + \partial_2 (-\partial_2 u) + u
\end{align*}
with $u$ given as \eqref{definition_solution_u}, we can apply Definition \ref{negative_norm} to find that 
\begin{align}
\label{equiv_step5_3}
\left[ f \right]_{\alpha -2} \leq [ \partial_1^2 u ]_{\alpha} + [ \partial_2 u ]_{\alpha} + \left[ u \right]_{\alpha} + \|u\|.
\end{align}
Noticing that since $\| f \|_{\alpha -2} =1$ we have that 
\begin{align*}
\displaystyle\sup_{T\leq 1} (T^{\frac{1}{4}})^{2 - \alpha} \| (\mathcal{A} u)_T \| = \displaystyle\sup_{T\leq 1} (T^{\frac{1}{4}})^{2 - \alpha} \| f_T \|  \leq 1,
\end{align*}
which we process with \eqref{moment_bound} to, for $j, l \geq 0$, obtain
\begin{align}
\label{equiv_step5_0}
\displaystyle\sup_{T\leq 1} (T^{\frac{1}{4}})^{j + 2l + (2 - \alpha) -4} \| T (\mathcal{A} \partial_1^j \partial_2^l u)^m_T \| \lesssim 1.
\end{align}
We estimate the first three terms on the right-hand side of  \eqref{equiv_step5_3} by first applying \eqref{equiv_step1} from Step 1 and then using \eqref{equiv_step5_0}. We find that 
\begin{align*}
 & [ \partial_1^2 u ]_{\alpha}  + \left[ \partial_2 u \right]_{\alpha} + \left[ u \right]_{\alpha} \lesssim \displaystyle\sup_{T\leq 1}  (T^{\frac{1}{4}})^{- \alpha} (\|T ( \mathcal{A} \partial_1^2 u)_T \| + \|T ( \mathcal{A} \partial_2 u)_T \| + \|T ( \mathcal{A} u)_T \| )\lesssim 1.
\end{align*}
The bound $\|u \| \lesssim 1$ follows from \eqref{uniform_bound_norms} with $j,l=0$. \hfill  

\end{proof}

\section*{Acknowledgments}

Both authors would like to thank Felix Otto for his guidance throughout this project, which was completed during the first author's PhD work at the MPI in Leipzig. We would also like to thank Scott A. Smith for lots of discussions. Furthermore, we would like to thank the reviewers of this manuscript, whose comments and suggestions were very helpful for revisions.

\bigskip

\bibliographystyle{plain}  
\bibliography{Thesis.bib}

\appendix

\section{Proofs of Lemmas \ref{semigroup_bounds} and \ref{small_v}: Bounds for frozen-coefficient linear solutions}

\label{semigroup_bounds_appendix}

The bounds in Lemma \ref{semigroup_bounds} follow from \eqref{heat_kernel_representation}:
 
\begin{proof} We will use the change of variables $z = \frac{x_1 - y}{(4 x_2 a_0)^{\frac{1}{2}}}$ for which 
\begin{align}
\label{change_relations}
\frac{\partial z}{ \partial y}  = \frac{-1}{(4 x_2 a_0)^{\frac{1}{2}}}  \quad \textrm{and}  \quad  \frac{\partial z}{\partial a_0}  =  -\frac{1}{2} z a_0^{-1}.
\end{align} 
For $k \in \mathbb{N}$ we use the convention that $P_k$ represents a generic degree $k$ polynomial; additionally, $P_k(\cdot, a_0^{-\frac{1}{2}})$ indicates a polynomial of order $k$ with coefficients that are polynomials in $a_0^{-\frac{1}{2}}$. \\

\noindent \textit{i)} Fix $1 \leq k \leq 2$ and $ j \geq 0$ and let $0 \leq m \leq j$. Using the above change of variables, we obtain
\begin{align}
\label{heat_kernel_derivs}
\begin{split}
 \partial_{a_0}^m \partial_1^k G(x_1 - y, x_2, a_0)
 &=  \partial_{a_0}^m \partial_1^k \Big(  \frac{e^{-x_2}}{( 4 \pi a_0 x_2)^{\frac{1}{2}}} e^{-z^2} \Big)\\
& =   e^{-x_2} \partial_{a_0}^m ( (a_0 x_2)^{-\frac{1 +k}{ 2} } P_k(z) e^{-z^2} )\\
& =  e^{-x_2} P_{k + 2m}(z, a^{-\frac{1}{2}}_0)  e^{-z^2} x_2^{-\frac{1 +k}{2}} .
\end{split}
\end{align}
We then notice that
\begin{align}
\label{smuggle_in_point}
\begin{split}
&\partial_{a_0}^j  \int_{\mathbb{R}} \vi_{int}(y, a_0) \partial_1^k G(x_1 - y, x_2, a_0) \, \textrm{d} y \\
& = \int_{\mathbb{R}} \displaystyle\sum_{m = 0}^j  {j \choose m} (\partial_{a_0}^m\vi_{int}(y, a_0) - \partial_{a_0}^m \vi_{int}(x_1, a_0))\partial_{a_0}^{j-m} \partial_1^k G( x_1 - y,  x_2, a_0) \, \textrm{d} y,
\end{split}
\end{align}
where we have used that $k \geq 1$. To finish showing \eqref{heat_kernel_bound_1}, we use \eqref{heat_kernel_derivs} to calculate
\begin{align*}
\begin{split}
& \Big|  \int_{\mathbb{R}} \displaystyle\sum_{m = 0}^j {j\choose m} (\partial_{a_0}^m \vi_{int}(y, a_0) - \partial_{a_0}^m \vi_{int}(x_1,a_0))\partial_{a_0}^{j-m} \partial_1^k G(x_1 - y,  x_2, a_0) \, \textrm{d} y\Big|\\
& \lesssim    [\vi_{int}(\cdot, a_0)]_{\alpha,j}  x_2^{\frac{\alpha}{2}} \int_{\mathbb{R}} |z|^{\alpha}  \displaystyle\sum_{m = 0}^j  {j \choose m}  | \partial_{a_0}^{j-m} \partial_1^k G(x_1 - y , x_2, a_0)| \, \textrm{d} y\\
& \lesssim  [\vi_{int}(\cdot, a_0)]_{\alpha,j} e^{-x_2}  x_2^{\frac{\alpha-k}{2}}   \int_{\mathbb{R}} |z|^{\alpha} \displaystyle\sum_{m = 0}^j {j\choose m} P_{k + 2m}(z, a_0^{-\frac{1}{2}})  e^{-z^2}  \, \textrm{d} z \\
& \lesssim [\vi_{int}(\cdot, a_0)]_{\alpha,j}  e^{-x_2} x_2^{\frac{\alpha - k}{2}}.
\end{split}
\end{align*}

When the initial condition does not depend on $a_0$, then we also obtain \eqref{heat_kernel_bound_1} for $k =0$. This is clear once we make the observation that, since $\partial_{a_0} \int _{\mathbb{R}} G(x_1-y, x_2, a_0) \, \textrm{d}y = 0$,
the relation \eqref{smuggle_in_point} still holds.\\

\noindent \textit{ii)} Fix $j\geq0$. The relation  \eqref{heat_kernel_bound_2} then easily follows from \eqref{heat_kernel_representation} and \eqref{heat_kernel_derivs}. In particular, for $x \in \mathbb{R}^2_+$, we can write
 \begin{align*}
 \begin{split}
  &  \left| \partial_{a_0}^j \vi(x_1, x_2, a_0, \vi_{int}(a_0)) \right|\\
 & \lesssim  e^{-x_2}  \int_{\mathbb{R}} \displaystyle\sum_{m = 0}^j {j \choose m} |\partial_{a_0}^m \vi_{int}(y, a_0)|x_2^{-\frac{1}{2}}   |P_{2(j-m)}(z, a_0^{-\frac{1}{2}})| e^{-z^2}   \, \textrm{d} y \lesssim   e^{-x_2}  \| \vi_{int}\|_j. 
 \end{split}
\end{align*}

\noindent \textit{iii)} Differentiating \eqref{constant_coeff_ivp_body} in terms of $a_0$ gives:
\begin{align}
(\partial_2 - a_0 \partial_1^2 +1) \partial_{a_0} \vi(\cdot, a_0,  \vi_{int} (a_0)) & = \partial_1^2 \vi(\cdot, a_0,  \vi_{int} ( a_0)) && \textrm{in} \quad \mathbb{R}^2_+, \label{equation_one_deriv_V} \\
\partial_{a_0} \vi (\cdot, a_0,  \vi_{int} ( a_0)) & = \partial_{a_0} \vi_{int} (\cdot, a_0) && \textrm{on} 
\quad \partial \mathbb{R}^2_+. \nonumber
\end{align}
Taking one more derivative in $a_0$, we find that $\partial_{a_0}^2 \vi(\cdot, a_0,  \vi_{int} (\cdot, a_0))$ solves 
\begin{align}
(\partial_2 - a_0 \partial_1^2 +1) \partial^2_{a_0} \vi(\cdot, a_0,  \vi_{int} (a_0)) & = 2\partial_1^2  \partial_{a_0} \vi(\cdot, a_0, \vi_{int} (a_0)) && \textrm{in} \quad \mathbb{R}^2_+, \label{equation_two_deriv_V} \\
\partial_{a_0}^2 \vi (\cdot, a_0, \vi_{int} (a_0)) & = \partial_{a_0}^2 \vi_{int} (\cdot, a_0)&&  \textrm{on}\quad \partial \mathbb{R}^2_+. \nonumber
\end{align}
Differentiating a third time gives that $\partial_{a_0}^3 \vi(\cdot, a_0, \vi_{int} (a_0))$ solves 
\begin{align}
\label{equation_three_deriv_V} 
\begin{split}
(\partial_2 - a_0 \partial_1^2 +1) \partial^3_{a_0} \vi(\cdot, a_0, \vi_{int} ( a_0)) & = 3 \partial_1^2  \partial_{a_0}^2 \vi(\cdot, a_0, \vi_{int} (a_0)) \quad \, \, \, \textrm{in} \quad \mathbb{R}^2_+, \\
\partial_{a_0}^3 \vi (\cdot, a_0, \vi_{int} ( a_0)) & = \partial_{a_0}^3 \vi_{int} (\cdot, a_0) \qquad \qquad \qquad  \textrm{on}\quad \partial \mathbb{R}^2_+.
\end{split}
\end{align}
From these equations we can read-off \eqref{heat_kernel_bound_3.5} by using the Schauder estimate $[u]_{\alpha} \lesssim [f]_{\alpha -2} + \|g\|_{\alpha}$ for $u\in C^{\alpha}(\R^2_+)$ solving
\begin{align*}
(\partial_2 - a_0 \partial_1^2 +1) u & = f && \textrm{in} \quad \mathbb{R}^2_+,\\
u & = g &&  \textrm{on}\quad \partial \mathbb{R}^2_+. \nonumber
\end{align*}
This estimate follows from decomposing $f = \partial_2 f^2 + \partial_1^2 f^1 + f^3$ for a triplet $(f^1, f^2, f^3)$ of $C^{\alpha}$-functions that is near optimal in the sense of Definition \ref{negative_norm} and applying the classical Schauder estimate \cite[Lemma 9.2.1]{Kr} to the solutions of 
\begin{align*}
(\partial_2 - a_0 \partial_1^2 +1) u_i & = \partial_{i}^{3-i} f^i && \textrm{in} \quad \mathbb{R}^2_+,\\
u_i & = 0 &&  \textrm{on}\quad \partial \mathbb{R}^2_+. \nonumber
\end{align*}
and 
\begin{align*}
(\partial_2 - a_0 \partial_1^2 +1) u_{\partial} & = 0&& \textrm{in} \quad \mathbb{R}^2_+,\\
u_{\partial} & = g &&  \textrm{on}\quad \partial \mathbb{R}^2_+. 
\end{align*}
In particular, we can then use the linearity of the equation and the uniqueness of the solution $u$ to obtain the desired Schauder estimate. 

Using the Schauder estimate we find that
\begin{align*}
\begin{split}
[\partial_{a_0} \vi(\cdot, a_0,  \vi_{int} (a_0)) ]_{\alpha} & \lesssim [\partial_1^2 \vi(\cdot, a_0,  \vi_{int} (a_0))]_{\alpha-2} +  \|\vi_{int}(\cdot, a_0)\|_{\alpha,1} \lesssim \|\vi_{int}(\cdot, a_0)\|_{\alpha,1}, \\
 [\partial_{a_0}^2 \vi(\cdot, a_0,  \vi_{int} ( a_0)) ]_{\alpha} & \lesssim [\partial_1^2 \partial_{a_0} \vi(\cdot, a_0,  \vi_{int} (a_0))]_{\alpha-2} + \|\vi_{int}(\cdot, a_0)\|_{\alpha,2}\\
 & \lesssim [\partial_{a_0} \vi(\cdot, a_0,  \vi_{int} (a_0))]_{\alpha} +  \|\vi_{int}(\cdot, a_0)\|_{\alpha,2}
  \lesssim \|\vi_{int}(\cdot, a_0)\|_{\alpha,2}, \\
 \textrm{and } \,  [\partial_{a_0}^3 \vi(\cdot, a_0,\vi_{int} ( a_0)) ]_{\alpha} & \lesssim [\partial_1^2 \partial^2_{a_0} \vi(\cdot, a_0, \vi_{int} ( a_0))]_{\alpha-2} + \| \vi_{int}(\cdot, a_0)\|_{\alpha,3} \\
 & \lesssim [\partial_{a_0}^2 \vi(\cdot, a_0, \vi_{int}( a_0))]_{\alpha} +   \| \vi_{int}(\cdot, a_0)\|_{\alpha,3} \lesssim  \| \vi_{int}(\cdot, a_0)\|_{\alpha,3}.\\
\end{split}
\end{align*}

\noindent \textit{iv)} Notice that \eqref{heat_kernel_bound_4} follows directly from \eqref{heat_kernel_bound_3.5} if either $y_2^{\frac{1}{2}} < d(x,y)$ or $x_2^{\frac{1}{2}} < \frac{1}{2}d(x,y)$. 

We consider the case that $y_2^{\frac12} \geq d(x,y)$ and $x_2^{\frac12} \geq \frac12 d(x,y)$. Fix two points $x,y \in \mathbb{R}^2_+$ and $0 \leq j \leq 1$. We first apply the triangle inequality as
\begin{align}
\begin{split}
& |\partial^j_{a_0} \vi(x_1,x_2, a_0, \vi_{int}( a_0)) - \partial_{a_0}^j \vi(y_1, y_2, a_0,  \vi_{int}( a_0))|\\
&\leq  |\partial_{a_0}^j \vi(x_1, x_2, a_0,  \vi_{int}(a_0)) -\partial_{a_0}^j \vi(x_1, y_2, a_0,  \vi_{int}(a_0))| \\
 & \quad + |\partial_{a_0}^j \vi(x_1, y_2, a_0,  \vi_{int}( a_0)) - \partial_{a_0}^j \vi(y_1, y_2, a_0,  \vi_{int}( a_0))|  \label{hkb_4_1}
\end{split}
\end{align}

The second term of \eqref{hkb_4_1} is estimated by $\| \vi_{int}\|_{\alpha,j}  (y_2^{-\frac{\alpha}{2} } +x_2^{-\frac{\alpha}{2} }  ) d^{2\alpha}(x,y)$. For this, we first use \eqref{heat_kernel_bound_1}, \eqref{heat_kernel_bound_3.5}, and Taylor's theorem:
\begin{align*}
\begin{split}
& |\partial_{a_0}^j \vi(x_1, y_2, a_0, \vi_{int}(a_0)) -\partial_{a_0}^j \vi(y_1,  y_2, a_0,  \vi_{int}( a_0))|\\
 &=  |\partial_{a_0}^j \vi(x_1, y_2, a_0, \vi_{int}( a_0)) - \partial_{a_0}^j \vi(y_1, y_2, a_0, \vi_{int}(a_0))|^{\frac{2 - 2\alpha}{2-\alpha}} \\
 & \quad  \times |\partial_{a_0}^j \vi(x_1, y_2, a_0, \vi_{int}(a_0)) - \partial_{a_0}^j \vi(y_1, y_2, a_0,  \vi_{int}(a_0))|^{\frac{\alpha}{2 - \alpha}}  \\
 & \lesssim   \Big(\| \vi_{int}(a_0) \|_{\alpha,j} d^{\alpha}(x,y)\Big)^{\frac{2 - 2\alpha}{2-\alpha}} \Big(d^2(x,y) [\vi_{int}]_{\alpha} y_2^{\frac{\alpha-2}{2}} + d(x,y) [\vi_{int}]_{\alpha} y_2^{\frac{\alpha -1}{2}} \Big)^{\frac{\alpha}{2 - \alpha}} .\\
 \end{split}
\end{align*}
The problem term is now the first-order piece of the Taylor expansion, to handle this term we consider the cases $y_2 \geq x_2$ and $x_2 > y_2$. Starting with the first case, notice that since $y_2^{\frac12} \geq d(x,y)$ and $x_2^{\frac12} \geq \frac12 d(x,y)$ and $y_2 \geq x_2$, we have that $y_2^{\frac12} - x_2^{\frac12} \geq \frac12 d(x,y)$. Furthermore using that the Lipschitz constant of $\sqrt{\cdot}$ on $[x_2, \infty)$ is $x_2^{-1/2}$, we may bound
\begin{align}
 d(x,y) y_2^{\frac{\alpha -1}{2}} \lesssim (y_2^{\frac12} - x_2^{\frac12}) y_2^{\frac{\alpha -1}{2}} \lesssim |y_2 - x_2|  y_2^{\frac{\alpha -1}{2}} x_2^{-\frac12} \lesssim  |y_2 - x_2|  \Big(y_2^{\frac{\alpha -2}{2}} + x_2^{\frac{\alpha -2}{2}}\Big),
\end{align}
where the last bound is an application of Young's inequality. The case that $x_2 > y_2$ is the same, expect that it does not require an application of Young's inequality at the end.

The first term of \eqref{hkb_4_1} is treated using the equations \eqref{constant_coeff_ivp_body} and \eqref{equation_one_deriv_V}. In particular, after applying \eqref{heat_kernel_bound_1} and \eqref{heat_kernel_bound_2} we have that
\begin{align}
\label{hkb_4_3}
\begin{split}
\| \partial_2 \vi(\cdot, x_2, a_0, \vi_{int}(a_0)) \|  \leq & \| \partial_1^2 \vi(\cdot, x_2, a_0, \vi_{int}(a_0)) \| + \|\vi(\cdot, x_2, a_0, \vi_{int}(a_0)) \| \quad\\
 \lesssim &   \|\vi_{int}(\cdot, a_0) \|_{\alpha} ( x_2^{ \frac{\alpha -2}{2}} + e^{-x_2})
 \lesssim  \|\vi_{int} (\cdot, a_0) \|_{\alpha} x_2^{ \frac{\alpha -2}{2}} \quad
\end{split}
\end{align}
and, similarly, 
\begin{align}
\label{hkb_4_4}
\begin{split}
& \| \partial_2 \partial_{a_0} \vi(\cdot, x_2, a_0, \vi_{int}(a_0)) \|\\
 &\leq   \| \partial_1^2 \vi(\cdot, x_2, a_0,\vi_{int}(a_0)) \| +  \| \partial_1^2 \partial_{a_0} \vi(\cdot, x_2, a_0,\vi_{int}(a_0)) \| \\
 & \quad + \| \partial_{a_0} \vi(\cdot, x_2, a_0,\vi_{int}(a_0))\|\\
 & \lesssim \|\vi_{int}(\cdot, a_0)\|_{\alpha,1} ( x_2^{ \frac{\alpha -2}{2}} +  e^{-x_2})
 \lesssim   \|\vi_{int}(\cdot, a_0)\|_{\alpha,1}  x_2^{ \frac{\alpha -2}{2}}.
 \end{split}
\end{align}
Using \eqref{hkb_4_3} (when $j = 0$) or \eqref{hkb_4_4} (when $j = 1$), we then obtain
\begin{align}
\label{hkb_4_5}
\begin{split}
& |\partial_{a_0}^j\vi(x_1, x_2, a_0, \vi_{int}(a_0)) - \partial_{a_0}^j\vi(x_1,  y_2, a_0, \vi_{int}(a_0))|\\
 &\leq  |\partial_{a_0}^j\vi(x_1, x_2, a_0, \vi_{int}(a_0)) - \partial_{a_0}^j\vi(x_1,  y_2, a_0, \vi_{int}(a_0))|^{\frac{2-2\alpha}{2-\alpha}} \\
 & \quad \quad \quad \times |\partial_{a_0}^j\vi(x_1, x_2, a_0, \vi_{int}(a_0)) - \partial_{a_0}^j\vi(x_1,  y_2, a_0, \vi_{int}(a_0))|^{\frac{\alpha}{2-\alpha}}\\
  & \leq  (|x_2 - y_2|^{\frac{\alpha}{2}} [\vi_{int}(a_0)]_{\alpha,j})^{\frac{2-2\alpha}{2-\alpha}}  ( |x_2 - y_2| \| \partial_2 \vi(\cdot, x_2, a_0) \|_j)^{ \frac{\alpha}{2-\alpha}} \\
& \lesssim \|\vi_{int}(\cdot, a_0)\|_{\alpha,1} x_2^{-\frac{\alpha}{2}} d^{2\alpha} (x,y).
 \end{split}
\end{align}

\vspace{.2cm}

\noindent \textit{v)} Our claim immediately follows from the above arguments using \eqref{heat_kernel_representation}, but with an extra factor of $e^{-x_2}$ in the definition of the heat kernel \eqref{heat_kernel}.

\end{proof}

The argument for Lemma \ref{small_v} depends on classical Schauder theory and Definition \ref{negative_norm}: 

\begin{proof} Let $f = \partial_1^2 f^1 + \partial_2 f^2 + f^3$ be a near optimal decomposition of $f$ in the sense of Definition \ref{negative_norm}. Furthermore, let $v^i(\cdot, a_0)$ be the $C^{\alpha}$- solution of
\begin{align*}
(\partial_2 -  a_0 \partial_1^2+1)v^i(\cdot, a_0) & = f^i && \textrm{in} \quad \mathbb{R}^2
\end{align*}
for $i = 1,2,3$. Notice that, for $i = 1,2$, we may assume that $f^i$ has vanishing average.  By classical Schauder theory we have that  $\|v^i(\cdot, a_0 )\|_{\alpha +2} \lesssim \|f^i \|_{\alpha}$ for each $i = 1,2,3$. Using this and the convention \eqref{definition_beta_12.2} we obtain that 
\begin{align*}
\displaystyle\sum_{i=1}^3 ( [ \partial_1^2 v^{i}(\cdot, a_0) ]_{\alpha} + [ \partial_2 v^{i}(\cdot, a_0) ]_{\alpha} + \|v^i(\cdot, a_0) \| ) \lesssim  \sum_{i = 1}^3 \|f^i\|_{\alpha} \lesssim [ f ]_{\alpha -2},
\end{align*}
where the last bound follows from the vanishing average condition for $f^{i}$ when $i = 1,2$. To conclude our argument, we notice that by the uniqueness of $C^{\alpha}$-solutions to \eqref{periodic_mean_free} we know that $v(\cdot, a_0) = \partial_1^2 v^1(\cdot, a_0) + \partial_2 v^2(\cdot, a_0) +v^3(\cdot, a_0)$. 

For the bounds on the higher-order parameter derivatives, we emulate the argument from part \textit{iii)} of Lemma \ref{semigroup_bounds}. In particular, differentiating \eqref{periodic_mean_free} in terms of $a_0$ gives that
\begin{align}
\label{diff_once_v}
(\partial_2  - a_0 \partial_1^2 +1) \partial_{a_0 }v(\cdot, a_0) & = \partial_1^2 v(\cdot, a_0) && \textrm{in} \quad \mathbb{R}^2,
\end{align}
which by the above gives that $\|\partial_{a_0} v(\cdot, a_0) \|_{\alpha} \lesssim [\partial_1^2 v(\cdot, a_0) ]_{\alpha -2} \lesssim [f]_{\alpha -2}$. Differentiating in terms of $a_0$ again we find that $\partial_{a_0}^2v(\cdot, a_0)$ solves 
\begin{align}
\label{diff_twice_v}
(\partial_2  - a_0 \partial_1^2 +1) \partial_{a_0}^2v(\cdot, a_0) & = 2 \partial_1^2 \partial_{a_0} v(\cdot, a_0)  && \textrm{in} \quad \mathbb{R}^2,
\end{align}
which again yields that $\|  \partial_{a_0}^2 v(\cdot, a_0)\|_{\alpha} \lesssim  [  \partial_1^2 \partial_{a_0} v(\cdot, a_0) ]_{\alpha -2} \lesssim [\partial_{a_0} v(\cdot, a_0) ]_{\alpha} \lesssim [f]_{\alpha -2}$.
\end{proof}

\section{Proofs of auxiliary lemmas for Proposition \ref{linear_IVP}}
\label{technical_lemmas_2}

We start with the argument for Lemma \ref{lemma:norm}:

\begin{proof} We start by showing \eqref{equiv_seminorm_set_0}. For this we fix $x \in \mathbb{R}^2$ and use the growth condition \eqref{growth_condition_ivp_1} and the standard rescaling \eqref{standard_rescaling} to write
\begin{align}
\label{corollary_3_calc_1}
 \begin{split}
& |\partial_1^j\partial_2^l f_T (x) | \leq  C_f  (T^{\frac{1}{4}})^{-j-2l+2\alpha -2} \int_{\mathbb{R}^2} |\hat{x}_2 - \hat{y}_2|^{\frac{2\alpha -2}{2}} |\partial_1^j\partial_2^l\psi_1(\hat{y})|\, \textrm{d} \hat{y} \lesssim  C_f (T^{\frac{1}{4}})^{-j-2l + 2\alpha -2}. \quad 
 \end{split}
\end{align}
For the last inequality we have relied on $\psi_1$ being a Schwartz function and that $-1< \frac{2\alpha- 2}{2}  <0$. 

For \eqref{equiv_seminorm_set_3}, we use \eqref{equiv_seminorm_set_0} to obtain
\begin{align*}
\begin{split}
 |f_T(x)-f_T(y)|&\le \|\partial_1 f_T\| d(y,x) + \|\partial_2 f_T\| d^2(y,x)\\
 & \lesssim C_f ((T^{\frac{1}{4}})^{2\alpha - 3} d(y,x) + (T^{\frac{1}{4}})^{2\alpha - 4} d^2(y,x) )\\
 &\lesssim C_f \begin{cases}
       (T^{\frac{1}{4}})^{\alpha - 2} d^\alpha(y,x), & \text{if } d(y,x)\le T^{\frac14}, \\
       (T^{\frac{1}{4}})^{2\alpha - 4} d^2(y,x), &  \text{if } d(y,x)> T^{\frac14}.
      \end{cases}
      \end{split}
\end{align*}
The estimate for $d(y,x)\le T^{\frac14}$ is already in the desired form. The estimate for $d(y,x)> T^{\frac14}$ can be interpolated with
\begin{align*}
 |f_T(x)-f_T(y)| \le 2\|f_T\|\lesssim C_f (T^{\frac{1}{4}})^{2\alpha - 2},
\end{align*}
in order to yield
\begin{align*}
 |f_T(x)-f_T(y)|=|f_T(x)-f_T(y)|^{1-\frac{\alpha}{2}}|f_T(x)-f_T(y)|^{\frac{\alpha}{2}}\lesssim C_f (T^{\frac{1}{4}})^{\alpha - 2} d^\alpha(y,x).
\end{align*}
This proves \eqref{equiv_seminorm_set_3}.

If $f$ is additionally only supported for positive times, then for any $x\in\R^2_L$ we may write
\begin{align}
\begin{split}
 |\partial_1^j\partial_2^l f_T (x) | &\leq  \int_{\mathbb{R}^2} |f(x - y)| |\partial_1^j\partial_2^l\psi_T(y)|\, \textrm{d} y\\
 & \leq L^{-\delta} \int_{\mathbb{R}^2} |f(x - y)| |y_2|^\delta |\partial_1^j\partial_2^l\psi_T(y)|\, \textrm{d} y \\
 &\lesssim C_f  L^{-\delta}(T^{\frac{1}{4}})^{-j-2l+2\alpha -2 + 2\delta} \int_{\mathbb{R}^2} |\hat{x}_2 - \hat{y}_2|^{\frac{2\alpha -2}{2}} |\hat{y}_2|^\delta|\partial_1^j\partial_2^l\psi_1(\hat{y})|\, \textrm{d} \hat{y} \\
 &\lesssim  C_f  L^{-\delta}(T^{\frac{1}{4}})^{-j-2l+2\alpha -2 + 2\delta},
 \end{split}
\end{align}
which is \eqref{j002}. Here, as above, we have relied on $-1< \frac{2\alpha- 2}{2}  <0$. Now \eqref{equiv_seminorm_set_2} follows from \eqref{j002}.

\end{proof}

Here is the proof of Lemma \ref{semigroup_bound_modelling}:

\begin{proof}

This proof is essentially a corollary of the argument for part $i)$ of Lemma \ref{semigroup_bounds}. In particular, we use the heat kernel representation \eqref{heat_kernel_representation}, the modelling of $u$, that the heat kernel $G( x_1, x_2, a_0)$ given in \eqref{heat_kernel} is even in $x_1$, and the identity \eqref{heat_kernel_derivs} to write
\begin{align}
\begin{split}
& | E_{tr} \partial_1^2 \vi(x, a_0, u- v(a_0)) |\\
&=   \Big| \int_{\mathbb{R}} ( u(y,0) - v(y,0, a_{tr}(x)) ) \partial_1^2 G (x_1 - y, x_2 , a_{tr}(x)) \, \textrm{d} y\Big|\\
&\leq    \Big | \int_{\mathbb{R}} ( u(y,0) - u(x_1,0) - ( v(y, 0, a_{tr}(x)) - v(x_1,0, a_{tr}(x))  ) - \nu(x_1) (y - x_1) )  \\
& \hspace{6.5cm} \times \partial_1^2 G ( x_1 - y, x_2, a_{tr}(x) )\textrm{d} y\Big|\\
&\lesssim  M_{\partial} \int_{\mathbb{R}} |y - x_1|^{2\alpha} | \partial_1^2 G (x_1 - y, x_2, a_{tr}(x) )| \,\textrm{d} y\\
&\lesssim  M_{\partial} |x_2|^{\frac{2\alpha -2}{2}}.
\end{split}
\end{align}

\end{proof}

\end{document}